\documentclass[12pt]{amsart} % final version after revision
\usepackage[usenames]{color}
\usepackage{tikz-cd,a4wide,enumitem,amsmath,amsfonts,amssymb,amsthm,graphics,amsrefs}
\usepackage{hyperref,comment,microtype}

\makeindex

\newcommand{\univ}{\mathcal{U}}
\newcommand{\gggg}{\mathfrak{g}_\infty}
\newcommand{\mnfld}{\mathcal{M}}
\newcommand{\Z}{\mathbb{Z}}
\newcommand{\R}{\mathbb{R}}
\newcommand{\lctvs}{LCTVS}
\newcommand{\HHH}{\mathbb{H}}

\newcommand{\SL}{\operatorname{SL}}
\newcommand{\Ad}{\operatorname{Ad}}
\newcommand{\Egam}{\mathbf{\Gamma}}
\newcommand{\chars}{X}
\newcommand{\charsA}{\tilde\chars}
\newcommand{\dom}{\operatorname{ld}}
\newcommand{\m}{\operatorname{m}}
\newcommand{\quas}[1]{{#1}^{\sharp}}
\newcommand{\idem}{\mathbf{e}}
\newcommand{\spc}{\mathfrak{E}}
\newcommand{\spcc}{\mathfrak{F}}
\newcommand{\spccc}{\mathfrak{G}}
\newcommand{\bnch}{\mathfrak{B}}
\newcommand{\bnchh}{\mathfrak{C}}
\newcommand{\un}{\operatorname{un}}
\newcommand{\cnst}{C}
\newcommand{\totcusp}{\mathcal{T}^{\cusp}}
\newcommand{\csp}[1]{#1_c}
\newcommand{\Lin}{\mathfrak{L}}
\newcommand{\Lins}{\Lin_{\mathbf{s}}}
\newcommand{\Linb}{\Lin_{\mathbf{b}}}
\newcommand{\lead}{\mathcal{L}}
\newcommand{\proj}{\mathfrak{p}}
\newcommand{\unq}{\operatorname{unq}}
\newcommand{\der}{\operatorname{der}}
\newcommand{\AAA}{A}
\newcommand{\smth}{\operatorname{sm}}
\newcommand{\stds}{\mathcal{P}}
\newcommand{\rest}{\lvert}
\newcommand{\aut}{\operatorname{aut}}
\newcommand{\funct}{\mathfrak{H}}          % Hilbert spcae of L^2 functions
\newcommand{\functb}{\mathfrak{F}}         % sup norm
\newcommand{\umd}{\functb_{\operatorname{umg}}}
\newcommand{\umdf}{C^\infty}
\newcommand{\rapid}{C_c^\infty}%{\functb_{\operatorname{rpd}}}
\newcommand{\Weyl}{W}
\newcommand{\Ker}{\operatorname{Ker}}
\newcommand{\cusp}{\operatorname{cusp}}
\newcommand{\Siegel}{\mathcal{S}}
\newcommand{\autspace}{\mathcal{X}}
\newcommand{\zspace}{\mathcal{Z}}
\newcommand{\aux}{\operatorname{aux}}
\newcommand{\fnl}{\operatorname{fnl}}
\newcommand{\supp}{\operatorname{supp}}
\newcommand{\modulus}{\delta}
\newcommand{\srts}{\Delta}
\newcommand{\CT}{\operatorname{pr}}
\newcommand{\C}{\mathbb{C}}
\newcommand{\Q}{\mathbb{Q}}
\newcommand{\at}{\operatorname{main}}
\newcommand{\diff}{\mathbf{D}}
\newcommand{\Id}{\operatorname{Id}}
\newcommand{\ball}{\mathcal{B}}
\newcommand{\K}{\mathbf{K}}
\newcommand{\A}{\mathbb{A}}
\newcommand{\z}{\mathfrak{z}}
\newcommand{\PE}{\mathfrak{P}}
\newcommand{\aaa}{\mathfrak{a}}
\newcommand{\Ht}{H}
\newcommand{\AF}{\mathcal{A}}
\newcommand{\iii}{{\rm i}}
\newcommand{\bs}{\backslash}
\newcommand{\sprod}[2]{\left\langle#1,#2\right\rangle}
\newcommand{\abs}[1]{\left|{#1}\right|}
\newcommand{\norm}[1]{\lVert#1\rVert}
\newcommand{\dsum}{\oplus}
\newcommand{\Exp}{\mathcal{E}}
\newcommand{\Img}{\operatorname{Im}}
\newcommand{\y}{\operatorname{y}}
\renewcommand{\Re}{\operatorname{Re}}
\newcommand{\Sol}{\operatorname{Sol}}
\newcommand{\Hom}{\operatorname{Hom}}
\newcommand{\vol}{\operatorname{vol}}

\newcommand{\one}{\mathbf{1}}
\newcommand{\loc}{\operatorname{loc}}
\newcommand{\rpd}{\operatorname{rpd}}

\newtheorem*{theorem}{Theorem}
\newtheorem*{lemma}{Lemma}

\newtheorem*{proposition}{Proposition}
\newtheorem*{corollary}{Corollary}
\newtheorem*{example}{Example}
\newtheorem*{definition}{Definition}
\newtheorem*{remark}{Remark}
\newtheorem{claim}{Claim}

\numberwithin{equation}{section}

\begin{document}

\title[Eisenstein series]{On the meromorphic continuation of Eisenstein series}
\author{Joseph Bernstein}
\address{School of Mathematical Sciences, Tel Aviv University, Tel Aviv 6997801, Israel}
\email{bernstei@tauex.tau.ac.il}
\thanks{J.B. is partially supported by ERC grant 291612 and ISF grant ``Integrals of Automorphic functions''}
\author{Erez Lapid}
\address{Department of Mathematics, Weizmann Institute of Science, Rehovot 7610001, Israel}
\email{erez.m.lapid@gmail.com}
\date{\today}

\begin{abstract}
Eisenstein series are ubiquitous in the theory of automorphic forms.
The traditional proofs of the meromorphic continuation of Eisenstein series, due to Selberg and Langlands,
start with cuspidal Eisenstein series as a special case, and deduce the general case from spectral theory.

We present a ``soft'' proof which relies only on rudimentary Fredholm theory (needed only in the number field case).
It is valid for Eisenstein series induced from an arbitrary automorphic form.

The proof relies on the principle of meromorphic continuation. It is close in spirit to Selberg's later
proofs.
\end{abstract}

\maketitle

\setcounter{tocdepth}{1}
\tableofcontents

\section{Introduction}

The theory of Eisenstein series is fundamental to the spectral theory of automorphic forms.
%on locally symmetric spaces of negative curvature of finite volume.
It was initiated by Maass and Roelcke and developed in earnest by Selberg in the 1950s, mostly in the rank one case \cite{MR0088511}.
The general case was undertaken by Langlands in the 1960s \cites{MR0579181, MR1361168}.

Let $G$ be a semisimple Lie group and $\Gamma$ a lattice in $G$.
The goal of Selberg and Langlands was to write down the spectral decomposition of $L^2(\Gamma\bs G)$
in terms of Eisenstein series induced from the discrete spectrum of automorphic quotients of Levi subgroups.
% type $\Gamma_M\bs M^{\der}$ where $M$ is a Levi subgroup of $G$ and
%$\Gamma_M$ is a lattice of the semisimple part $M^{\der}$ of $M$.
The Eisenstein series are automorphic forms on $\Gamma\bs G$ that are defined as Poincar\'e series,
i.e., as sums over certain quotients of $\Gamma$.
They are built from automorphic forms on automorphic quotients of Levi subgroups $M$,
and depend on a spectral parameter $s$ in an $r$-dimensional complex space %$r$ complex variables
where $r$ is the corank of $M$ in $G$.
The series converges absolutely when $\Re s$ is sufficiently regular in the positive Weyl chamber.
However, for the $L^2$ decomposition one needs to consider the Eisenstein series for $\Re s=0$.

The first order of business is to meromorphically continue the Eisenstein series to the whole complex $r$-space.
Langlands carried this out in two steps. The first was to prove it in the special case where the inducing data is cuspidal.
This was done by extending Selberg's methods. The second step, which is much more complicated and is one of Langlands's
greatest achievements, is to describe the discrete spectrum in terms of residues of Eisenstein series induced
from cusp forms. This is delicate because, among other things, it can happen
that more than $r$ singular hyperplanes intersect at a point and the multidimensional residue calculus becomes
highly nontrivial.

Subsequently, it was proved by Franke that every automorphic form is a linear combination
of Laurent coefficients of Eisenstein series induced from cusp forms \cite{MR1603257}*{Corollary 1 on p. 236}.
%-- see \cite{MR1361168}*{Appendix II}.
This provides meromorphic continuation of Eisenstein series induced
from any automorphic form.

\nocite{MR0176097,MR0249539,MR993313,CoSar,MR0081967} \nocite{MR0244260}
%(\cite{MR0579181}
%see also \cite{CoSar}, \cite{MR1361168}*{Ch. 4} and \cite{MR1008186} which are based on a subsequent proof by Selberg).
%In his celebrated work, Langlands described the discrete part of $L^2(G(F)\bs G(\A)^1)$ in terms of
%residues of Eisenstein series for $\varphi\in\AF_P^{\cusp}$ and used it to extend the theorem to the case
%where $\varphi\in\AF_P$ is square-integrable on $M(F)\bs M(\A)^1$.
%(These Eisenstein series furnish the continuous spectrum of $L^2(G(F)\bs G(\A)^1)$.)
%In fact, the setup of Selberg and Langlands is not confined to arithmetic lattices.
%We refer to \cite{MR1361168} for a complete account of Langlands's theory in the adelic setup, including covering groups,
%and additional references.
%For a general $\varphi\in\AF_P$, the theorem follows from the fact that every automorphic form
%is linear combination of Laurent coefficients of cuspidal Eisenstein series -- see \cite{MR1603257}*{Corollary 1 on p. 236} %and \cite{MR1361168}*{Appendix II}.

The goal of this paper is to give a ``soft'' uniform proof of the meromorphic continuation of Eisenstein series
induced from a general automorphic form (not necessarily cuspidal, or in the discrete spectrum).
The proof evinces that the meromorphic continuation is an ``easy'' part of the theory of Eisenstein series.
We do not appeal to either Langlands's description of the discrete spectrum or Franke's theorem,
or in fact to \emph{any} spectral theory beyond rudimentary Fredholm theory. % (in the number field case).

We work in the adelic setting, which is pertaining to lattices arising as congruence subgroups
of reductive groups over number fields.
However, just like in Selberg's and Langlands's proofs, the method works for non-arithmetic lattices as well.
%theorem above uniformly for all $\varphi\in\AF_P$

Moreover, the proof, together with ideas of Delorme and the second author,
considerably simplifies Langlands's proof of the spectral decomposition of $L^2(\Gamma\bs G)$ -- see \cite{2006.12893}.

The idea of the proof was initially conceived already in the 1980s by the first author.
A key ingredient is a general principle of meromorphic continuation (PMC).
% (Theorem \ref{thm: anal} below, which is proved in the appendix).
The PMC splits the proof into two rather separate statements about automorphic forms, which are of independent interest.

The first %(Theorem \ref{thm: lead})
is the fact that any automorphic form is determined by
its ``leading cuspidal components'' -- namely the terms corresponding to the unnormalized cuspidal exponents
whose real part is in the closure of the positive Weyl chamber.
This is an extension of a basic result of Langlands.
The second %(Theorem \ref{thm: mainfin})
is a locally uniform finiteness result for automorphic forms. %(in the number field case).
This is a technical strengthening,
proved along the same lines, of the well known result of Harish-Chandra on the finite dimensionality of the space
of automorphic forms with a given $\K$ and $\z$-type \cite{MR0232893}.

%(\cite{MR0232893}*{Theorem 1}).

The theory of Eisenstein series is also applicable to reductive groups over function fields.
%(which give rise to discrete analogues of locally symmetric spaces).
Our proof applies equally well to this case, and in fact it is much easier.
It relies on an algebraic version of the PMC.

In \S\ref{sec: prelim} we will give the precise statement of the meromorphic continuation of Eisenstein series
whose proof is our main goal. We will also provide notation that will be used throughout.
The PMC, which is the point of departure for the proof
will be stated in \S\ref{sec: prmc}, where we also recall some generalities about analytic functions
valued in topological vector spaces.

The PMC is pertaining to the solutions of a system of linear
equations that depends analytically on a parameter. It requires uniqueness (which is essential)
and local finiteness (which is more technical).

Basic Fredholm theory provides a key tool for proving local finiteness.
The system of equations $\Xi(s)$ for which we will apply the PMC to prove the main theorem
will be outlined in broad strokes in \S\ref{sec: sketch}. It will be written down precisely later on in
\S\ref{sec: conclusion}.
The case of the classical Eisenstein series for $\SL_2(\Z)$ will be presented in detail in
\S\ref{sec: SL2} (see also \cite{MR1482800}).
This case already illustrates the lion's share of the ideas for the general case.
The heart of the paper is \S\ref{sec: unique} where we prove that an automorphic form is determined
by its ``leading cuspidal components''. In particular, this yields a characterization of Eisenstein series
sufficiently deep in the range of absolute convergence, in terms of their inducing data.
In turn, this characterization guarantees the required uniqueness for the system $\Xi(s)$ in this region.

Eventually, in view of the PMC, this means that we can \emph{define} Eisenstein series as the solutions of $\Xi(s)$.

The local finiteness of (a subsystem of the system of equations) $\Xi(s)$ in the number field case will be proved in \S\ref{sec: LF}.
As mentioned before, this is closely related to the Harish-Chandra finiteness result and the proof is based on the same ideas.
However, the proof given here provides a more precise result and arguably, simplifies and clarifies the roles of the various ingredients.
In fact, this proof shows that a system of conditions that describe automorphic forms with
fixed $\K$ and $\z$ types is locally finite.

The culmination of the proof of the main result in the number field case is achieved in \S\ref{sec: conclusion}.
Finally, we discuss the function field case (which is considerably simpler) in
\S\ref{sec: conclusion2}.

%(Theorem \ref{thm: anal} below, which is proved in the appendix).

\begin{comment}
The deduction of the theorem will be explained in \S\ref{sec: conclusion} in the number field case
and in \S\ref{sec: conclusion2} in the function field case (which is much easier and relies on an algebraic
version of the principle of meromorphic continuation). We give a brief summary in \S\ref{sec: sketch}
and work out the example of the classical Eisenstein series for $\SL_2(\Z)$ in \S\ref{sec: SL2} (see also \cite{MR1482800}).
\end{comment}
We point out that the uniqueness statement and Fredholm theory are also the key ingredients in Selberg's second proof
of the meromorphic continuation of Eisenstein series in the case where the inducing data is cuspidal
(cf.\ \cite{MR1008186}). Nonetheless, on a technical level our treatment is slightly easier even in this case.

%\Erez{Mention Garrett's book \cites{MR3837525, MR3837526}? Other references}

%\Erez{More history? Bernstein's 1985 letter. Include as an appendix?}

%\Erez{Mention special cases. Mention Harder?}

There are two important additional aspects of Eisenstein series which will not be addressed here.
The first is the extension of the results to non-$\K$-finite, smooth automorphic forms.
The second is the finiteness of order (as meromorphic functions in $s$) of Eisenstein series.
The first point was dealt with in \cite{MR2402686} and in a broader scope in \cite{MR3219530}, with a more recent different approach by Wallach.
The second point is addressed (in the cuspidal case) in the proof of meromorphic continuation by W. M\"uller \cite{MR1025165}*{\S4}
who extended the method of Colin de Verdi\`ere \cite{MR639175}, which in turn is based on the approach of Lax--Phillips \cite{MR1037774}.
At the moment, we do not know how to prove either of the two additional statements with the methods of the present paper.

It goes without saying that Eisenstein series play a ubiquitous role in automorphic forms well beyond
the spectral decomposition of $L^2(\Gamma\bs G)$. It is certainly beyond the scope of this paper to
discuss any of that.

\begin{comment}
\begin{remark}
Let $\Ad_P$ denote the adjoint representation of $P$ on its Lie algebra and let $\modulus_P(p)=\abs{\det\Ad_P(p)}$,
$p\in P(\A)$. (In the conventions of \cite{MR2098271}*{Ch.~7}, the modular function of $P(\A)$ is $\modulus_P^{-1}$.)
Let $\xi$ be a unitary character of $Z_M(F)\bs Z_M(\A)$ and let
$\AF_{P,\xi}^2$ \index{AFPxi@$\AF_{P,\xi}^2$} be the linear subspace of $\AF_P$ consisting of those $\varphi$ such that
\begin{gather*}
\varphi(zg)=\modulus_P^{\frac12}(z)\xi(z)\varphi(g)\ \forall z\in Z_M(\A), g\in G(\A)\text{ and }\\
\norm{\varphi}_{\AF_{P,\xi}^2}^2:=\int_{Z_M(\A)\bs\autspace_P}\abs{\varphi(g)}^2\ dg<\infty.
\end{gather*}
Then, for any $w\in\Weyl(P,Q)$, $\varphi\in\AF_{P,\xi}^2$ and $\varphi'\in\AF_{Q,w\xi}^2$
we have
\[
(M(w,\lambda)\varphi,\varphi')_{\AF_{Q,w\xi}^2}=
(\varphi,M(w^{-1},-w\overline{\lambda})\varphi')_{\AF_{P,\xi}^2}
\]
(\cite{MR1361168}*{II.1.8}).
It follows from the theorem above that $M(w,\lambda):\AF_{P,\xi}^2\rightarrow\AF_{Q,w\xi}^2$ is
holomorphic and unitary for $\Re\lambda=0$ --  cf.~\cite{MR1361168}*{IV.3.12}.
It is also true that $E(\varphi,\lambda)$ is holomorphic near $\iii\aaa_P^*$,
and this can be proved independently of Langlands's description
of the discrete spectrum in terms of residues of Eisenstein series
(cf. \cite{MR2767521}, at least in the number field case).
\end{remark}
\end{comment}

\subsection*{Acknowledgement}
We would like to thank Eitan Sayag for many discussions of the proof that led to significant simplifications of many parts of the proof
(in particular, the uniqueness property).

We are grateful to Peter Sarnak for many fruitful discussions and his input that illuminated different ideas related to this topic,
as well as for encouraging us to write up this note.

We are also indebted to Herv\'e Jacquet for his careful reading of the text and for his many suggestions.

The detailed comments by the referees were also very useful and are very much appreciated.

Finally, the first author would like to thank the Max Planck Institute for Mathematics, where a large part of this work was carried out,
for a very productive atmosphere.

\section{Preliminaries and statement of main result} \label{sec: prelim}

We will use some standard notation and results.
We refer to the standard text \cite{MR1361168} for more details.
(However, we will only use the first two ``easy'' chapters of [ibid.].)

\subsection{General notation}

Let $G$ be a reductive group over a global field $F$ with ring of adeles \index{adeles@$\A$} $\A=\A_F$.
For convenience (although it is not absolutely necessary) we fix a minimal parabolic subgroup $P_0$ of $G$ defined over $F$
with a Levi decomposition $P_0=M_0\ltimes U_0$ over $F$. \index{Par0@$P_0$} \index{M0@$M_0$} \index{U0@$U_0$}

Let $\aaa_0$ be the finite-dimensional $\R$-vector space \index{aaa0*@$\aaa_0^*$, $\aaa_0$}
\[
\aaa_0=\Hom(X^*(P_0),\R)=\Hom(X^*(M_0),\R)=\Hom(X^*(Z_{M_0}),\R)
\]
where $Z_H$ \index{ZH@$Z_H$} is the center of a group $H$ and $X^*(\cdot)$ \index{X*@$X^*(\cdot)$}
denotes the lattice of characters defined over $F$. (If $G$ is split over $F$, then $\aaa_0$ is the Cartan space.)
The dual vector space is
\[
\aaa_0^*=X^*(P_0)\otimes\R=X^*(M_0)\otimes\R=X^*(Z_{M_0})\otimes\R.
\]

Denote by $\stds$ \index{par@$\stds$} the finite set of standard parabolic subgroups of $G$ (i.e., those containing $P_0$)
that are defined over $F$.
Any $P\in\stds$ admits a unique Levi decomposition $P=M\ltimes U$ over $F$ such that $M\supset M_0$.
(If $P$ is not clear from the context, we write $M=M_P$ and $U=U_P$.) \index{MP@$M_P$} \index{UP@$U_P$}
Let $\aaa_P$ be the finite-dimensional $\R$-vector space \index{aaaP*@$\aaa_P^*$, $\aaa_P$}
\[
\aaa_P=\Hom(X^*(P),\R)=\Hom(X^*(M),\R)=\Hom(X^*(Z_M),\R).
\]
We can view $\aaa_P$ canonically both as a subspace and as a quotient of $\aaa_0$.
The dual vector space of $\aaa_P$ is
\[
\aaa_P^*=X^*(P)\otimes\R=X^*(M)\otimes\R=X^*(Z_M)\otimes\R.
\]

For any $P,Q\in\stds$ denote by $\Weyl(P,Q)$ \index{Weyl@$\Weyl$, $\Weyl(P,Q)$} the (possibly empty) finite set of cosets $wM_P(F)$, $w\in G(F)$
such that $wM_Pw^{-1}=M_Q$.
In particular, $\Weyl=\Weyl(P_0,P_0)=N_{G(F)}(M_0)/M_0(F)$ is the Weyl group of $G$.
Any $w\in\Weyl(P,Q)$ induces a linear isomorphism $\aaa_P\rightarrow\aaa_Q$, which uniquely determines $w$.

We let $\Ht_M:M(\A)\rightarrow\aaa_P$ \index{H@$\Ht_M$} be the continuous group homomorphism given by
\[
e^{\sprod{\chi}{\Ht_M(m)}}=\abs{\chi(m)},\ \ \forall m\in M(\A),\chi\in X^*(M)
\]
where we view $\chi$ as a homomorphism $M(\A)\rightarrow\A^*$.
Denote the kernel of \index{MA@$M(\A)^1$} $\Ht_M$ by $M(\A)^1$.

We denote by \index{XP@$\chars_P$} $\chars_P$ the group of continuous quasi-characters of $M(\A)/M(\A)^1=\Ht_M(M(\A))$.
(We will use additive notation for the group $\chars_P$ and write the image of $m\in M(\A)$ under $\lambda\in\chars_P$
by $m^\lambda$.)
In the number field case, $\Ht_M$ is surjective and
\[
\chars_P=X^*(M)\otimes_{\Z}\C=\aaa_P^*\otimes_{\R}\C.
\]
In the function field case, $\Ht_M(M(\A))$ is a lattice in $\aaa_P$ and
we can identify $\chars_P$ with the quotient of $\aaa_P^*\otimes_{\R}\C$ by
a lattice in $\iii\aaa_P^*$, namely the dual lattice of $\Ht_M(M(\A))$ scaled by $2\pi\iii$.
In this case we view $\chars_P$ as a complex algebraic variety isomorphic to $(\C^*)^{\dim\aaa_P}$.
In both cases the map $\Re:\chars_P\rightarrow\aaa_P^*$ is well defined.
(See \cite{MR1361168}*{I.1.4}.)

Fix once and for all a maximal compact subgroup $\K$ \index{K@$\K$} of $G(\A)$ that is
in a ``good position'' with respect to $M_0$ (\cite{MR1361168}*{I.1.4}). In particular, for any $P\in\stds$,
$G(\A)=M(\A)U(\A)\K$ and $M(\A)\cap\K$ is a maximal compact subgroup of $M(\A)$.
Denote by
\[
\m_P:G(\A)\rightarrow M(\A)/M(\A)^1
\]
the left-$U(\A)$ right-$\K$-invariant map extending the canonical projection
\[
M(\A)\rightarrow M(\A)/M(\A)^1.
\]

\subsection{Eisenstein series and intertwining operators}
Write $\autspace=G(F)\bs G(\A)$, \index{XG@$\autspace$, $\autspace_P$} and more generally,
\[
\autspace_P=U(\A)P(F)\bs G(\A)=U(\A)M(F)\bs G(\A)
\]
for any $P=M\ltimes U\in\stds$. The space $\autspace_P$ can be identified with the fibered product $(M(F)\bs M(\A))\times_{M(\A)\cap\K}\K$.
Denote by $\AF_P$ \index{AFP@$\AF_P$} the space of automorphic forms on $\autspace_P$ (\cite{MR1361168}*{I.2.17}).
For any $\varphi\in\AF_P$ and $\lambda\in\chars_P$ set \index{phivar@$\varphi_\lambda$}
$\varphi_\lambda(g)=\varphi(g)\m_P(g)^\lambda$. We have $\varphi_\lambda\in\AF_P$.
Consider the Eisenstein series defined by \index{Evarphi@$E(\varphi,\lambda)$}
\begin{equation} \label{def: eisen}
E(g,\varphi,\lambda)=\sum_{\gamma\in P(F)\bs G(F)}\varphi_\lambda(\gamma g),\ \ \ g\in G(\A).
\end{equation}
(We do not include $P$ in the notation -- hopefully it will be always clear from the context.)
The series converges absolutely and locally uniformly in $g$ and $\lambda$ for $\Re(\lambda)$ sufficiently regular in the positive
Weyl chamber of $\aaa_P^*$ (\cite{MR1361168}*{II.1.5}).
For any $w\in\Weyl(P,Q)$, the intertwining operator $M(w,\lambda):\AF_P\rightarrow\AF_Q$ \index{Minterw@$M(w,\lambda)$} is defined by
the formula\footnote{For any unipotent group $V$ defined over $F$, the Haar measure on $V(\A)$ is normalized so that $\vol(V(F)\bs V(\A))=1$.}
\[
[M(w,\lambda)\varphi]_{w\lambda}(g)=\int_{(wU_Pw^{-1}\cap U_Q)(\A)\bs U_Q(\A)}\varphi_\lambda(w^{-1}ug)\ du,\ \ \ g\in G(\A).
\]
The integral converges locally uniformly in $g$ and $\lambda$ provided that $\sprod{\Re\lambda}{\alpha^\vee}\gg0$ for every
root $\alpha\in\Phi_P$ such that $w\alpha<0$ (\cite{MR1361168}*{II.1.6}).

In the number field case, let $\umd(\autspace)$ \index{Fmod@$\umd(\autspace)$} be the space of smooth functions of
uniform moderate growth on $\autspace$
(\cite{MR1361168}*{I.2.3}).
It is a countable union of Fr\'echet spaces (see \S\ref{sec: LF}).

\subsection{The main result}

The goal of the paper is to give a soft proof of the following result.

\begin{theorem} \label{thm: main}
Let $P\in\stds$ and $\varphi\in\AF_P$.
\begin{enumerate}
\item In the number field case, the Eisenstein series $E(\varphi,\lambda)$, originally defined and holomorphic for
$\sprod{\Re\lambda}{\alpha^\vee}\gg0$ $\forall\alpha\in\srts_P$, extends to a meromorphic function
$\lambda\mapsto E(\varphi,\lambda)\in\umd(\autspace)$ on $\chars_P$.
Whenever regular, $E(\varphi,\lambda)\in\AF_G$.
\item In the function field case, there exists a polynomial $p$ on $\chars_P$ such that
$\lambda\mapsto p(\lambda)E(g,\varphi,\lambda)$ is a polynomial on $\chars_P$ for all $g\in G(\A)$.
Moreover, $p(\lambda)E(\cdot,\varphi,\lambda)\in\AF_G$ for all $\lambda\in\chars_P$.
\item For any $w\in\Weyl(P,Q)$, the map $\lambda\mapsto M(w,\lambda)\varphi$, taking values in a finite-dimensional
linear subspace of $\AF_Q$, admits a meromorphic continuation to $\chars_P$
(which is a rational function on $\chars_P$ in the function field case).
\item For any $w\in\Weyl(P,Q)$ we have the functional equation
\[
E(M(w,\lambda)\varphi,w\lambda)=E(\varphi,\lambda)\ \ \ \lambda\in\chars_P.
\]
\item For any $w\in\Weyl(P,Q)$ and $w'\in\Weyl(Q,Q')$ we have
\[
M(w'w,\lambda)=M(w',w\lambda)\circ M(w,\lambda)\ \ \ \lambda\in\chars_P.
\]
\item The singularities of $M(w,\lambda)\varphi$ are along root hyperplanes.
The same is true for the singularities of $E(\varphi,\lambda)$.
\end{enumerate}
\end{theorem}

As mentioned above, and will be explained in more detail in \S\ref{sec: sketch} below,
we use the principle of meromorphic continuation (Theorem \ref{thm: anal} below).

\section{A principle of meromorphic continuation} \label{sec: prmc}
Throughout this section, $\mnfld$ is a complex analytic manifold.
\subsection{Meromorphic functions in locally convex topological vector spaces}
Let $\spc$ be a complex, Hausdorff, locally convex topological vector space (\lctvs).
As usual, we denote by $\spc'$ (the dual of $\spc$) the space of continuous linear forms on $\spc$.

We say that a function $f:\mnfld\rightarrow \spc$ is analytic (or holomorphic) if for every $\mu\in\spc'$,
the scalar-valued function $\sprod{\mu}{f(s)}:\mnfld\rightarrow\C$ is analytic.

Let $U$ be an open dense subset of $\mnfld$. We say that a holomorphic function $f:U\rightarrow \spc$ is meromorphic on $\mnfld$
if for every $s_0\in \mnfld$ there exist a connected neighborhood $W$ and holomorphic functions $0\not\equiv g:W\rightarrow\C$ and
$h:W\rightarrow \spc$ such that $g(s)f(s)=h(s)$ for all $s\in U\cap W$.

The above notion of analyticity is discussed in \cite{MR0058865}*{\S2}.
In particular, every analytic function is continuous (cf.\ footnote 4(a) in the proof of \cite{MR0058865}*{Th\'eor\`eme 1}).
Moreover, suppose that the closed, absolutely convex hull of any compact set in $\spc$ is compact.
(This holds for any quasi-complete space, in particular for any Fr\'echet space.) Then,
\begin{itemize}
\item If $\mnfld$ is an open subset of $\C^n$,
then $f:\mnfld\rightarrow \spc$ is holomorphic $\iff$ $f$ admits partial derivatives with respect to each variable
$\iff$ $f$ admits a convergent power series expansion in $\spc$ around every point of $\mnfld$ (\cite{MR0058865}*{Th\'eor\`eme 1}).
\item A function $f:\mnfld\rightarrow\spc$ is analytic if and only if
it is continuous and $\sprod{\mu}{f(s)}:\mnfld\rightarrow\C$ is analytic for all $\mu$ in a separating subset of $\spc'$
(\cite{MR0058865}*{\S2, Remarque 1}).
(Recall that a subset of $\spc'$ is called separating if its annihilator in $\spc$ is trivial.)
This gives a practical criterion to check whether a function is analytic.
\end{itemize}

\begin{example} \label{ex: ccinfty}
Suppose that $F$ is a number field and let $C_c^\infty(G(\A))$ \index{Ccinfty@$C_c^\infty(G(\A))$} be the algebra (under convolution)
of compactly supported, smooth functions on $G(\A)$.
As a \lctvs, it is the strict inductive limit (cf.\ \cite{MR910295}*{\S II.4.6}), over the compact subsets $C$ of $G(\A)$
and the open subgroups $K$ of $G(\A_f)$, of the Fr\'echet spaces
of bi-$K$-invariant functions in $G(\A)$ that are supported on $C$
and are $C^\infty$ as a function of $G(F_\infty)$.
Let $\univ(\gggg)$ \index{univ@$\univ(\gggg)$} be the universal enveloping algebra
of the complexification $\gggg$ of the Lie algebra of $G(F_\infty)$.

Suppose for simplicity that $\mnfld$ is connected.
Then, a function $h:\mnfld\rightarrow C_c^\infty(G(\A))$ (i.e., a family $h_s$, $s\in\mnfld$ of smooth,
compactly supported functions on $G(\A)$) is holomorphic
if and only if the following conditions are satisfied (cf.\ \cite{MR0058865}*{\S3}).
\begin{enumerate}
\item There exists a compact subset $C$ of $G(\A)$ such that $\supp h_s\subset C$ for all $s\in \mnfld$.
\item There exists an open subgroup $K$ of $G(\A_f)$ such that $h_s$ is bi-$K$-invariant for all $s\in \mnfld$.
\item For any $g\in G(\A)$, the function $s\mapsto h_s(g)$ is analytic.
\item For any $X\in\univ(\gggg)$, viewed as a differential operator on $G(F_\infty)$, the function $Xh_s$ is continuous on $\mnfld\times G(\A)$.
\end{enumerate}
In this case, we refer to $h_s$ as an analytic family of smooth, compactly supported functions on $G(\A)$.
\end{example}

We refer the reader to \cite{MR2357988} and the references therein for more discussion about analytic functions
and their subtleties, including some interesting counterexamples.

\subsection{Analytic families of operators}
Let $\spc$ and $\spcc$ be two Hausdorff \lctvs s.
For brevity, by an \emph{operator} from $\spc$ to $\spcc$ we will always mean a continuous linear map.
We denote by $\Lin(\spc,\spcc)$ \index{Lin@$\Lin(\spc,\spcc)$, $\Lins(\spc,\spcc)$, $\Linb(\spc,\spcc)$}
the space of operators from $\spc$ to $\spcc$.
Consider the pointwise convergence topology on $\Lin(\spc,\spcc)$, i.e., the coarsest topology for which the
evaluation maps $e_v:\Lin(\spc,\spcc)\rightarrow \spcc$, $v\in \spc$ given by $A\mapsto A(v)$,
are continuous. Equivalently, it is the Hausdorff, locally convex topology defined by the seminorms $p(A(v))$ where $v\in \spc$
and $p$ is a continuous seminorm on $\spcc$.
We write $\Lins(\spc,\spcc)$ for $\Lin(\spc,\spcc)$ with this topology.
Its dual space can be identified with the algebraic tensor product $\spc\otimes\spcc'$
\cite{MR0372565}*{Proposition 23 on p. 79}.
Thus (cf.\ \cite{MR0058865}*{\S2, Remarque 2}),
a function $A:\mnfld\rightarrow\Lins(\spc,\spcc)$ (i.e., a family of operators
$A_s:\spc\rightarrow \spcc$, $s\in \mnfld$) is holomorphic if and only if for every $v\in \spc$
the function $s\mapsto A_s(v)\in \spcc$ is holomorphic, or equivalently, for every $v\in\spc$ and
$\mu\in\spcc'$ the function $s\mapsto\mu(A_s(v))$ is holomorphic.
In this case, we will simply say that $A_s$, $s\in\mnfld$ is a holomorphic family of operators.

We may also consider the finer topology on $\Lin(\spc,\spcc)$ of uniform convergence on bounded sets, which is given by the seminorms
$\sup_{v\in B}p(A(v))$ where $p$ is a continuous seminorm on $\spcc$ and $B$ is a bounded subset of $\spc$.
We write $\Linb(\spc,\spcc)$ for $\Lin(\spc,\spcc)$ with this topology.
For instance, if $\spc$ and $\spcc$ are Banach spaces, then $\Linb(\spc,\spcc)$ is the Banach space with the
usual operator norm. Of course, if $\spc$ is finite-dimensional, then $\Lins(\spc,\spcc)$ and $\Linb(\spc,\spcc)$ coincide,
but otherwise the topologies are different.
In principle, we could have defined a strong analytic family of operators as an analytic function
from $\mnfld$ to $\Linb(\spc,\spcc)$.
However, it follows from the uniform boundedness principle (see \cite{MR0058865}*{\S2, Remarque 2} and
\cite{MR910295}*{\S III.4.3, Corollary 1})
that any analytic family of operators from $\spc$ to $\spcc$ is automatically analytic in the strong
sense if $\spc$ is barrelled (in particular, if $\spc$ is a Fr\'echet space,
or more generally, an arbitrary inductive limit of Fr\'echet spaces)
or if $\spc$ is semi-complete (i.e., every Cauchy sequence converges).
Fortunately, all \lctvs s considered in the body of the paper will be barrelled,
so we will not need to make the distinction between analytic and strong analytic families of operators.

If $\spc$, $\spcc$, $\spccc$ are Hausdorff \lctvs s and $A_s:\spc\rightarrow\spcc$
and $B_s:\spcc\rightarrow\spccc$, $s\in\mnfld$ are analytic families of operators,
then $B_s\circ A_s$ is an analytic family of operators from $\spc$ to $\spccc$.
This follows from Hartogs's Theorem on separate holomorphicity \cite{MR1846625}*{\S0.2 and Theorem 1.2.5}.
A similar statement holds for strongly analytic families
(although as was just pointed out, we will not need it).
Note that this argument makes the additional assumptions in \cite{MR0058865}*{\S2, Remarque 4} unnecessary.

\subsection{Analytic systems of linear equations}

Let $V$ be a vector space.
By a system $\Xi$ of linear equations on $v\in V$ we will mean
a collection $(W_i,A_i,w_i)$, $i\in I$ of triples consisting of a vector space $W_i$
a linear transformation $A_i:V\rightarrow W_i$ and a vector $w_i\in W_i$. The equations take the form
\[
A_iv=w_i,\ \ i\in I.
\]
We denote by $\Sol(\Xi)$ \index{Sol@$\Sol(\Xi)$} the set of solution of $\Xi$ in $V$.

It is easy to make sense of an analytic family of systems of linear equations, as follows.

\begin{definition} \label{def: anal}
Let $\spc$ be a Hausdorff \lctvs.
\begin{enumerate}
\item Let $(\spcc_i,\mu_i,c_i)$, $i\in I$ be a (possibly infinite) family of triples consisting of
\begin{itemize}
\item A Hausdorff \lctvs{} $\spcc_i$.
\item An analytic family $(\mu_i)_s$, $s\in\mnfld$ of operators from $\spc$ to $\spcc_i$.
\item An analytic function $c_i:\mnfld\rightarrow \spcc_i$.
\end{itemize}
We say that the system $\Xi(s)$ of linear equations (on $v\in \spc$)
\[
(\mu_i)_s(v)=c_i(s),\ i\in I
\]
depends analytically on $s$ (or simply, is an analytic family).
\item Let $A_s$, $s\in\mnfld$ be a family of subsets of $\spc$. We say that $A_s$ is of finite type if
there exist a finite-dimensional vector space
$L$ and an analytic family $\lambda_s$, $s\in\mnfld$ of injective operators $L\rightarrow\spc$ such that
$A_s\subset\Img\lambda_s$ for all $s\in\mnfld$.
%%If moreover $\lambda_s$ can be chosen to be injective for all $s\in\mnfld$, then we say that $A_s$ is strongly of finite type.
\item We say that $A_s$ is locally of finite type if for every $s_0\in\mnfld$ there exists an open neighborhood $W$
in $\mnfld$ such that $A_s$, $s\in W$ is of finite type.
\item Finally, we say that the family of equations $\Xi=(\Xi(s))_{s\in \mnfld}$ is (locally) of finite type
if the same is true for $\Sol(\Xi(s))$.
\end{enumerate}
\end{definition}

\begin{theorem}[Principle of meromorphic continuation] \label{thm: anal}
Let $\Xi=(\Xi(s))_{s\in \mnfld}$ be an analytic family of systems of linear equations
that is locally of finite type. Let \index{Munq@$\mnfld_{\unq}$}
\[
\mnfld_{\unq}=\{s\in \mnfld\mid\Sol(\Xi(s))=\{v(s)\}\}
\]
be the set of $s\in \mnfld$ for which the system $\Xi(s)$ admits a unique solution $v(s)$.
Suppose that $\mnfld$ is connected and that the interior $\mnfld_{\unq}^\circ$ of $\mnfld_{\unq}$ is nonempty.
Then,
\begin{enumerate}
\item $\mnfld_{\unq}$ is open and its complement is analytic.
\item $v$ is holomorphic on $\mnfld_{\unq}$.
\item $v$ is meromorphic on $\mnfld$.
\end{enumerate}
\end{theorem}

%%%Then, $\mnfld_{\unq}$ contains an open dense subset $U$ of $\mnfld$ such that $v$ is holomorphic on $U$ and meromorphic on $\mnfld$.

%%\begin{remark}
%%We do not know whether one can dispense with the stronger assumption for the last part of the theorem.
%%\end{remark}

The proof, which is a simple application of Cramer's rule, will be given in the appendix.

\subsection{}

We conclude this section by describing the basic tool for proving that a system is locally of finite type,
namely Fredholm theory.

We first recall that for any compact operator $K$ on a Banach space $\bnch$ the operator $T=\Id_{\bnch}-K$
is Fredholm, and hence it is invertible modulo operators of finite rank, that is, there exists
an operator $S$ on $\bnch$ such that $ST=\Id_{\bnch}+F$ and $TS=\Id_{\bnch}+F'$ for finite rank operators
$F$ and $F'$ \cite{MR1865513}*{Remark 3.3.3}. (In fact, since $T$ is of index $0$, we can take $S$ to be invertible.)

\begin{lemma}[Fredholm's criterion] \label{fred1}
Let $\bnch,\bnchh$ be Banach spaces and let $\mu_s$, $s\in\mnfld$ be an analytic family of operators from $\bnch$ to $\bnchh$.
Suppose that for some $s_0\in \mnfld$, $\mu_{s_0}$ is left-invertible modulo compact operators.
Then, the homogeneous system $\Xi(s)$ (in $\bnch$) given by
\[
\mu_sv=0
\]
is of finite type near $s_0$.
\end{lemma}

\begin{proof}
By the remark above, $\mu_{s_0}$ is left-invertible modulo finite-rank operators.
Thus, there exists an operator $D:\bnchh\rightarrow \bnch$ such that $F:=\Id_\bnch-D\mu_{s_0}:\bnch\rightarrow \bnch$ is of finite rank.
Replacing $\mu_s$ by $D\mu_s$ we may assume that $\bnchh=\bnch$ and $\mu_{s_0}=\Id_{\bnch}-F$.
For $s$ close to $s_0$ we have $\mu_s = X_s - F$ where $X_s$ is invertible and $X_s^{-1}$ is holomorphic in $s$.
However, it is clear that $\Ker(X_s - F) \subset X_s^{-1}\Img(F)$. Therefore, the solutions of our system are contained in
the holomorphic family $X_s^{-1}(L)$, where $L =\Img(F)$.
\end{proof}

In practice, we will use it in the following way.
We say that a \lctvs{} is \emph{Hilbertian} if it is isomorphic to a Hilbert space.
(That is, we do not care about the inner product itself, only the topology.)

\begin{corollary} \label{cor: fred2}
Let $\bnch,\bnchh$ be Hilbertian \lctvs s and let $\mu_s,\nu_s:\bnch\rightarrow\bnchh$, $s\in\mnfld$ be
two analytic families of operators.
Suppose that for some $s_0\in \mnfld$, $\mu_{s_0}$ is a strict embedding and $\nu_{s_0}$ is a compact operator.
Then, the homogeneous system in $\bnch$ given by
\[
\mu_sv=\nu_sv
\]
is of finite type near $s_0$.
\end{corollary}

Indeed, every strict embedding of Hilbertian spaces is left invertible.

\begin{comment}
We can extend it as follows.

\begin{corollary} \label{cor: fred2}
Let $\bnch$, $\bnchh$, $\mu_s$, $s\in\mnfld$ and $\Xi(s)$ be as above.
Suppose that $\bnch$ admits a direct sum decomposition $\bnch=\bnch_1\dsum \bnch_2$ and let
$p_i:\bnch\rightarrow\bnch_i$, $i=1,2$ be the corresponding projections.
Assume that $\mu_{s_0}\rest_{\bnch_1}$ is left-invertible modulo compact operators
and that the family $p_2(\Sol(\Xi(s)))$ of subsets of $\bnch_2$ is locally (strongly) of finite type.
Then, $\Xi(s)$ is (strongly) of finite type near $s_0$.
\end{corollary}

\begin{proof}
By passing to a neighborhood of $s_0$ and using the second condition, we may assume without loss of generality that
there exist a finite-dimensional space $L$ and an analytic
family $\nu_s$, $s\in\mnfld$ of (injective) operators from $L$ to $\bnch_2$ such that
$\Sol(\Xi(s))\subset \bnch_1\dsum\Img(\nu_s)$ for all $s\in\mnfld$.
Consider the system $\Xi_1(s)$ on $\bnch_1\dsum L$ given by
\[
\mu_s\circ(\Id_{\bnch_1}\dsum\nu_s)v=0.
\]
By assumption, $\mu_{s_0}\circ(\Id_{\bnch_1}\dsum\nu_{s_0})$ is left-invertible modulo compact operators.
Therefore, by the lemma above, $\Xi_1(s)$ is strongly of finite type near $s_0$.
Since $\Sol(\Xi(s))=(\Id_{\bnch_1}\dsum\nu_s)(\Sol\Xi_1(s))$, it follows that $\Xi(s)$ is (strongly) of finite type near $s_0$.
\end{proof}
\end{comment}

\section{The system of linear equations} \label{sec: sketch}
%\subsection{} \label{sec: xi12}
Let us give a brief outline of how we will apply the Principle of Meromorphic Continuation (Theorem \ref{thm: anal})
to prove our main result on meromorphic continuation of Eisenstein series.
Fix $P\in\stds$ and $\varphi\in\AF_P$.
In \S\ref{sec: equations} we devise a certain holomorphic system of linear equations $\Xi(\lambda)$,
$\lambda\in\chars_P$ on $\psi$ in the space $\umd(\autspace)$ of smooth functions of uniform moderate growth on $\autspace$
(see \S\ref{sec: LF} below).
Roughly, the system comprises the following two sets of equations.
\begin{enumerate}[label=($\Xi_{\arabic*}$)]
%\item The $\K$-types of $\psi$ are contained in those of $\varphi$.
\item $\psi$ is an eigenfunction, with a non-zero eigenvalue, of an integral operator
(namely, convolution by a smooth, bi-$\K$-finite, compactly supported function on $G(\A)$,
depending holomorphically on $\lambda$).
%\item The cuspidal support of $\psi$ is prescribed by the cuspidal support of $\varphi$.
%\item The exponents of $\psi$ along any parabolic subgroup are described in terms of the exponents
%of $\varphi$ and $\lambda$ in a simply way.
\item The cuspidal components of $\psi$ along any parabolic subgroup differ from those of $\varphi_\lambda$ by
terms with a prescribed set $A_\lambda$ of cuspidal exponents.
\end{enumerate}

For $\Re\lambda$ dominant and sufficiently regular (depending on $\varphi$),
the Eisenstein series $E(\varphi,\lambda)$ satisfies these equations.
This follows from the computation of the constant term of $E(\varphi,\lambda)$ (Lemma \ref{lem: GL}
and its Corollary) and a result of Harish-Chandra (Lemma \ref{lem: HC}).
Moreover, and this is a crucial point, in this region
the real parts of the elements of $A_\lambda$ are away from the closure of the positive Weyl chamber.
This fact is used in Proposition \ref{prop: unique}, which is the linchpin of the argument, to show that in this region
$E(\varphi,\lambda)$ is the unique automorphic form that satisfies $\Xi_2$.

The remaining issues are to show that any solution $\psi$ of $\Xi(\lambda)$
(for any $\lambda$) is an automorphic form and that the system $\Xi(\lambda)$ is locally of finite type.
In fact, the second point implies the first one (Lemma \ref{lem: autcusp}).
The local finiteness is merely a technical refinement of the results and techniques of Harish-Chandra.
It follows from Theorem \ref{thm: mainfin} which is the technical heart of the paper.
We refer the reader to \S\ref{sec: mainsys} for more details.

To summarize, the system $\Xi(\lambda)$ is locally of finite type and it admits $E(\varphi,\lambda)$ as its unique solution provided
that $\sprod{\Re\lambda}{\alpha^\vee}\gg0$ for all $\alpha\in\srts_P$ (Proposition \ref{prop: Ximain}).
The principle of meromorphic continuation will immediately imply the first part of Theorem \ref{thm: main}.
The other parts are then an easy consequence.

In the function field case the situation is easier (see \S\ref{sec: conclusion2}):
we replace $\umd(\autspace)$ simply by the space of smooth functions on $\autspace$
and use an algebraic version of the principle of meromorphic continuation which does not require local finiteness.
(The equation $\Xi_1$ is automatic in this case.)

\section{An example: \texorpdfstring{$\SL_2$}{SL2}} \label{sec: SL2}
We will illustrate the idea for the classical Eisenstein series for $\Gamma=\SL_2(\Z)$ \cite{MR31519}.
The proof in this case is not fundamentally different from Selberg's second proof of the meromorphic
continuation of Eisenstein series, although the principle of meromorphic continuation allows for a technical simplification in the argument.
In particular, no truncation is necessary.
Let $G=\SL_2(\R)$ and let $\HHH$ be the upper half-plane, with the left $G$-action by M\"obius transformations.
We identify $\HHH$ with $G/K$ where $K=\operatorname{SO}(2)$
and consider the $G$-invariant measure $\mu=\frac{dx\, dy}{y^2}$ on $\HHH$.
Let $\autspace=\Gamma\bs\HHH$. We view functions on $\autspace$ as $\Gamma$-invariant functions
on $\HHH$.
Whenever convergent, denote by $(\cdot,\cdot)_\autspace$ the $G$-invariant sesquilinear form
\[
(f_1,f_2)_{\autspace}=\int_{\autspace}f_1(z)\overline{f_2(z)}\ \mu(z)
\]
on functions on $\autspace$.

Any function on $\HHH$ can be lifted to a right $K$-invariant function on $G$.
Let
\begin{equation} \label{eq: defy}
\y:\HHH\rightarrow\R_{>0},\ \ \y(x+\iii y)=y
\end{equation}
be the imaginary part, considered also as a right $K$-invariant function on $G$.

Denote by $C_c^\infty(G//K)$ the algebra of smooth, bi-$K$-invariant, compactly supported functions on $G$.
This algebra acts on the right on locally $L^1$ functions on $\HHH$.
We denote this action by $f\mapsto\delta(h)f$. We have
\[
\delta(h)\y^{s+\frac12}=\hat h(s)\y^{s+\frac12}
\]
where $\hat h(s)$ is an entire function, which can be computed explicitly. All we need to know is the elementary fact that
for every $s\in\C$ there exists $h\in C_c^\infty(G//K)$ such that $\hat h(s)\ne0$.

For any $N\ge1$ let $\umd^N(\autspace)$ be the space of smooth functions $f$ on $\autspace$ such that
their lift $\tilde f$ to $G$ satisfies
\[
\abs{(\delta(X)\tilde f)(g)}\ll \y(g)^N\text{ when }\y(g)\ge1
\]
for any $X\in\univ(G)$ (acting on the right) where the implied constant depends on $X$.
This is a Fr\'echet space. Let $\umd(\autspace)$ be the union over $N\ge1$ of $\umd^N(\autspace)$ with the inductive
limit topology in the category of \lctvs s. (It is a Hausdorff space.)

Let $\Gamma_\infty=\{\left(\begin{smallmatrix}1&n\\0&1\end{smallmatrix}\right)\mid n\in\Z\}\subset\Gamma$.
We denote by $\cnst f$ the constant term of a $\Gamma_\infty$-invariant function $f$ on $\HHH$, i.e.
\[
\cnst f(y)=\int_{\Z\bs\R}f(x+\iii y)\ dx=\int_0^1f(x+\iii y)\ dx,\ \ y>0.
\]

Consider the Eisenstein series
\[
E(z;s)=\sum_{\gamma\in\pm\Gamma_\infty\bs\Gamma}\y(\gamma z)^{s+\frac12}=
\tfrac12\sum_{(m,n)\in\Z^2\mid\gcd(m,n)=1}\frac{\y(z)^{s+\frac12}}{\abs{mz+n}^{2s+1}}.
\]
The series converges absolutely for $\Re s>\frac12$ and defines a function in $\umd(\autspace)$.

We use three elementary properties of $E(z;s)$.
\begin{claim} \label{claim: 1}
Consider the region $\Re s>\frac12$.
\begin{enumerate}
\item For any $h\in C_c^\infty(G//K)$ we have
\[
\delta(h)E(\cdot;s)=\hat h(s)E(\cdot;s).
\]
\item There exists a holomorphic function $m(s)$ such that
\[
\cnst E(y;s)=y^{s+\frac12}+m(s)y^{-s+\frac12},\ \ \ y>0.
\]
\item We have $(E(\cdot;s),f)_{\autspace}=0$ for any cusp form $f$ on $\autspace$.
\end{enumerate}
\qed
\end{claim}

\begin{remark}
In reality,
\[
m(s)=\pi^{\frac12}\frac{\Egam(s)\zeta(2s)}{\Egam(s+\frac12)\zeta(2s+1)}\text{ where }\zeta(s)=\sum_{n=1}^\infty\frac1{n^s}
\ \ \text{and }\Egam(s)=\int_0^\infty e^{-t}t^s\frac{dt}t
\]
but we are of course not entitled to know a priori that $m(s)$ has meromorphic continuation to $\C$.
Rather, this will be a consequence of the meromorphic continuation of $E(z;s)$ and would give a proof of the meromorphic
continuation of the Riemann zeta function which is different in nature from the original one using the Poisson summation formula.
\end{remark}

For any $a>0$ denote by $T_a$ the normalized shift operator on functions on $\R_{>0}$ given by
\[
T_af(y)=a^{-\frac12}f(ay).
\]
These operators pairwise commute.
Consider the holomorphic system $\Xi(s)$, $s\in\C$ on $\psi\in\umd(\autspace)$
given by the following three sets of linear equations:\footnote{In the notation of \S\ref{sec: sketch},
$\Xi_1$ corresponds to \eqref{eq: 1cls} and $\Xi_2$ with respect to the Borel subgroup and $G$ itself
correspond to \eqref{eq: 2csls} and \eqref{eq: 3cls} respectively.}
\begin{subequations}
\begin{gather}
\label{eq: 1cls} \delta(h)\psi=\hat h(s)\psi\ \ \ \forall h\in C_c^\infty(G//K),\\
\label{eq: 2csls} (T_a-a^{-s})(\cnst\psi(y)-y^{s+\frac12})\equiv0\ \ \forall a>0,\\
\label{eq: 3cls} (\psi,f)_{\autspace}=0\ \ \text{for every cusp form $f$ on $\autspace$}.
\end{gather}
\end{subequations}

\begin{claim}
The Eisenstein series $E(z;s)$ is the unique solution of $\Xi(s)$ in the region $\Re s>\frac12$.
\end{claim}

\begin{proof}
By Claim \ref{claim: 1}, $E(z;s)$ satisfies $\Xi(s)$ whenever $\Re s>\frac12$.
Conversely, suppose that $\psi\in\Sol(\Xi(s))$ and $\Re s>\frac12$. We need to show that $\psi=E(z;s)$.
Consider $\psi'=\psi-E(z;s)$.
By \eqref{eq: 2csls}, the constant term $\cnst\psi'$ is proportional to $\y^{\frac12-s}$.
In particular, it is bounded for $\y\ge\frac12$, say. On the other hand, for any $f\in\umd(\autspace)$
the function $f-\cnst f$ is bounded (and in fact, rapidly decreasing in $\y$) for $\y\ge\frac12$.
Therefore, $\psi'$ is bounded on $\y\ge\frac12$, and hence on $\autspace$ (since it is $\Gamma$-invariant).
This implies that $\cnst\psi'$ is bounded (not just for $\y\ge\frac12$).
Since $\cnst\psi'$ is proportional to $\y^{\frac12-s}$ and $\Re s>\frac12$ this means that $\cnst\psi'\equiv0$.
Thus, $\psi'$ is a cusp form and by \eqref{eq: 3cls}, $\psi'\equiv0$, i.e., $\psi=E(z;s)$.
(Note that we didn't use the equations \eqref{eq: 1cls} for this argument.)
\end{proof}

In order to apply the principle of meromorphic continuation (Theorem \ref{thm: anal}), it remains to prove the following.
\begin{claim} \label{claim: 2}
The system $\Xi(s)$ is locally of finite type.
\end{claim}

We sketch the proof.
The idea is to pass to an auxiliary system on a Hilbert space and to replace the complicated space
$\autspace$ by a simpler one which approximates it at the cusp.

Let $\zspace=\Gamma_\infty\bs\HHH$.
The function $\y$ of \eqref{eq: defy} descends to a function $\zspace\rightarrow\R_{>0}$
which we also denote by $\y$.
For any $c>0$ consider the inverse image $\zspace^c$  of $(c,\infty)$ under $\y$ in $\zspace$.
The restriction of the projection map $\proj:\zspace\rightarrow\autspace$ to $\zspace^c$ is finite-to-one.
It is surjective if $c$ is sufficiently small and injective if $c$ is sufficiently large (in fact for $c>1$).
Fix $c_0$ such that $\proj\rest_{\zspace^{c_0}}$ is surjective.

For $N>0$ consider the Hilbert space $\funct^N(\zspace)=L^2(\zspace;\y^{-2N}\mu)$
and its quotient $\funct^N(\zspace^{c_0})=L^2(\zspace^{c_0};\y^{-2N}\mu)$.
For any function $f$ on $\autspace$ denote by $f^{c_0}=f\circ\proj\rest_{\zspace^{c_0}}$ its pullback to $\zspace^{c_0}$.
The inverse image $\funct^N(\autspace)$ of $\funct^N(\zspace^{c_0})$ under $f\mapsto f^{c_0}$ is a Hilbertian space of functions on $\autspace$
which is independent of the choice of $c_0$. We thus have a strict embedding
\[
\funct^N(\autspace)\rightarrow\funct^N(\zspace^{c_0}), \ \ f\mapsto f^{c_0}.
\]
On the other hand, the constant term
%\[
%\int_{\Z\bs\R}f(x+\cdot)\ dx
%\]
defines an orthogonal projection $\CT$ of $\funct^N(\zspace^{c_0})$ whose image is the space of functions $\funct^N_{\y}(\zspace^{c_0})$ that factor through $\y$.
Thus, we have an orthogonal decomposition
\[
\funct^N(\zspace^{c_0})=\funct^N_{\cusp}(\zspace^{c_0})\oplus\funct^N_{\y}(\zspace^{c_0})\text{ where }\funct^N_{\cusp}(\zspace^{c_0})=\Ker\CT.
\]
For any function $f$ on $\autspace$ we write $f^{c_0}=f^{c_0}_{\cusp}+f^{c_0}_{\y}$ with respect to this decomposition.
Fix $a>1$. We consider the operator $T_a$ on $\funct^N_{\y}(\zspace^{c_0})$.
Let $\zspace^{c_0,c_0a^2}\subset\zspace^{c_0}$ be the rectangle $c_0<y\le c_0a^2$.
Let $h\in C_c^\infty(G//K)$ and let $c(s)$ be an entire function.
Consider the holomorphic system $\tilde\Xi^N(s)$ of linear equations on $f\in\funct^N(\autspace)$
\begin{align*}
c(s)f^{c_0}_{\cusp}&=(\delta(h)f)^{c_0}_{\cusp}\\
c(s)f^{c_0}_{\y}\rest_{\zspace^{c_0,c_0a^2}}&=(\delta(h)f)^{c_0}_{\y}\rest_{\zspace^{c_0,c_0a^2}}\\
(T_a-a^s)(T_a-a^{-s})(f^{c_0}_{\y})&=0.
\end{align*}

\begin{claim} \label{claim: 3}
The system $\tilde\Xi^N(s)$ is of Fredholm type for $\abs{s}<N$ and $c(s)\ne0$, and hence it is locally of finite type.
More precisely,
\begin{enumerate}
\item The operator
\begin{align*}
\funct^N(\autspace)&\rightarrow\funct^N_{\cusp}(\zspace^{c_0})\oplus\funct^N_{\y}(\zspace^{c_0})\oplus\funct^N_{\y}(\zspace^{c_0,c_0a^2})\\
f&\mapsto (f^{c_0}_{\cusp},(T_a-a^s)(T_a-a^{-s})(f^{c_0}_{\y}),f^{c_0}_{\y}\rest_{\zspace^{c_0,c_0a^2}})
\end{align*}
is a strict embedding.
\item The operator
\[
\funct^N(\autspace)\rightarrow\funct^N_{\cusp}(\zspace^{c_0}),\ \    f\mapsto(\delta(h)f)^{c_0}_{\cusp}
\]
is compact (and in fact Hilbert--Schmidt).
\end{enumerate}
\end{claim}

\begin{proof}
The second statement is well known (cf.\ \cite{MR1482800}*{Theorem 9.5}).

For the first one, the operator in question is the composition of the strict embedding
\[
\funct^N(\autspace)\rightarrow\funct^N_{\cusp}(\zspace^{c_0})\oplus\funct^N_{\y}(\zspace^{c_0}),\ \ f\mapsto (f^{c_0}_{\cusp},f^{c_0}_{\y})
\]
with the operator
\begin{align*}
\funct^N_{\y}(\zspace^{c_0})&\rightarrow\funct^N_{\y}(\zspace^{c_0})\oplus\funct^N_{\y}(\zspace^{c_0,c_0a^2})\\
f&\mapsto ((T_a-a^s)(T_a-a^{-s})f,f\rest_{\zspace^{c_0,c_0a^2}}).
\end{align*}
It is easy to show that the latter is a strict embedding for $\abs{s}<N$.
This boils down to the elementary fact that the operator
\[
L^2(\R_+,e^{-2rx}\ dx)\rightarrow L^2(\R_+,e^{-2rx}\ dx)\oplus L^2([0,1]),\ \ f\mapsto (f(x+1)-\lambda f(x),f\rest_{[0,1]})
\]
is a strict embedding provided that $e^r>\abs{\lambda}$.
\end{proof}

We can now complete the proof of Claim \ref{claim: 2}.

Let $s_0\in\C$. Choose $h\in C_c^\infty(G//K)$ such that $\hat h(s_0)\ne0$.
Suppose that $f$ is a solution of $\Xi(s)$.
Then, the constant term $\cnst f$ satisfies
\[
(T_a-a^s)(T_a-a^{-s})f=0.
\]
Since $f-\cnst f$ is rapidly decreasing, this implies that $f\in\funct^N(\autspace)$ provided that $N>\abs{s}$,
Thus, $f$ is a solution of $\tilde\Xi^N(s)$ with respect to $c(s)=\hat h(s)$.
By Claim \ref{claim: 3} there exists an open neighborhood $U$ os $s_0$, a finite-dimensional vector space
$L$ and a holomorphic family of injective linear maps $\tilde\lambda_s:L\rightarrow\funct^N(\autspace)$ such that
$\Img(\tilde\lambda_s)\supset(\Sol(\Xi(s)))$ for all $s\in U$.

By a standard result (cf.\ \cite{MR1482800}*{Proposition 5.7}), $\delta(h)$ is continuous from $\funct^N(\autspace)$ to $\umd^{N'}(\autspace)$
provided that $N'$ is sufficiently large (depending on $N$).
We deduce that $\lambda_s:=\delta(h)\tilde\lambda_s$ is a holomorphic system of
operators $L\rightarrow\umd^{N'}(\autspace)$ and $\Sol(\Xi(s))\subset\Img(\lambda_s)$ for all $s\in U$.
Note that $\lambda_{s_0}$ may not necessarily be injective at our given point $s_0$.
However, let $u\in C_c^\infty(G//K)$ and let $v_s=h+\hat h(s)u-u*h=h+u*(\hat h(s)\one_K-h)$.
Then, $\delta(v_s)f=\hat h(s)f$ whenever $\delta(h)f=\hat h(s)f$, so that
$\Sol(\Xi(s))\subset\Img(\delta(v_s)\tilde\lambda_s)$. If we take $u$ to be non-negative, supported near $K$
and of total mass $1$, then $\delta(u)$ acts approximately as the identity on the finite-dimensional space
$\Img((\hat h(s_0)I-\delta(h))\tilde\lambda_{s_0})$. This will ensure that $\delta(v_s)\tilde\lambda_s$
is close to $\hat h(s)\tilde\lambda_s$, and hence is injective near $s_0$.
Thus, $\Xi(s)$ is locally of finite type.

Note that the equations \eqref{eq: 3cls} are superfluous for the local finiteness.

\begin{remark}
Grosso modo, the general case follows the same pattern, except that there are more parabolic subgroups
and the constant terms are more complicated.
\end{remark}

\section{Uniqueness} \label{sec: unique}

In this section we show that any automorphic form is determined by its cuspidal components pertaining to the unnormalized cuspidal
exponents whose real parts are in the closure of the positive Weyl chamber. (See Theorem \ref{thm: lead} for the precise statement.)
This will give a simple characterization (which can serve as an alternative definition) of the Eisenstein series $E(\varphi,\lambda)$
with $\Re\lambda$ dominant and sufficiently regular (Proposition \ref{prop: unique}).

\subsection{Roots and coroots} (\cite{MR1361168}*{I.1.6})

Let $S_0$ \index{S0@$S_0$, $S_M$} be the maximal $F$-split torus in the center of $M_0$.
More generally, for any $P=M\ltimes U\in\stds$ let $S_M\subset S_0$ be the maximal split torus of $Z_M$,
so that $S_0=S_{M_0}$. Thus, $M=C_G(S_M)$,
\[
\aaa_P^*=X^*(S_M)\otimes\R
\]
and
\[
\aaa_P=X_*(S_M)\otimes\R
\]
where $X_*(\cdot)$ \index{X*@$X_*(\cdot)$} is the lattice of co-characters defined over $F$.
We also write $(\aaa_0^P)^*=X^*(S_0^M)\otimes\R$ and the dual space
$\aaa_0^P=X_*(S_0^M)\otimes\R$ where \index{aaa0P@$\aaa_0^P$} \index{S0M@$S_0^M$} \index{Mder@$M^{\der}$}
$S_0^M=S_0\cap M^{\der}$, a maximal split torus in the derived group $M^{\der}$ of $M$.
%\Erez{Always connected?}
We have direct sum decompositions
\[
\aaa_0=\aaa_0^P\oplus\aaa_P,\ \ \aaa_0^*=(\aaa_0^P)^*\oplus\aaa_P^*.
\]
For any $\lambda\in\aaa_0^*$ we write  \index{lambda0P@$\lambda_0^P$, $\lambda_P$}
$\lambda=\lambda_0^P+\lambda_P$ according to this decomposition.
More generally, for any $P\subset Q$ we have a decomposition
\[
\aaa_P=\aaa_P^Q\oplus\aaa_Q,\ \ \aaa_P^*=(\aaa_P^Q)^*\oplus\aaa_Q^*
\]
where $\aaa_P^Q=\aaa_P\cap\aaa_0^Q$. Accordingly, we write $\lambda=\lambda_P^Q+\lambda_Q$
for any $\lambda\in\aaa_P^*$.

Let $\srts_0\subset X^*(S_0)\subset\aaa_0^*$ \index{Daelta@$\srts_0$, $\srts_0^P$, $\srts_P$} be the set of simple roots of $S_0$ on $U_0$
and let $\srts_0^\vee\subset\aaa_0$ be the set of simple coroots.
Thus, $\srts_0$ is a basis for the vector space $(\aaa_0^G)^*$ and $\srts_0^\vee$ is a basis for $\aaa_0^G$.
For any $\alpha\in\srts_0$ denote by $\alpha^\vee\in\srts_0^\vee$ the corresponding simple coroot.
For any $P$ let $\srts_0^P\subset\srts_0$ \index{Daeltav@$\srts_0^\vee$, $\srts_P^\vee$} be the set of simple roots of $S_0$ on $U_0\cap M=U_0\cap M^{\der}$.
(Thus, $\srts_0^P$ is a basis for $(\aaa_0^P)^*$.)
Denote by $\srts_P\subset\aaa_P^*$ the image of $\srts_0\setminus\srts_0^P$ under the projection
$\aaa_0^*\rightarrow\aaa_P^*$. (This defines a bijection between $\srts_0\setminus\srts_0^P$ and $\srts_P$.)
Similarly, $\srts_P^\vee$ is the image of $\srts_0^\vee\setminus(\srts_0^P)^\vee$ under the projection
$\aaa_0\rightarrow\aaa_P$. We continue to denote by $\alpha\mapsto\alpha^\vee$ the ensuing bijection
$\srts_P\rightarrow\srts_P^\vee$.

More generally, we denote by $\Phi_P\subset X^*(S_M)\subset\aaa_P^*$ \index{Phi@$\Phi_P$} the set (containing $\srts_P$)
of indivisible roots of $S_M$ on the Lie algebra of $U$.
For any $\alpha\in\Phi_P$ denote by $\alpha^\vee\in\aaa_P$ the corresponding coroot (\cite{MR1361168}*{I.1.11}).

The following result is well known. For convenience we include a proof.
\begin{lemma} \label{lem: alpalp}
Let $P\in\stds$ and $\alpha\in\srts_P$. Then,
\begin{enumerate}
\item All the coefficients of $\alpha$ with the respect to the basis $\srts_0$ are non-negative.
\item $\sprod{\alpha}{\alpha^\vee}>0$.
\end{enumerate}
\end{lemma}

\begin{proof}
Let $\beta\in\srts_0\setminus\srts_0^P$ be the simple root that projects to $\alpha$.
Write $\alpha=\beta+\gamma$ where $\gamma\in(\aaa_0^P)^*$. Then,
\[
\sprod{\gamma}{\alpha_0^\vee}=-\sprod{\beta}{\alpha_0^\vee}\ge0
\]
for any $\alpha_0\in\srts_0^P$. Thus, $\gamma$ is in the closed positive Weyl chamber of $(\aaa_0^P)^*$
and hence it has non-negative coefficients with respect to the basis $\srts_0^P$.
This proves the first part.
For the second part, upon replacing $G$ by the Levi subgroup of $Q$ where $\srts_0^Q=\srts_0^P\cup\{\beta\}$,
we may assume without loss of generality that $P$ is maximal.
Note that $\sprod{\alpha}{\alpha^\vee}=\sprod{\alpha}{\beta^\vee}$.
Assume on the contrary that $\sprod{\alpha}{\beta^\vee}\le0$. Since $\sprod{\alpha}{\alpha_0^\vee}=0$
for all $\alpha_0\in\srts_0^P$ and $P$ is maximal, we would conclude that $\alpha$ is in the closed negative Weyl chamber,
and in particular its coefficients with respect to $\srts_0$ are non-positive.
This contradicts the fact that the $\beta$-coefficient of $\alpha$ is $1$.
\end{proof}

We extend the homomorphism $\Ht_M:M(\A)\rightarrow\aaa_M$ to a left-$U(\A)$ right-$\K$-invariant function $\Ht_P$ on $G(\A)$.
Thus, $\Ht_P=\Ht_M\circ\m_P$.
We also write $\Ht_0=\Ht_{P_0}$.  \index{HP@$\Ht_P$} \index{H0@$\Ht_0$}

\subsection{Cuspidal exponents} \label{sec: cuspexp}

%Recall that $S_0$ is the maximal $F$-split torus in the center of $M_0$.

In the number field case, let $\AAA_0=S_0(\R)^0$ be the connected component of the identity (in the real topology)
of $S_0(\R)$ viewed as a subgroup of $S_0(\A)$ by embedding $\R$ in $\A_F$ via
$\R\hookrightarrow\A_{\Q}\hookrightarrow\A_{\Q}\otimes F=\A_F$.
In the function field case, fix once and for all a place $v_0$ of $F$ and a uniformizer $\varpi$ of $F_{v_0}$
and let $\AAA_0$ be the image of $(\varpi^{\Z})^d$ in $S_0(F_{v_0})$
where we identify $S_0$ with $\mathbb{G}_m^d$ (over $F$) and $d=\dim S_0$ (\cite{MR1361168}*{I.2.1}).
Then, $\AAA_0$ is a $\Weyl$-invariant lattice of rank $d$.

Now let $P\in\stds$.
We set $\AAA_P=\AAA_0\cap S_M(\A)$ and $\AAA_P^Q=\AAA_P\cap M_Q^{\der}(\A)=\AAA_P\cap M_Q(\A)^1$ for any $Q\supset P$. \index{AAP@$\AAA_P$}
Let $\charsA_P$ be the group of quasi-characters of $\AAA_P$. \index{XPtilde@$\charsA_P$}
(We will use additive notation for this group.)
The restriction to $\AAA_P$ defines a surjective map
\begin{equation} \label{eq: restA}
\chars_P\rightarrow\charsA_P.
\end{equation}
In the number field case, this map is an isomorphism and we have
\[
M(\A)=\AAA_P\times M(\A)^1,
\]
i.e, the restriction $\Ht_P\rest_{\AAA_P}:\AAA_P\rightarrow\aaa_P$ is an isomorphism of topological groups.
In the function field case, the subgroup $\Ht_P(\AAA_P)\subset\Ht_P(M(\A))$ is of finite index, hence also a lattice in $\aaa_P$
and the map \eqref{eq: restA} is an algebraic covering map.

In both cases the map $\Re:\chars_P\rightarrow\aaa_P^*$ factors through \eqref{eq: restA},
so that $\Re:\charsA_P\rightarrow\aaa_P^*$ is well defined.

For $P\subset Q$ we denote by $\lambda_P$ (resp., $\lambda_P^Q$) the image of $\lambda\in\charsA_0$ under the restriction map
$\charsA_0\rightarrow\charsA_P$ (resp., $\charsA_0\rightarrow\charsA_P^Q$).
This is consistent with the previous notation since $\Re\lambda_P=(\Re\lambda)_P$.

The space $\AF_P$ admits a left action by $\AAA_P$.
It is advantageous however to consider the twisted action given by \index{acdotphi@$a\cdot\phi$}
\begin{equation} \label{eq: ta}
a\cdot\phi(g)=\modulus_P(a)^{-\frac12}\phi(ag).
\end{equation}
We decompose $\AF_P$ according to this action (\cite{MR1361168}*{I.3.2}). Namely, we write
\begin{equation} \label{eq: AFd}
\AF_P=\dsum_{\lambda\in\charsA_P}\AF_{P,\lambda}
\end{equation}
where $\AF_{P,\lambda}$ \index{AFP@$\AF_{P,\lambda}$} is the $\lambda$-generalized eigenspace of $\AF_P$ with respect to
the twisted $\AAA_P$-action.
Thus, for every $\varphi\in\AF_{P,\lambda}$ there exists $n\ge0$ such that for every $g\in G(\A)$ the function
$a\in\AAA_P\mapsto (a\cdot\varphi)(g)a^{-\lambda}$ is a polynomial in
$\Ht_P(a)\in\aaa_P$ of degree $\le n$.

The determinant of the adjoint representation of $P$ on its Lie algebra
is an element of $X^*(P)$, which we write as $2\rho_P$ where $\rho_P\in\aaa_P^*$. \index{rho@$\rho_P$}

One of the reasons to the above normalization is that if $\phi\in\AF_{P,\mu}$ where $\mu\in\charsA_P$ then for any $w\in\Weyl(P,Q)$
we have
\begin{equation} \label{eq: expMw}
M(w,\lambda)\phi\in\AF_{Q,w\mu}
\end{equation}
whenever $M(w,\lambda)\phi$ is well-defined.
This follows from the fact that for any $m\in M(\A)$ and $g\in G(\A)$ we have
\begin{equation} \label{eq: monm}
\begin{gathered}
\modulus_Q(wmw^{-1})^{-\frac12}(M(w,\lambda)\varphi)(wmw^{-1}g)=\\
\m_Q(g)^{-w\lambda}\int_{(U_Q\cap wU_Pw^{-1})(\A)\bs U_Q(\A)}
\modulus_P(m)^{-\frac12}\varphi(mw^{-1}ug)\m_P(w^{-1}ug)^{\lambda}\ du,
\end{gathered}
\end{equation}
since the modulus character of $M_Q$ on $(wU_Pw^{-1}\cap U_Q)(\A)\bs U_Q(\A)$ corresponds to $\rho_Q-w\rho_P\in\aaa_Q^*$.
(This can be seen by decomposing $w$ as a product of elementary symmetries \cite{MR1361168}*{I.1.8, II.1.6}.)

For any $Q\in\stds$ with $Q\subset P$, the constant term \index{cnstPphi@$\cnst_{P,Q}\phi$}
\[
\cnst_{P,Q}\phi(g)=\int_{U_Q(F)\bs U_Q(\A)}\phi(ug)\ du
\]
defines a linear map
\[
\AF_P\rightarrow\AF_Q.
\]
We have $\cnst_{P,Q}(\cnst_{G,P}\phi)=\cnst_{G,Q}(\phi)$ for any $\phi\in\AF_G$.
For consistency, it will be useful to set $\cnst_{P,Q}\phi=0$ if $Q\not\subset P$.

We denote by $\AF_P^{\cusp}$ \index{AFPcusp@$\AF_P^{\cusp}$} the cuspidal part of $\AF_P$, i.e., the space of $\phi\in\AF_P$
such that $\cnst_{P,Q}\phi=0$ for all $Q\subsetneq P$. We have a decomposition
\begin{equation} \label{eq: AFd2}
\AF_P^{\cusp}=\dsum_{\lambda\in\charsA_P}\AF_{P,\lambda}^{\cusp}
\end{equation}
where $\AF_{P,\lambda}^{\cusp}=\AF_{P,\lambda}\cap \AF_P^{\cusp}$.
We also have a linear projection (see \cite{MR1361168}*{I.3.5}) \index{phicusp@$\phi^{\cusp}$}
\[
\phi\in\AF_P\mapsto\phi^{\cusp}\in\AF_P^{\cusp}
\]
(the \emph{cuspidal projection} of $\phi$) characterized by the equalities
\[
(\phi^{\cusp},\psi)_{\autspace_P}=(\phi,\psi)_{\autspace_P}
\]
for any function $\psi$ on $\autspace_P$ (necessarily rapidly decreasing) of the form $(f\circ\Ht_P) \cdot\psi'$
where $f\in C_c^\infty(\aaa_P)$ and $\psi'\in\AF_P^{\cusp}$. Here,
\begin{equation} \label{eq: innerP}
(f_1,f_2)_{\autspace_P}=\int_{\autspace_P}f_1(g)\overline{f_2(g)}\ dg
\end{equation}
whenever the integral is absolutely convergent. Equivalently, for any cusp form $\psi$ on $\autspace_M=M(F)\bs M(\A)$
we have
\begin{equation} \label{eq: altc}
(\phi(\cdot g),\psi)_{\autspace_M^1}=(\phi^{\cusp}(\cdot g),\psi)_{\autspace_M^1},\ \ g\in G(\A)
\end{equation}
where $\autspace_M^1=M(F)\bs M(\A)^1$ and
\[
(f_1,f_2)_{\autspace_M^1}=\int_{\autspace_M^1}f_1(m)\overline{f_2(m)}\ dm
\]
whenever the integral is absolutely convergent.

Let $\varphi\in\AF_P$ and $w\in\Weyl(P,Q)$.
Then, for any $\lambda\in\chars_P$ in the range of absolute convergence of the intertwining operator $M(w,\lambda)\varphi$ we have
\begin{equation} \label{eq: commcusp}
(M(w,\lambda)\varphi)^{\cusp}=M(w,\lambda)(\varphi^{\cusp}).
\end{equation}
The statement reduces to two sub-statements:
\begin{enumerate}
\item If $\varphi\in\AF_P^{\cusp}$ then $M(w,\lambda)\varphi\in\AF_Q^{\cusp}$.
\item If $\varphi^{\cusp}\equiv0$ then $(M(w,\lambda)\varphi)^{\cusp}\equiv0$.
\end{enumerate}
Both parts easily follow from \eqref{eq: monm} and \eqref{eq: altc}.

For any $\phi\in\AF_P$ and $Q\subset P$ let \index{cnstQphicusp@$\cnst_{P,Q}^{\cusp}\phi$}
$\cnst_{P,Q}^{\cusp}\phi:=(\cnst_{P,Q}\phi)^{\cusp}\in\AF_Q^{\cusp}$ be the \emph{cuspidal component of $\phi$ along $Q$}.
Thus, we get a linear map \index{totcsup_P@$\totcusp_P$}
\[
\totcusp_P:\AF_P\rightarrow\oplus_{Q\in\stds}\AF_Q^{\cusp},\ \ \ \phi\mapsto(\cnst_{P,Q}^{\cusp}\phi)_{Q\in\stds}
\]
where by convention $\cnst_{P,Q}^{\cusp}\phi=0$ unless $Q\subset P$.

We denote by $\Exp_Q^{\cusp}(\phi)\subset\charsA_Q$ \index{Exp@$\Exp_Q^{\cusp}(\phi)$, $\Exp^{\cusp}(\phi)$} the (finite) set of
\emph{cuspidal exponents of $\phi$ along $Q$}.
Thus, $\lambda\in\Exp_Q^{\cusp}(\phi)$ if and only if $\cnst_{P,Q}^{\cusp}\phi$ has a non-zero
$\lambda$-coordinate with respect to the decomposition \eqref{eq: AFd2}.
We write
\[
\Exp^{\cusp}(\phi)=\{(Q,\lambda):Q\in\stds,\ Q\subset P,\ \lambda\in\Exp_Q^{\cusp}(\phi)\}.
\]
The following basic fact is due to Langlands -- see \cite{MR1361168}*{I.3.4}.
\begin{proposition} \label{prop: cuspne0}
Let $\phi\in\AF_G$ be non-zero. Then, $\Exp^{\cusp}(\phi)\ne\emptyset$. In other words, the map $\totcusp_G$ is injective.
\end{proposition}
In fact, the set $\Exp^{\cusp}(\phi)$ determines the growth of $\phi$ (\cite{MR1361168}*{I.4.1}).

\subsection{Leading cuspidal components} \label{sec: leading}

For the time being (until and including \S\ref{sec: prfl}) it will be convenient to use the non-twisted action
of $\AAA_P$ on $\AF_P$. Correspondingly, we write \index{AFPlamun@$\AF_{P,\lambda}^{\un}$, $\AF_{P,\lambda}^{\cusp,\un}$}
\begin{equation} \label{eq: decun}
\AF_P=\dsum_{\lambda\in\charsA_P}\AF_{P,\lambda}^{\un},\ \ \
\AF_P^{\cusp}=\dsum_{\lambda\in\charsA_P}\AF_{P,\lambda}^{\cusp,\un}
\end{equation}
where $\AF_{P,\lambda}=\AF_{P,\lambda+\rho_P}^{\un}$ and $\AF_{P,\lambda}^{\cusp}=\AF_{P,\lambda+\rho_P}^{\cusp,\un}$.
(The superscript stands for ``unnormalized''.)

Let \index{aaa0+@$\aaa_{0,+}^*$}
\[
\aaa_{0,+}^*=\{\lambda\in\aaa_0^*\mid\sprod{\lambda}{\alpha^\vee}\ge0\text{ for all }\alpha\in\srts_0\}
\]
be the closed positive Weyl chamber.
We say that $\lambda\in\charsA_P$ is \emph{leading} if $\Re\lambda\in\aaa_{0,+}^*$.

We denote by $\AF_P^{\cusp,\dom}$ the direct sum \index{AFPdom@$\AF_P^{\cusp,\dom}$}
\[
\AF_P^{\cusp,\dom}=\oplus\AF_{P,\lambda}^{\cusp,\un}
\]
over the leading $\lambda\in\charsA_P$ and by \index{projPdom@$p_P^{\dom}$}
\[
p_P^{\dom}:\AF_P^{\cusp}\rightarrow\AF_P^{\cusp,\dom}
\]
the projection according to \eqref{eq: decun}.

For any $\phi\in\AF_G$ we define $\Exp_P^{\cusp,\un}(\phi)$ with respect to the non-twisted action so that
$\Exp_P^{\cusp,\un}(\phi)$ is the translate of $\Exp_P^{\cusp}(\phi)$ by $\rho_P$. \index{ExpPun@$\Exp_P^{\cusp,\un}(\phi)$}
We define the \emph{leading (unnormalized) cuspidal exponents} of $\phi$ along $P$
to be the elements of $\Exp_P^{\cusp,\un}(\phi)$ that are leading and denote it by
$\Exp_P^{\cusp,\un,\dom}(\phi)$.
We write \index{Expcuspdom@$\Exp_P^{\cusp,\un,\dom}(\phi)$, $\Exp^{\cusp,\un,\dom}(\phi)$}
\[
\Exp^{\cusp,\un,\dom}(\phi)=\{(P,\lambda)\mid P\in\stds,\lambda\in\Exp_P^{\cusp,\un,\dom}(\phi)\}.
\]
We define the \emph{leading cuspidal component} of $\phi\in\AF_G$ along $P$ to be $p_P^{\dom}(\cnst_{G,P}^{\cusp}\phi)$.
Thus, we get a linear map \index{lead@$\lead$} \index{leadP@$\lead_P$}
\begin{subequations}
\begin{equation} \label{def: lead}
\lead:\AF_G\rightarrow\oplus_{P\in\stds}\AF_P^{\cusp,\dom}
\end{equation}
which is the composition of $\totcusp_G$ with $\oplus_{P\in\stds}p_P^{\dom}$.
More generally, for any $P\in\stds$ let
\begin{equation} \label{def: leadP}
\lead_P:\AF_P\rightarrow\oplus_{Q\in\stds}\AF_Q^{\cusp,\dom}
\end{equation}
be the composition of $\totcusp_P$ with $\oplus_{Q\in\stds}p_Q^{\dom}$.
\end{subequations}

The following result is an extension of Proposition \ref{prop: cuspne0}, which will be proved in \S\ref{sec: prfl} below.
\begin{theorem} \label{thm: lead}
Suppose that $\phi\in\AF_G$ is non-zero. Then, $\Exp^{\cusp,\dom}(\phi)\ne\emptyset$.
In other words, the map $\lead$ is injective.
\end{theorem}

\subsection{}
We first prove the following special case of Theorem \ref{thm: lead}.
\begin{lemma} \label{lem: nonexist}
Let $\phi\in\AF_G$. Assume that for every $(P,\lambda)\in\Exp^{\cusp,\un}(\phi)$ the following two properties are satisfied.
\begin{itemize}
\item $P$ is a maximal (proper) parabolic subgroup of $G$, i.e. $\srts_P$ is a singleton.
\item Writing $\srts_P=\{\alpha\}$ we have $\sprod{\Re\lambda}{\alpha^\vee}<0$.
\end{itemize}
Then, $\phi=0$.
\end{lemma}

\begin{proof}
In the proof we will use the following elementary fact. Suppose that $f$ is a function on $\R$ (resp., $\Z$) of the form
\[
f(x)=\sum_{i=1}^ne^{\lambda_ix}P_i(x)
\]
where $P_1,\dots,P_n$ are non-zero polynomials and $\lambda_1,\dots,\lambda_n$ are distinct elements of $\C$ (resp., $\C/2\pi\iii\Z$).
Assume that $f$ is bounded on $\R$ (resp., $\Z$). Then, $\Re\lambda_i=0$ for all $i$.
Indeed, if $h$ (in $\R$ or $\Z$) is such that $e^{\lambda_ih}$ are distinct, then by applying the difference operators
$f\mapsto f(\cdot+h)-e^{\lambda_ih}f$ ($\deg P_i+1$ times for $i>1$ and $\deg P_1$
times for $i=1$) we may assume that $n=1$ and $\deg P_1=0$, in which case the statement is clear.

Back to the statement of the lemma, it easily follows from \cite{MR1361168}*{I.4.1} and the assumptions on $\phi$
that any right translate of $\phi$ is bounded on $G(\A)^1$. (Cf.\ the argument in \cite{MR1361168}*{I.4.11}.)
This implies that for any $P\in\stds$, any right translate of the constant term $\cnst_{G,P}\phi$ is also bounded on $G(\A)^1$.
Note that by Proposition \ref{prop: cuspne0} and our assumption, $\cnst_{G,P}\phi\equiv0$ for any non-maximal proper $P\in\stds$.
If $P$ is maximal, then $\cnst_{G,P}\phi$ is cuspidal and for any $g\in G(\A)$ the function $a\in\AAA_P^G\mapsto\cnst_{G,P}\phi(ag)$
is a polynomial exponential function in $\Ht_P(a)\in\aaa_P^G$ with exponents $\lambda^G$,
$\lambda\in\Exp_P^{\cusp,\un}(\phi)$.
Since by assumption $\Re\lambda^G\ne0$ for every $\lambda\in\Exp_P^{\cusp,\un}(\phi)$,
such a function, if bounded, must be identically $0$ by the above. Hence $\cnst_{G,P}\phi\equiv0$.
It follows that $\phi$ is cuspidal.
Since by assumption $\Exp_G^{\cusp}(\phi)=\emptyset$ we conclude that $\phi=0$.
\end{proof}

Let $\phi\in\AF_G$ be non-zero.
We say that a parabolic subgroup $P\in\stds$ is \emph{minimal with respect to $\phi$}
if $\Exp_P^{\cusp}(\phi)\ne\emptyset$ but $\Exp_Q^{\cusp}(\phi)=\emptyset$ for any $Q\in\stds$
of smaller semisimple rank than $P$.
Clearly, such $P$ exists by Proposition \ref{prop: cuspne0}.

For any $P\in\stds$ and $\alpha\in\srts_P$ let $P_\alpha\in\stds$ \index{Paralpha@$P_\alpha$}
\index{salpha@$s_\alpha$} \index{ParalphaP'@$P\xrightarrow\alpha P'$} be such that
$\srts_0^{P_\alpha}=\srts_0^P\cup\{\beta\}$ where $\beta\in\srts_0\setminus\srts_P$
is the unique simple root that projects to $\alpha$. Thus, $P$ is a maximal parabolic subgroup
of $P_\alpha$.

\begin{corollary} \label{cor: uniquemax}
Let $\phi\in\AF_G$, $(P,\lambda)\in\Exp^{\cusp,\un}(\phi)$ and $\alpha\in\srts_P$.
Let $P_\alpha$ be as above.
Assume that $P$ is minimal with respect to $\phi$ and $\sprod{\Re\lambda}{\alpha^\vee}<0$.
Then, there exists $(Q,\mu)\in\Exp^{\cusp,\un}(\phi)$ (possibly with $Q=P$) with the following properties.
\begin{itemize}
\item $Q$ is a maximal parabolic subgroup of $P_\alpha$.
\item $Q$ is minimal with respect to $\phi$.
\item $\mu_{P_\alpha}=\lambda_{P_\alpha}$.
\item $\sprod{\Re\mu}{\beta^\vee}\ge0$ where $\srts_Q^{P_\alpha}=\{\beta\}$.
\end{itemize}
\end{corollary}

\begin{proof}
We first remark that by the minimality of $P$, if $(Q,\mu)\in\Exp^{\cusp}(\phi)$ and $Q\subsetneq P_\alpha$ for some
$\alpha\in\srts_P$, then $Q$ is a maximal parabolic subgroup of $P_\alpha$ and hence, $Q$ is minimal with respect to $\phi$
(since $P$ and $Q$ have the same semisimple rank).

For $k\in\K$ let $\phi'_k=\cnst_{G,P_\alpha}\phi(\cdot k)$ which is an automorphic form on $\autspace_{M_{P_\alpha}}$.
There exists $k$ such that $(P\cap M_{P_\alpha},\lambda)\in\Exp^{\cusp,\un}(\phi'_k)$.
Therefore, by passing to $\phi'_k$ and replacing $G$ by $M_{P_\alpha}$
and $P$ by $P\cap M_{P_\alpha}$, we reduce the corollary to the case that $P_\alpha=G$, i.e., $P$ is maximal.
Upon subtracting the cuspidal projection $\phi^{\cusp}$ of $\phi$, we may also assume without loss of generality that $\Exp_G^{\cusp}(\phi)=0$.
Finally, by decomposing $\phi$ according to the action of $\AAA_G$ \eqref{eq: decun} we can assume that
$\mu_G=\lambda_G$ for all $(Q,\mu)\in\Exp^{\cusp,\un}(\phi)$.
Under these assumptions, the first three conditions hold automatically for any $(Q,\mu)\in\Exp^{\cusp,\un}(\phi)$.
Therefore, the corollary follows from the lemma above.
\end{proof}

\subsection{Proof of Theorem \ref{thm: lead}}  \label{sec: prfl}
Let $0\not\equiv\phi\in\AF_G$.
Fix $\varpi^\vee\in\aaa_0$ such that $\sprod\alpha{\varpi^\vee}>0$ for all $\alpha\in\srts_0$.
Let $(P,\lambda)\in\Exp^{\cusp,\un}(\phi)$ be such that $P$ is minimal with respect to $\phi$ and
$\sprod{\Re\lambda}{\varpi^\vee}$ is maximal.
We claim that $\lambda$ is leading, i.e., $\Re\lambda\in\aaa_{0,+}^*$. Assume on the contrary that this is not the case.
Then, there exists $\alpha\in\srts_P$ such that $\sprod{\Re\lambda}{\alpha^\vee}<0$.
Let $(Q,\mu)\in\Exp^{\cusp,\un}(\phi)$ be as in Corollary \ref{cor: uniquemax} and write $\srts_Q^{P_\alpha}=\{\beta\}$.
Note that $(\varpi^\vee)_P^{P_\alpha}$ is a positive multiple of
$\alpha^\vee$ since they are proportional and by Lemma \ref{lem: alpalp} $\sprod{\alpha}{(\varpi^\vee)_P^{P_\alpha}}=\sprod{\alpha}{\varpi^\vee}>0$
and $\sprod{\alpha}{\alpha^\vee}>0$. Similarly, $(\varpi^\vee)_Q^{P_\alpha}$ is a positive multiple of $\beta^\vee$.
Thus,
\[
(\Re\lambda)_{P_\alpha}=(\Re\mu)_{P_\alpha}
\]
and
\[
\sprod{(\Re\lambda)^{P_\alpha}}{\varpi^\vee}=\sprod{\Re\lambda}{(\varpi^\vee)_P^{P_\alpha}}<0\le
\sprod{\Re\mu}{(\varpi^\vee)_Q^{P_\alpha}}=\sprod{(\Re\mu)^{P_\alpha}}{\varpi^\vee}.
\]
Hence,
\begin{align*}
\sprod{\Re\lambda}{\varpi^\vee}=&\sprod{(\Re\lambda)_{P_\alpha}}{\varpi^\vee}+\sprod{(\Re\lambda)^{P_\alpha}}{\varpi^\vee}<\\
&\sprod{(\Re\mu)_{P_\alpha}}{\varpi^\vee}+\sprod{(\Re\mu)^{P_\alpha}}{\varpi^\vee}=\sprod{\Re\mu}{\varpi^\vee},
\end{align*}
gainsaying the maximality of $\sprod{\Re\lambda}{\varpi^\vee}$. \hfill $\qed$

\begin{remark}
Using the coarse spectral decomposition for automorphic forms (\cite{MR1361168}*{III})
we can get additional information on the set $\Exp^{\cusp}(\phi)$ of an automorphic form $\phi\in\AF_G$ as follows.
(We will not use this result in the sequel.)
\end{remark}

\begin{lemma} \label{lem: uniquemax}
Let $\phi\in\AF_G$, $(P,\lambda)\in\Exp^{\cusp}(\phi)$ and $\alpha\in\srts_P$.
Assume that $\sprod{\Re\lambda+\rho_P}{\alpha^\vee}<0$.
Let $s_\alpha$ be the elementary symmetry corresponding to $\alpha$ (\cite{MR1361168}*{I.1.7}). Thus,
$P_\alpha$ is generated by $P$ and $s_\alpha$, and $s_\alpha\in\Weyl(P,P')$ where
$P'$ is a maximal parabolic subgroup of $P_\alpha$. Then, $s_\alpha\lambda\in\Exp_{P'}^{\cusp}(\phi)$.
\end{lemma}

To prove the lemma, we first recall the coarse spectral decomposition for automorphic forms.

Consider the equivalence relation on pairs $(P,\lambda)$, $P\in\stds$, $\lambda\in\charsA_P$
given by $(P,\lambda)\sim(P',\lambda')$ if there exists $w\in\Weyl(P,P')$ such that
$w\lambda=\lambda'$.
For any $\sim$-equivalence class $\theta$ let \index{AFtheta@$\AF_\theta$}
\[
\AF_\theta=\{\phi\in\AF_G\mid\Exp^{\cusp}(\phi)\subset\theta\}.
\]
We have a direct sum decomposition
\begin{equation} \label{eq: decomp}
\AF_G=\dsum_\theta\AF_\theta
\end{equation}
where $\theta$ ranges over the equivalence classes of pairs $(P,\lambda)$, $\lambda\in\charsA_P$
(\cite{MR1361168}*{III.2.6}). (In fact, the decomposition in [loc.\ cit.]\ is more refined, but for our purposes
\eqref{eq: decomp} is enough.)
Thus, if $\phi\in\AF_G$ and $\phi=\sum_\theta\phi_\theta$ is the corresponding decomposition, then
$\phi_\theta\ne0$ if and only if $\lambda\in\Exp_P^{\cusp}(\phi)$ for some $(P,\lambda)\in\theta$.

By considering $\cnst_{G,P_\alpha}\phi$, we reduce the lemma to the case that $P$ and $P'$ are maximal.
By the decomposition (\ref{eq: decomp}) we can also assume that
$\Exp^{\cusp}(\phi)\subset\{(P,\lambda),(P',s_\alpha\lambda)\}$.
In this case, the lemma follows from Lemma \ref{lem: nonexist}. \hfill $\qed$

\subsection{Polynomial exponential functions} \label{sec: polyexp}

The restriction of any automorphic form in $\AF_P$ to $\AAA_P$ is a simple function, namely a polynomial exponential.
It will be convenient to set some notation pertaining to this type of functions.

Let $V$ be either a finite-dimensional real vector space or a lattice thereof.
Denote by $\quas{V}$ the group of continuous quasi-characters of $V$.
We will write $v^\lambda$ for the value of $\lambda\in\quas{V}$ on $v\in V$.
For any tuple $\underline{\lambda}=(\lambda_1,\dots,\lambda_n)\in(\quas{V})^n$ we denote by \index{PE@$\PE_V(\underline{\lambda})$}
\[
\PE_V(\underline{\lambda})
\]
the finite-dimensional linear space consisting of the polynomial exponential functions on $V$ with exponents in
$\underline{\lambda}$, such that the degree of the polynomial pertaining to an exponent $\mu$ is
smaller than $\#\{i:\lambda_i=\mu\}$.
Equivalently, for any $\lambda\in\quas{V}$ and $v\in V$ let $D_v^\lambda$ be the difference operator \index{Diff1@$D_v^\lambda$}
\[
D_v^\lambda f(u)=f(u+v)-v^\lambda f(u),\ \ u\in V
\]
on functions on $V$. Then,
\[
\PE_V(\underline{\lambda})\text{ is the space of functions $f$ on $V$ such that }
D_{v_1}^{\lambda_1}\dots D_{v_n}^{\lambda_n}f\equiv0\ \ \forall v_1,\dots,v_n\in V.
\]
Of course, $\PE_V(\underline{\lambda})$ depends only on $\underline{\lambda}$ up to permutation of coordinates.

Note that $\dim\PE_V(\underline{\lambda})\ge n$ with equality if and only if
$\lambda_1,\dots,\lambda_n$ are distinct or $\dim V=1$.
In particular, if $\dim V>1$, then the family $\PE_V(\underline{\lambda})$ is not flat in the parameter $\underline{\lambda}$.

If $f_1\in\PE_V(\underline{\lambda})$ and $f_2\in\PE_V(\underline{\mu})$, then $f_1+f_2\in\PE_V(\underline{\lambda}\vee\underline{\mu})$
where $\underline{\lambda}\vee\underline{\mu}$ \index{lambda@$\underline{\lambda}\vee\underline{\mu}$}
denotes the concatenation of $\underline{\lambda}$ and $\underline{\mu}$.

Given $P\in\stds$ and $\underline{\lambda}\in\charsA_P^n$ we may define similarly the space
\index{PElambda@$\PE_{\AAA_P}(\underline{\lambda})$}
\begin{equation} \label{def: PEAP}
\PE_{\AAA_P}(\underline{\lambda})=\{f\circ\Ht_P\rest_{\AAA_P}:f\in\PE_{\Ht_P(\AAA_P)}(\underline{\lambda})\}
\end{equation}
of polynomial exponential functions on $\AAA_P$.

Slightly more generally, if $Q\subset P$ then for any $\underline{\lambda}\in\charsA_Q^n$ we write
$\PE_{\AAA_P}(\underline{\lambda})=\PE_{\AAA_P}(\underline{\lambda'})$
where $\underline{\lambda'}\in\charsA_P^n$ is obtained from $\underline{\lambda}$
by projecting each coordinate to $\charsA_P$.

Finally, we write \index{AFPlambda@$\AF_P(\underline{\lambda})$, $\AF_P^{\cusp}(\underline{\lambda})$}
\begin{equation} \label{eq: defpe}
\AF_P(\underline{\lambda})=\{\phi\in\AF_P\mid a\mapsto (a\cdot\phi)(g)\in\PE_{\AAA_P}(\underline{\lambda})
\ \forall g\in G(\A)\}
\end{equation}
and
\[
\AF_P^{\cusp}(\underline{\lambda})=\AF_P(\underline{\lambda})\cap\AF_P^{\cusp}.
\]
\label{sec: muP}

Any $\phi\in\AF_P$ belongs to $\AF_P(\underline{\mu})$ for some integer $m\ge0$ and $\underline{\mu}\in\charsA_P^m$.
Moreover, if $m$ is the minimal such integer, then $\underline{\mu}$ is unique up to permutation.
In particular, we can take $\underline{\mu}$ whose coordinates (as a set) index the nonzero coordinates in the decomposition
of $\varphi$ with respect to \eqref{eq: AFd}.
Thus, for any $\phi\in\AF_G$ we have $\cnst_{G,P}^{\cusp}\phi\in \AF_P(\underline{\mu})$ where the set of coordinates
of $\underline{\mu}$ is $\Exp_P^{\cusp}(\phi)$.

%Although we shall not use this fact directly, we remark that by the argument of \cite{MR1361168}*{p. 50} we have
%\begin{lemma} \label{lem: p50}
\begin{remark}
Let $\phi\in\AF_P$.
For any $Q\subset P$ let $m_Q\ge0$ be an integer and $\underline{\mu}_Q\in\charsA_Q^{m_Q}$ be such that
$\cnst_{P,Q}^{\cusp}\phi\in\AF_P^{\cusp}(\underline{\mu}_Q)$.
Then, for any $Q\subset P$ we have
\[
\cnst_{P,Q}\phi\in\AF_Q(\vee_{Q'\subset Q}\underline{\mu}_{Q'}),
\]
where $\vee_{Q'\subset Q}\underline{\mu}_{Q'}$ means the concatenation of $\underline{\mu}_{Q'}$,
$Q'\subset Q$ in arbitrary order \cite{MR1361168}*{p. 50}.
(See also the argument in Lemma \ref{lem: autcusp} below.)
\end{remark}

We note that if $\phi\in\AF_P(\underline{\mu})$ where $\underline\mu\in\charsA_P^n$ then for any $w\in\Weyl(P,Q)$ we have
\begin{equation} \label{eq: Mwexp}
M(w,\lambda)\phi\in\AF_Q(w\underline{\mu})
\end{equation}
(cf.\ \eqref{eq: expMw}).

\subsection{Weyl group double cosets}
Denote by $\Weyl^P$ \index{WeylP@$\Weyl^P$, $_Q\Weyl_P$} the Weyl group of $M_P$, viewed as a subgroup of $\Weyl$.
For any $P,Q\in\stds$ let
\[
_Q\Weyl_P=\{w\in\Weyl\mid w\alpha>0\ \forall\alpha\in\srts_0^P\text{ and }w^{-1}\alpha>0\ \forall\alpha\in\srts_0^Q\}.
\]

The following is standard.

\begin{lemma}[Bruhat Decomposition] \label{lem: bruhat} \
\begin{enumerate}
\item The set $\,_Q\Weyl_P$ is a set of representatives for the double cosets $\Weyl^Q\bs\Weyl/\Weyl^P$,
as well as for $Q(F)\bs G(F)/P(F)$.
\item For any $w\in\,_Q\Weyl_P$, the group $M_P\cap w^{-1}M_Qw$ is the Levi subgroup
of a standard parabolic subgroup $P_w$ of $P$; \index{Parw@$P_w$, $Q_w$} likewise,
$M_Q\cap wM_Pw^{-1}$ is the Levi subgroup of a standard parabolic subgroup $Q_w$ of $Q$.
\item We have
\[
Q\cap wPw^{-1}\subset Q_w,\ \ U_Q\cap wPw^{-1}=U_Q\cap wU_{P_w}w^{-1}
\]
and
\[
U_{Q_w}=(M_Q\cap wU_Pw^{-1})\ltimes U_Q=(U_{Q_w}\cap wU_{P_w}w^{-1})U_Q.
\]
\end{enumerate}
\end{lemma}

Let \index{WeylPQ@$\Weyl^{\supset Q}(P)$}
\[
\Weyl^{\supset Q}(P)=\{w\in\Weyl\mid wM_Pw^{-1}\supset M_Q\text{ and }w\alpha>0\ \forall\alpha\in\srts_0^P\}=
\{w\in\,_Q\Weyl_P\mid Q_w=Q\}.
\]
Thus, if $w\in\Weyl^{\supset Q}(P)$, then $\Weyl^Qw\Weyl^P=w\Weyl^P$.
We may identify $\Weyl(P,Q)$ with the set
\[
\{w\in\,_Q\Weyl_P\mid wM_Pw^{-1}=M_Q\}=\Weyl^{\supset Q}(P)\cap\Weyl^{\supset P}(Q)^{-1}.
\]
Clearly, if $w\in\,_Q\Weyl_P$, then $w\in\Weyl(P_w,Q_w)$.
In particular, if $w\in\Weyl^{\supset Q}(P)$, then $w\in\Weyl(P_w,Q)$.

If $w\in\Weyl(P,Q)$, then $w$ induces a bijection $\srts_0^P\rightarrow\srts_0^Q$.

\subsection{}
Recall that the interior of $\aaa_{0,+}^*$ is a fundamental domain for the action of the Weyl group
on $\aaa_0^*$. It follows that if $\lambda\in\aaa_{0,+}^*$ is sufficiently regular, then
$w\lambda$ is far away from $\aaa_{0,+}^*$ for any $w\ne e$, and hence there exists $\alpha\in\srts_0$
such that $\sprod{w\lambda}{\alpha^\vee}$ is very negative.
The following is a variant of this basic fact.

\begin{lemma} \label{lem: far}
For any $c>0$ there exists $c'>0$ with the following property.
Let $e\ne w\in\Weyl^{\supset Q}(P)$ and $\lambda\in\aaa_P^*$.
Suppose that $\sprod{\lambda}{\alpha^\vee}>c'$ for all $\alpha\in\srts_P$.
Then, there exists $\gamma\in\srts_Q$ such that $\sprod{w\lambda}{\gamma^\vee}<-c$.\footnote{Note that
$w\lambda\in\aaa_Q^*$ since $w\in\Weyl^{\supset Q}(P)$.}
\end{lemma}

\begin{proof}
Since $w\in\Weyl(P_w,Q)$ and $w\ne e$, there exists $\alpha\in\srts_{P_w}$
such that $\beta:=w\alpha<0$. Moreover, $\alpha\notin\srts_{P_w}^P$ since $w\in\Weyl^{\supset Q}(P)$.
Therefore, $\sprod{w\lambda}{\beta^\vee}=\sprod{\lambda}{\alpha^\vee}>c'$.
Expanding $\beta^\vee$ with respect to $\srts_Q^\vee$ we infer that $\sprod{w\lambda}{\gamma^\vee}<-c$ for some
$\gamma\in\srts_Q$, provided that $c'$ is sufficiently large with respect to $c$.
\end{proof}

\subsection{A uniqueness property of Eisenstein series} \label{sec: uniqEis}

For the rest of the section we fix $\varphi\in\AF_P$ and $\lambda\in\chars_P$ such that
$\sprod{\Re\lambda}{\alpha^\vee}\gg0$ for all $\alpha\in\srts_P$.
(The implied constant depends on $\varphi$.)
Recall the Eisenstein series $\psi=E(\varphi,\lambda)$ defined in \eqref{def: eisen}. Clearly,
\begin{equation} \label{eq: cEis}
\text{if $P\ne G$, then }\Exp_G^{\cusp}(\psi)=\emptyset.
\end{equation}
Let $\tilde\lambda$ be the image of $\lambda$ under the projection $\chars_P\rightarrow\charsA_P$.
Suppose that $\varphi\in\AF_P(\underline{\mu})$ where $\underline{\mu}=(\mu_1,\dots,\mu_n)\in\charsA_P^n$.
Then, $\psi\in\AF_G(\underline{\mu}+\tilde\lambda)$ where $\underline{\mu}+\tilde\lambda=(\mu_1+\tilde\lambda,\dots,\mu_n+\tilde\lambda)$
and we view $\mu_i$ and $\tilde\lambda$ as elements of $\charsA_G$. (See the convention before \eqref{eq: defpe}.)

Using the notion of the leading cuspidal component (\S\ref{sec: leading}),
we can characterize the Eisenstein series (in the range above) as follows.

\begin{proposition} \label{prop: unique}
The leading cuspidal components of $\psi$ coincide with those of $\varphi_\lambda$, i.e.,
\[
\lead(\psi)=\lead_P(\varphi_\lambda)
\]
(see \eqref{def: lead} and \eqref{def: leadP}).
Moreover, by Theorem \ref{thm: lead}, this relation uniquely characterizes $\psi$.
\end{proposition}

The proposition will be proved below. It is a consequence of the computation of the constant term of Eisenstein series in terms of intertwining
operators, which is analogous to the computation of the Jacquet module of induced representations in the local case \cite{MR0579172}.

\subsection{Geometric Lemma}
For any $Q\subset P$ we will consider $\chars_P$ as a subgroup of $\chars_Q$
(by restricting a character of $M_P(\A)/M_P(\A)^1$ to $M_Q(\A)$) \cite{MR1361168}*{I.1.4}.

\begin{lemma} \label{lem: GL}
For any $Q\in\stds$ we have
\begin{equation} \label{eq: EserCT}
\cnst_{G,Q}\psi=\sum_{w\in\,_Q\Weyl_P}E^Q(M(w,\lambda)(\cnst_{P,P_w}\varphi),w\lambda)
\end{equation}
where the superscript indicates that we replace the sum over $P(F)\bs G(F)$ in \eqref{def: eisen}
by the sum over $Q_w(F)\bs Q(F)$.
In particular, by \eqref{eq: cEis} and \eqref{eq: commcusp},
\begin{equation} \label{eq: consterm}
\cnst_{G,Q}^{\cusp}\psi=\sum_{w\in\Weyl^{\supset Q}(P)}[M(w,\lambda)(\cnst_{P,P_w}\varphi)]_{w\lambda}^{\cusp}
=\sum_{w\in\Weyl^{\supset Q}(P)}[M(w,\lambda)(\cnst_{P,P_w}^{\cusp}\varphi)]_{w\lambda}.
\end{equation}
\end{lemma}

Each summand on the right-hand side of \eqref{eq: EserCT} is a composition of three operations:
taking a constant term (from $\AF_P$ to $\AF_{P_w}$), intertwining operator (from $\AF_{P_w}$
to $\AF_{Q_w}$, where we view $w$ as an element of $\Weyl(P_w,Q_w)$) and Eisenstein series (from $\AF_{Q_w}$ to $\AF_Q$).
The last two operations are taken in their range of convergence.

\begin{proof}
The lemma is a straightforward generalization of the computation of \cite{MR1361168}*{II.1.7}.
For completeness we provide a proof. By Bruhat decomposition (Lemma \ref{lem: bruhat}), we can write
\begin{align*}
\psi(g)&=\sum_{w\in\,_Q\Weyl_P}\sum_{\gamma\in (Q\cap wPw^{-1})(F)\bs Q(F)}\varphi_\lambda(w^{-1}\gamma g)
\\&=\sum_{w\in\,_Q\Weyl_P}\sum_{\gamma\in Q_w(F)\bs Q(F)}
\sum_{v\in (U_Q\cap wU_{P_w}w^{-1})(F)\bs U_Q(F)}\varphi_\lambda(w^{-1}v\gamma g).
\end{align*}
Therefore, $\cnst_{G,Q}\psi(g)$ is the sum over $w\in\,_Q\Weyl_P$ and $\gamma\in Q_w(F)\bs Q(F)$ of
\begin{align*}
&\int_{U_Q(F)\bs U_Q(\A)}\sum_{v\in (U_Q\cap wU_{P_w}w^{-1})(F)\bs U_Q(F)}\varphi_\lambda(w^{-1}v\gamma ug)\ du\\=&
\int_{U_Q(F)\bs U_Q(\A)}\sum_{v\in (U_Q\cap wU_{P_w}w^{-1})(F)\bs U_Q(F)}\varphi_\lambda(w^{-1}vu\gamma g)\ du\\=&
\int_{(U_Q\cap wU_{P_w}w^{-1})(F)\bs U_Q(\A)}\varphi_\lambda(w^{-1}u\gamma g)\ du,
\end{align*}
which we write as the integral over $u\in (U_Q\cap wU_{P_w}w^{-1})(\A)\bs U_Q(\A)$ of
\begin{align*}
&\int_{(U_Q\cap wU_{P_w}w^{-1})(F)\bs (U_Q\cap wU_{P_w}w^{-1})(\A)}\varphi_\lambda(w^{-1}vu\gamma g)\ dv
\\=&\int_{(U_{P_w}\cap w^{-1}U_Qw)(F)\bs (U_{P_w}\cap w^{-1}U_Qw)(\A)}\varphi_\lambda(vw^{-1}u\gamma g)\ dv
\\=&\int_{(M\cap U_{P_w})(F)\bs (M\cap U_{P_w})(\A)}\varphi_\lambda(vw^{-1}u\gamma g)\ dv
=\cnst_{P,P_w}\varphi_\lambda(w^{-1}u\gamma g).
\end{align*}
On the other hand, (again by Lemma \ref{lem: bruhat}) for any $\varphi'\in\AF_{P_w}$ and $g\in G(\A)$ we have
\begin{align*}
M(w,\lambda)\varphi'_\lambda(g)&=\int_{(U_{Q_w}\cap wU_{P_w}w^{-1})(\A)\bs U_{Q_w}(\A)}\varphi'_\lambda(w^{-1}ug)\ du
\\&=\int_{(U_Q\cap wU_{P_w}w^{-1})(\A)\bs U_Q(\A)}\varphi'_\lambda(w^{-1}ug)\ du.
\end{align*}
The lemma follows.
\end{proof}

For any $Q\subset P$ and $\underline{\mu}=(\mu_1,\dots,\mu_n)\in\charsA_Q^n$
we write $\underline{\mu}+\tilde\lambda=(\mu_1+\tilde\lambda,\dots,\mu_n+\tilde\lambda)$
where $\tilde\lambda$ is the image of $\lambda$, viewed as an element of $\chars_Q$, under $\chars_Q\rightarrow\charsA_Q$.

By \eqref{eq: Mwexp} we conclude:

\begin{corollary} \label{cor: uniq}
For any $Q\subset P$ let $m_Q\ge0$ (resp., $m'_Q\ge0$) be an integer and
$\underline{\mu}_Q\in\charsA_Q^{m_Q}$ (resp., $\underline{\mu}_Q^{\cusp}\in\charsA_Q^{m'_Q}$)
be such that $\cnst_{P,Q}\varphi\in\AF_Q(\underline{\mu}_Q)$ (resp., $\cnst_{P,Q}^{\cusp}\varphi\in\AF_Q^{\cusp}(\underline{\mu}_Q^{\cusp})$).
Then, for any $Q\in\stds$ we have
\[
\cnst_{G,Q}\psi\in\AF_Q(\vee_{w\in\,_Q\Weyl_P}w(\underline{\mu}_{P_w}+\tilde\lambda))
\]
(see \eqref{eq: defpe})\footnote{The coordinates of $w(\underline{\mu}_{P_w}+\tilde\lambda)$ are in $\charsA_{Q_w}$.
By our convention the space $\AF_Q(\dots)$ depends only on the image of these coordinates under $\charsA_{Q_w}\rightarrow\charsA_Q$.}
and
\[
\cnst_{G,Q}^{\cusp}(\psi)\in\AF_Q^{\cusp}(\vee_{w\in\Weyl^{\supset Q}(P)}w(\underline{\mu}_{P_w}^{\cusp}+\tilde\lambda)).
\]
where $\vee$ denotes concatenation in an arbitrary order. Hence,
\begin{equation} \label{eq: phiqexpc}
\Exp_Q^{\cusp}(\psi)\subset\bigcup_{w\in\Weyl^{\supset Q}(P)}w(\Exp_{P_w}^{\cusp}(\varphi)+\tilde\lambda)
\end{equation}
Moreover, for any $w'\in\Weyl(P,P')$, $P'\in\stds$ we have
\begin{equation} \label{eq: cuspew}
\cnst_{G,Q}^{\cusp}(\psi)-\cnst_{P',Q}^{\cusp}([M(w',\lambda)\varphi]_{w'\lambda})
\in\AF_Q^{\cusp}(\vee_{w\in\Weyl^{\supset Q}(P)\setminus\{w'\}}w(\underline{\mu}_{P_w}^{\cusp}+\tilde\lambda)).
\end{equation}
(Recall that by our convention, the second term on the left-hand side is interpreted as $0$ unless $P'\supset Q$,
i.e., unless $w'\in\Weyl^{\supset Q}(P)$.)
In particular,
\begin{equation} \label{eq: cuspe}
\cnst_{G,Q}^{\cusp}(\psi)-\cnst_{P,Q}^{\cusp}(\varphi_\lambda)\in
\AF_Q^{\cusp}(\vee_{w\in\Weyl^{\supset Q}(P)\setminus\{e\}}w(\underline{\mu}_{P_w}^{\cusp}+\tilde\lambda)).
\end{equation}
\end{corollary}

Proposition \ref{prop: unique} now follows from \eqref{eq: cuspe} and Lemma \ref{lem: far}.
Moreover, under our standing condition on $\lambda$,
\begin{equation} \label{eq: Eunique}
\psi=E(\varphi,\lambda)\text{ is the unique automorphic form satisfying \eqref{eq: cuspe}}.
\end{equation}
It is also true and easy to show that the union on the right-hand side of \eqref{eq: phiqexpc} is disjoint,
although we will not use this fact.

\section{Local finiteness} \label{sec: LF}

In this section we prove a local finiteness result (Theorem \ref{thm: mainfin}).
Throughout this section, $F$ is a number field.
%Thus, $\chars_P=\charsA_P=\aaa_P^*\otimes_{\R}\C$.

By convention, a function between measurable spaces is always implicitly assumed to be measurable.

\subsection{Functions of uniform moderate growth} \label{sec: fumg}

In general, we denote by $\delta(g)$ the right translation by $g\in G(\A)$ on
spaces of functions on homogenous spaces of $G(\A)$.
This notation will also be used for the action of $C_c^\infty(G(\A))$, or of $\univ(\gggg)$,
if appropriate.

Recall that we fixed a minimal parabolic subgroup $P_0$ of $G$ (defined over $F$) and
a maximal compact subgroup $\K$ of $G(\A)$.
Let \index{Z@$\zspace$}
\[
\zspace=P_0(F)\bs G(\A).
\]
The function $\Ht_0:G(\A)\rightarrow\aaa_0$ descends to a proper function
\[
\Ht_0:\zspace\rightarrow\aaa_0
\]
that factors through $\autspace_{P_0}$.
This function has the following strong uniform continuity property:
for any compact set $C\subset G(\A)$ the set
\begin{equation} \label{eq: Huc}
\{\Ht_0(xg)-\Ht_0(x)\mid x\in\zspace,g\in C\}\text{ is compact}.
\end{equation}
Indeed, it is equal to $\{\Ht_0(kg))\mid k\in\K,g\in C\}$.

Although the geometries of the spaces $\zspace$ and $\aaa_0$ are rather different,
we can nevertheless translate some analytic notions on $\aaa_0$ to $\zspace$.
For instance, on $\aaa_0$ we have the notion of functions of moderate growth, namely
functions that are majorized by (i.e., bounded by a constant multiple of) $e^{\norm{\cdot}}$
for some norm $\norm{\cdot}$ of $\aaa_0$. This notion immediately translates to $\zspace$.
For convenience, fix a $\Weyl$-invariant Euclidean norm $\norm{\cdot}$ on $\aaa_0$.
The space of functions of moderate growth on $\zspace$ is the union over $R>0$
of the Banach spaces $\functb^R(\zspace)$ of functions $f$ on $\zspace$ that
are majorized by $e^{R\norm{\Ht_0(\cdot)}}$. %(Obviously, this union does not depend on the choice of norm on $\aaa_0$.)
The smooth part $\functb_{\smth}^R(\zspace)$ of the space $\functb^R(\zspace)$ is
the union over open subgroups $K$ of $G(\A_f)$ (with the inductive limit topology in the category of \lctvs s)
of the Fr\'echet spaces of smooth, right $K$-invariant functions $f$ on $\zspace$
such that $\delta(X)f$ belongs to $\functb^R(\zspace)$ for all $X\in\univ(\gggg)$.
The union over $R>0$ of $\functb_{\smth}^R(\zspace)$ is by definition, the space
$\umd(\zspace)$ of smooth functions of uniform moderate growth on $\zspace$.

Reduction theory tells us that roughly speaking, we can model the space $\autspace$ in a cone in $\aaa_0$.
(See \cite{MR3026269}*{Chapitre 3}, \cite{MR1603257}*{\S2} and the references therein.)
More precisely, let \index{aaa0@$\aaa_{0,+}$}
\[
\aaa_{0,+}=\{X\in\aaa_0\mid\sprod{\alpha}X>0\text{ for all }\alpha\in\srts_0\}
\]
be the positive (open) Weyl chamber. Fix $T_0\in\aaa_0$ and let \index{Siegel@$\Siegel$}
\[
\Siegel=\Siegel^{T_0}\subset\zspace
\]
be the inverse image of $T_0+\aaa_{0,+}$ under $\Ht_0$.
Thus, $\Siegel$ is open in $\zspace$.
It is essentially a Siegel domain.\footnote{Traditionally, Siegel domains are usually defined
as certain subsets of $G(\A)$. However, it is advantageous to consider them as subsets of $\zspace$.}
Let \index{proj@$\proj$}
\[
\proj:\zspace\rightarrow\autspace
\]
be the projection and \index{projT0@$\proj^{\Siegel}$}
let $\proj^{\Siegel}$ be the restriction of $\proj$ to $\Siegel$.
By reduction theory, the fibers of $\proj^{\Siegel}$ are finite
and their cardinalities and diameters (or more precisely, the diameter of their image under $\Ht_0$)
are uniformly bounded (in terms of $T_0$).
%Moreover, if $\proj^{\Siegel}(x_1)=\proj^{\Siegel}(x_2)$ then $\Ht_0(x_1)-\Ht_0(x_2)$ lies in a compact set
%which depends on $T_0$ only (\cite{MR3026269}*{Lemme 3.5.5}).
Moreover, $\proj^{\Siegel}$ is surjective provided that
$\sprod{\alpha}{T_0}\ll0$ for all $\alpha\in\srts_0$.
We will fix such $T_0$ once and for all.

Thus, for analytic purposes we can model functions on $\autspace$ by functions on $\Siegel$
via the pullback $f^{\Siegel}=f\circ\proj^{\Siegel}$.
In particular, for any $R>0$ we can consider the space of functions on $\Siegel$
that are majorized by $e^{R\norm{\Ht_0(\cdot)}}$, and correspondingly the space $\functb^R(\autspace)$.
The space $\umd(\autspace)$ of functions of uniform moderate growth on $\autspace$ is, by definition,
the union over $R>0$ of the smooth part $\functb^R_{\smth}(\autspace)$ of $\functb^R(\autspace)$.\footnote{This
is equivalent to the prevalent definition in the literature using a height function on $G(\A)$.}
Note that since $\Siegel$ is open in $\zspace$, the space $\functb^R_{\smth}(\Siegel)$ is also well defined.

We remark that the notions of functions of (uniform) moderate growth on $\zspace$ and $\autspace$ are compatible in the sense
that for all $R>0$, the pullback by $\proj$ defines a $G(\A)$-equivariant operator
\[
\funct^R(\autspace)\rightarrow\funct^R(\zspace).
\]
That is, if the pullback of a function on $\autspace$ to $\zspace$ is majorized by $e^{R\norm{\Ht_0(\cdot)}}$
on the Siegel domain, then it is majorized by it throughout $\zspace$.
To see this, we note that there exists a constant $C$ (depending on the choice of $\Siegel$) such that
\[
\norm{\Ht_0(x)}\le \norm{\Ht_0(y)}+C
\]
for any $x\in\Siegel$ and $y\in\zspace$ such that $\proj(x)=\proj(y)$.
Indeed, this follows from the fact that for any $s\in\Weyl$, an element $\gamma$ in the double coset $P_0(F)s P_0(F)$,
and $g\in G(\A)$ we have
\[
s^{-1}\Ht_0(\gamma g)=\Ht_0(g)+\sum x_\beta\beta^\vee
\]
where the sum on the right-hand side is over the positive roots $\beta$ such that $s\beta<0$,
and $x_\beta\in\R$ are bounded below (depending only on $G$). %(cf. \cite{MR3026269}*{Lemma 3.3.2}).
In particular, if $H_0(g)$ lies in a fixed translate of the positive Weyl chamber, then
\[
\norm{\Ht_0(g)}-\norm{\Ht_0(\gamma g)}
\]
is bounded above.

Let $P\in\stds$. \label{sec: relSiegel}
We can define similarly the spaces $\functb^R(\autspace_P)$ by considering \index{Siegelrel@$\Siegel^P$}
the relative Siegel domain $\Siegel^P=\Siegel^{P,T_0}$,
which is the inverse image under $\Ht_0:\zspace_P\rightarrow\aaa_0$ of the translate by $T_0$ of the cone
\[
\{X\in\aaa_0\mid\sprod{\alpha}X>0\text{ for all }\alpha\in\srts_0^P\}.
\]
As before, the space $\umd(\autspace_P)$ of functions of uniform moderate growth on $\autspace_P$ is
the union over $R>0$ of the smooth part $\functb^R_{\smth}(\autspace_P)$ of $\functb^R(\autspace_P)$.

The constant term $f\mapsto\cnst_{G,P}f$ defines operators $\functb^R(\autspace)\rightarrow\functb^R(\autspace_P)$,
$\functb_{\smth}^R(\autspace)\rightarrow\functb_{\smth}^R(\autspace_P)$ and
$\umd(\autspace)\rightarrow\umd(\autspace_P)$.

\subsection{Statement} \label{sec: mainsys}
\begin{comment}
Recall
\[
\autspace=G(F)\bs G(\A),\ \ \autspace_P=P(F)U(\A)\bs G(\A).
\]
$\AAA_P$ acts on $\autspace_P$.
\end{comment}

%For any $\alpha\in\srts_0$ let $P_\alpha$ be the corresponding maximal standard parabolic subgroup of $G$.

\begin{comment}
We complete $\srts_0$ to a basis $\srtscmp$ of $\aaa_0^*$ arbitrarily.
For $\alpha\in\srtscmp\setminus\srts_0$ we write $P_\alpha=G$.
\end{comment}

For any simple root $\alpha\in\srts_0$ let $P_\alpha$ be the corresponding maximal parabolic subgroup of $G$.
Fix once and for all elements $a_\alpha\in\AAA_{P_\alpha}^G$, $\alpha\in\srts_0$ (in particular $a_\alpha$
lies in the center of the Levi part of $P_\alpha$) such that $\sprod{\alpha}{\Ht_0(a_\alpha)}>0$.
Denote by $T_{a_\alpha}$ the twisted action by $a_\alpha$ on functions on $\autspace_{P_{\alpha}}$ \eqref{eq: ta}.

In the case where $\AAA_G\ne1$ (i.e., when there is a nontrivial split torus in the center of $G$)
we also fix elements $z_1,\dots,z_r$ in $\AAA_G$ such that $\Ht_0(z_1),\dots,\Ht_0(z_r)$ is a basis of $\aaa_G$.
The operators $T_{z_j}$ on functions on $\autspace$ are simply translation by $z_j$.

Let $\mnfld$ be a complex manifold.
Suppose that we are given the following data.
\begin{itemize}
\item $I$ is a (possibly infinite) index set and for every $i\in I$,
$h_i(s)$, $s\in\mnfld$ is a holomorphic family of smooth, compactly supported functions on $G(\A)$
(see example \ref{ex: ccinfty}) and $c_i$ is a scalar-valued analytic function on $\mnfld$.
We assume that for every $s\in \mnfld$ there exists $i\in I$ such that $c_i(s)\ne0$.
\item For each $\alpha\in\srts_0$ %\srtscmp$ an integer $m_{\alpha}\ge0$ and
a family $D_\alpha(s)$, $s\in\mnfld$ of monic polynomials of degree $m_{\alpha}$
in one variable whose coefficients depend holomorphically on $s$.
\item If $\AAA_G\ne1$ we have in addition for every $j=1,\dots,r$
a family $\tilde D_j(s)$, $s\in\mnfld$ of monic polynomials of degree $\tilde m_j$
in one variable whose coefficients depend holomorphically on $s$ and whose constant coefficient
is nowhere vanishing on $\mnfld$.
%For $\alpha\in\srtscmp\setminus\srts_0$ we assume that the constant term of $D_\alpha(X,s)$ is nonzero for all $s\in\mnfld$.
\end{itemize}

Consider the holomorphic system $\Xi_{\at}(s)$ of linear equations on $f\in\umd(\autspace)$ given by
%\begin{subequations}
\begin{align*}
%\label{eq: sysmain1}
\delta(h_i(s))f&=c_i(s)f,\ \ i\in I,\\
%\label{eq: sysmain2}
D_{\alpha}(s)(T_{a_\alpha})(\cnst_{G,P_{\alpha}}f)&=0\text{ for every }\alpha\in\srts_0,\\ %cmp.
\tag{in the case $\AAA_G\ne1$}
%\label{eq: sysmain3}
\tilde D_j(s)(T_{z_j})(f)&=0\text{ for every }j=1,\dots,r. %cmp.
\end{align*}
%\end{subequations}
%\Erez{explain notation}

Note that for the second equation we only need to consider the constant term pointwise.
In particular, we do not need to consider function spaces on $\autspace_P$ for $P\subsetneq G$.

\begin{theorem} \label{thm: mainfin}
The system $\Xi_{\at}(s)$ is locally of finite type.
\end{theorem}

\begin{remark}
\begin{enumerate}
\item As we will recall in \S\ref{sec: HC} below, by a result of Harish-Chandra, such a system is satisfied by Eisenstein series.
\item In the case where $\mnfld$ is a point $s_0$, the theorem amounts to the finite-dimensionality of
the space $\Sol(\Xi_{\at}(s_0))$.
This is a variant of the standard finiteness theorem of Harish-Chandra. (See also Lemma \ref{sec: conc1} below.)
\item %In essence, it is not any harder to prove Theorem \ref{thm: mainfin} in the general case than in the case where
%$\mnfld$ is a single point.
Morally, any reasonable proof in the case $\mnfld=\{s_0\}$ should extend,
at least in principle, to the general case. %\Erez{Reformulate}
%\item In essence, it is not any harder to prove Theorem \ref{thm: mainfin} in the general case than in the case where
%$\mnfld$ is a single point. Put differently, any reasonable proof in the case $\mnfld=\{s_0\}$ should extend,
%at least in principle, to the general case.
\end{enumerate}
\end{remark}

Let us give an outline of the proof of the theorem.
It is based on familiar ideas, except that we have to pay attention to local uniformity in the parameters.

We will assume for simplicity that $\AAA_G=1$. The necessarily modifications for the general case
(which are of bookkeeping nature) will be explained in \S\ref{sec: AGne1}.

First, since the statement is local in $\mnfld$, we may assume without loss of generality
that the coefficients of $D_{\alpha}(s)$ are bounded
and that $\Xi_{\at}(s)$ contains an equation of the form
\[
\delta(h(s))f=f
\]
for a (single) holomorphic family $h(s)$ of smooth compactly supported functions on $G(\A)$.

Next, it is more convenient to work with Hilbert spaces.
For any $\lambda\in\aaa_0^*$ there is a Hilbertian space $\funct^\lambda(\autspace)$ of functions on $\autspace$
that was considered by Franke \cite{MR1603257}.
The union over $\lambda$ of the smooth part of $\funct^\lambda(\autspace)$ coincides with the space $\umd(\autspace)$
of smooth functions of uniform moderate growth on $\autspace$.
%(There is no preferred Hilbert norm, but as a \lctvs{} it is well defined.)
For every $\lambda$ we define a holomorphic system of equations $\Xi^\lambda(s)$ on $\funct^\lambda(\autspace)$.
%for a certain weight function $\tilde w_\lambda$ on $\autspace$ which depends on a parameter .
We will show that for sufficiently positive $\lambda$, $\Xi^\lambda(s)$ is locally of finite type and its solutions
contain those of $\Xi_{\at}(s)$.
By an easy argument, this will imply that $\Xi_{\at}(s)$ is locally of finite type.
%, since $\delta(h(s))$ defines an analytic family of operators
%\[
%\funct^\lambda(\autspace)\rightarrow\umd(\autspace).
%\]
%In fact, we show that $\Xi^\lambda(s)$ is locally, strongly of finite type.
%which depends on a parameter $\lambda\in\aaa_0^*$ which is sufficiently positive.

%Let $\proj:\zspace\rightarrow\autspace$ be the projection map.
%In order to define $\funct^\lambda(\autspace)$ and $\Xi^\lambda(s)$, we fix a Siegel domain $\Siegel$
%in $\zspace=P_0(F)\bs G(\A)$ such that $\proj(\Siegel)=\autspace$.
%For any function $f$ on $\autspace$ we denote by $f^{\Siegel}=f\circ\proj\rest_{\Siegel}$ its pullback to $\Siegel$.
We consider a weighted $L^2$-space
\[
\funct^\lambda(\Siegel)=L^2(\Siegel,w_\lambda^{-2}\ dx)
\]
for a certain explicit weight function $w_\lambda$ (see \S\ref{sec: HS}).
%Here $T_0\in\aaa_0$ is a parameter such that $\sprod{\alpha}{T_0}\ll0$ for all $\alpha\in\srts_0$.
By definition, $\funct^\lambda(\autspace)$ is the space of functions $f$ on $\autspace$ such that
$\norm{f^{\Siegel}}_{\funct^\lambda(\Siegel)}<\infty$.
(Up to equivalence, this norm does not depend on the choice of $\Siegel$.)
%In other words, $f\mapsto f^{\Siegel}$ defines a strict embedding
%\begin{equation} \label{eq: SE}
%\[
%\text{strict embedding }
%\funct^\lambda(\autspace)\rightarrow \funct^\lambda(\Siegel).
%\]
%\end{equation}
%for a certain weight function $w_\lambda$. (The
%Here $w_\lambda=e^{-\sprod{\lambda}{\Ht_0(\cdot)}}$.
%Therefore, we can perform the analysis on  instead.
The advantage of the space $\funct^\lambda(\Siegel)$ is that it admits a Harish-Chandra decomposition
\begin{equation} \label{eq: HCd}
\funct^\lambda(\Siegel)=\oplus_{P\in\stds}\funct^\lambda_{\cusp}(\Siegel_P)
\end{equation}
(see \S\ref{sec: HCdecomp}). Here $\Siegel_P$ is the image of $\Siegel$ under the projection $\zspace\rightarrow\zspace_P=U(\A)P_0(F)\bs G(\A)$.
Moreover $T_{a_\alpha}$ acts on the $P$-th summand of \eqref{eq: HCd} for any $\alpha\in\srts_0\setminus\srts_0^P$.
For any $f\in\funct^\lambda(\autspace)$ we denote by $f_P^{\Siegel}$, $P\in\stds$ the components of $f^{\Siegel}$
with respect to \eqref{eq: HCd}.

\begin{comment}
Unfortunately, it is not possible to encode the system $\Xi^\lambda(s)$ directly on $\funct^\lambda(\Siegel)$,
since the operators $\delta(h(s))$ are not defined there.
However, we can define operators
\[
\delta^{\Siegel',\Siegel}(h(s)):L^2(\Siegel',w_\lambda\ dx)\rightarrow \funct^\lambda(\Siegel)
\]
provided that $T_0-T_0'$ is sufficiently positive with respect to the support of $h(s)$. These operators
are compatible with the decomposition \eqref{eq: HCd}.

Suppose that $\phi$ is a solution of $\Xi^\lambda(s)$. Let $\phi^{\Siegel}$ be the pullback of $\phi$ to $\Siegel$
and similarly for $\phi^{T_0'}$. Let $\phi_P^{\Siegel}$, $P\in\stds$ be the components of $\phi^{\Siegel}$ with respect to
\eqref{eq: HCd}. %; similarly for $\phi_P^{T_0'}$.
Then, for any $P\in\stds$
\begin{gather*}
D_{\alpha}(s)(T_{a_\alpha})\phi_P^{\Siegel}=0\text{ for all }\alpha\in\srts_0\setminus\srts_0^P,\\
\delta^{\Siegel',\Siegel}(h(s))\phi_P^{T_0'}=\phi_P^{\Siegel}.
\end{gather*}
\end{comment}

The system $\Xi^\lambda(s)$ consists of the following equations on $f\in \funct^\lambda(\autspace)$.
%We will show that for sufficiently positive $\lambda\in\aaa_0^*$, the system of equations
\begin{gather*}
D_{\alpha}(s)(T_{a_\alpha})(f_P^{\Siegel})=0\text{ for all }P\in\stds, \alpha\in\srts_0\setminus\srts_0^P,\\
%(\delta^{\Siegel',\Siegel}(h(s))\phi_P^{T_0'})\rest_{\Siegel_P^{\ball}}=\phi_P^{\Siegel}\rest_{\Siegel_P^{\ball}},
f_P^{\Siegel}\rest_{\Siegel_P^{\ball}}=(\delta(h(s))f)_P^{\Siegel}\rest_{\Siegel_P^{\ball}}\text{ for all }P\in\stds.
\end{gather*}
Here $\Siegel_P^{\ball}$ is a certain subset of $\Siegel_P$ obtained by bounding the directions along the simple roots outside $\srts_0^P$ (see \eqref{def: spball}).

We show that $\Xi^\lambda(s)$ satisfies the conditions of the Fredholm criterion (Corollary \ref{cor: fred2}) provided that $\lambda$ is sufficiently positive.
More precisely, let $\funct^\lambda(\Siegel_P^{\ball})=L^2(\Siegel_P^{\ball},w_\lambda^{-2}\ dx)$. Define operators
\begin{subequations} \label{eq: munu}
\begin{equation}
\mu_s,\nu_s:\funct^\lambda(\autspace)\rightarrow\oplus_{P\in\stds}\big(\funct^\lambda(\Siegel_P)^{\srts_0\setminus\srts_0^P}
\oplus \funct^\lambda(\Siegel_P^{\ball})\big)
\end{equation}
by
\begin{gather}
\label{eq: defmu} \mu_sf=\big((D_{\alpha}(s)(T_{a_\alpha})(f_P^{\Siegel}))_{\alpha\in\srts_0\setminus\srts_0^P},
f_P^{\Siegel}\rest_{\Siegel_P^{\ball}}\big)_{P\in\stds}\\
\nu_sf=\big(0,(\delta(h(s))f)_P^{\Siegel}\rest_{\Siegel_P^{\ball}}\big)_{P\in\stds}.
\end{gather}
\end{subequations}
Then, the equations in $\Xi^\lambda(s)$ can be rewritten as
\[
\mu_sf=\nu_sf.
\]
We show that (assuming, as we recall $\AAA_G=1$)
\begin{enumerate}
\item If $\lambda$ is sufficiently positive, then $\mu_s$ is a strict embedding for every $s\in\mnfld$.
\item $\nu_s$ is compact (and in fact, Hilbert--Schmidt) for every $s\in\mnfld$.
\end{enumerate}
\begin{comment}
\begin{enumerate}
\item The operators
\begin{align*}
\funct^\lambda(\autspace)&\rightarrow\oplus_{P\in\stds}\big(\funct^\lambda(\Siegel_P)^{\srts_0\setminus\srts_0^P}
\oplus \funct^\lambda(\Siegel_P^{\ball})\big)\\
f&\mapsto\big((D_{\alpha}(s)(T_{a_\alpha})(f_P^{\Siegel}))_{\alpha\in\srts_0\setminus\srts_0^P},
f_P^{\Siegel}\rest_{\Siegel_P^{\ball}}\big)_{P\in\stds}
\end{align*}
are strict embeddings for every $s\in\mnfld$.
\item For every $P\in\stds$ the operators
\[
\funct^\lambda(\autspace)\rightarrow \funct^\lambda(\Siegel_P^{\ball})=L^2(\Siegel_P^{\ball},w_\lambda^{-2}\ dx),\ \
f\mapsto (\delta(h(s))f)_P^{\Siegel}\rest_{\Siegel_P^{\ball}}
\]
for every $s\in\mnfld$.
\end{enumerate}
\end{comment}
Both claims are proved using analysis on the Siegel domain.
For the first one, we write $\mu_s$ as the composition of the strict embedding
\[
\funct^\lambda(\autspace)\rightarrow\oplus_{P\in\stds}\funct^\lambda(\Siegel_P),\ \ f\mapsto (f_P^{\Siegel})_{P\in\stds}
\]
with the direct sum over $P\in\stds$ of the operators
\begin{align*}
\funct^\lambda(\Siegel_P)&\rightarrow \funct^\lambda(\Siegel_P)^{\srts_0\setminus\srts_0^P}
\oplus \funct^\lambda(\Siegel_P^{\ball})\\
f&\mapsto \big((D_{\alpha}(s)(T_{a_\alpha})(f))_{\alpha\in\srts_0\setminus\srts_0^P},f\rest_{\Siegel_P^{\ball}}\big).
\end{align*}
It is easy to show that these operators are strict embeddings provided that
$e^{\sprod{\lambda}{\Ht_0(a_\alpha)}}>\abs{r}$ for every root $r$ of $D_\alpha(s)$
(see Proposition \ref{prop: IP(s)},
which also implies that the solutions of $\Xi_{\at}(s)$ are in $\funct^\lambda(\autspace)$).
As in the $\SL_2$ case (\S\ref{sec: SL2}), this essentially boils down to the elementary fact that the operator
\[
L^2(\R_+,e^{-2ax}\ dx)\rightarrow L^2(\R_+,e^{-2ax}\ dx)\oplus L^2([0,1]),\ \ f\mapsto (f(x+1)-rf(x),f\rest_{[0,1]})
\]
is a strict embedding provided that $e^a>\abs{r}$.

The compactness of $\nu_s$ is a standard result (cf.\ \cite{MR0232893}*{\S I.4} and Lemma \ref{lem: GGPS} below).

In the following subsections we fill in the details in the proof above.

\subsection{Harish-Chandra's decomposition \cite{MR0232893}*{\S I.3}} \label{sec: HCdecomp}
Recall
\[
\zspace=P_0(F)\bs G(\A).
\]
More generally, for any $P=M\ltimes U\in\stds$ let \index{ZP@$\zspace_P$}
\[
\zspace_P=P_0(F)U(\A)\bs G(\A).
\]
In particular, $\zspace=\zspace_G$ and $\zspace_{P_0}=\autspace_{P_0}$.

\begin{remark}
Let $P_0^M=P_0\cap M$, a minimal parabolic subgroup of $M$ defined over $F$ and $\zspace^M=P_0^M(F)\bs M(\A)$,
the analogue of $\zspace$ with respect to $M$. Then, $\zspace_P$ is the fibered product
\[
\zspace_P=P_0^M(F)U(\A)\bs G(\A)=\zspace^M\times_{M(\A)\cap\K}\K.
\]
Thus, working with functions on $\zspace_P$ is essentially the same as working with functions on $\zspace^M$.
\end{remark}

We have a proper surjection \index{betaP@$\beta_P$}
\[
\beta_P:\zspace\rightarrow\zspace_P.
\]
\begin{comment}
For any integrable function $f$ on $\zspace_P$ we have
\begin{equation} \label{eq: betaP}
\int_{\zspace}f\circ\beta_P(g)\ dx=\int_{\zspace_P}f(g)\ dx.
\end{equation}
\end{comment}
We identify the space $L^1_{\loc}(\zspace_P)$ of locally $L^1$ functions on $\zspace_P$ with a subspace
of $L^1_{\loc}(\zspace)$ via the pullback by $\beta_P$.
Define constant term projections \index{CPQ@$\CT_P$}
\[
\CT_P:L^1_{\loc}(\zspace)\rightarrow L^1_{\loc}(\zspace_P),\ \
\CT_Pf(g)=\int_{U(F)\bs U(\A)}f(ug)\ du,\ \ g\in \zspace.
\]
These maps are $G(\A)$-equivariant.
\begin{comment}
If $f\in L^1(\zspace,dg)$, then
\begin{equation} \label{eq: intsame}
\int_{\zspace}f(g)\ dx=\int_{\zspace_P}\CT_Pf(g)\ dx.
\end{equation}
This follows from uniqueness of $G(\A)$-invariant functional on $L^1(\zspace_P)$.
\end{comment}

For any $P_1,P_2\in\stds$ we have $\CT_{P_1\cap P_2}=\CT_{P_1}\circ \CT_{P_2}$
(since $U_{P_1}U_{P_2}=U_{P_1\cap P_2}$).
In particular, the operators $\CT_P$, $P\in\stds$ pairwise commute.
Let $L^1_{\loc,\cusp}(\zspace_P)$ be the cuspidal part of $L^1_{\loc}(\zspace_P)$, i.e.,
\[
L^1_{\loc,\cusp}(\zspace_P)=\bigcap_{Q\subsetneq P}\Ker(\CT_Q\rest_{L^1_{\loc}(\zspace_P)})\subset L^1_{\loc}(\zspace_P).
\]
Then, we have a direct sum decomposition
\begin{equation} \label{eq: orthd1}
L^1_{\loc}(\zspace)=\oplus_{P\in\stds}L^1_{\loc,\cusp}(\zspace_P).
\end{equation}
The projection on $L^1_{\loc,\cusp}(\zspace_P)$ is given by
\[
\sum_{Q\subset P}(-1)^{\dim\aaa_Q^P}\CT_Q.
\]

The decomposition \eqref{eq: orthd1} also holds for smaller classes of functions on $\zspace$ such as continuous, smooth,
of uniform moderate growth, etc.
Another useful example is a weighted $L^2$ space $L^2(\zspace;w\ dx)$ where $w:\zspace\rightarrow\R_{\ge0}$
is a locally bounded function that factors through $\beta_{P_0}$.
For any $P\in\stds$ we can identify $L^2(\zspace_P;w\ dx)$ with the closed subspace
of functions in $L^2(\zspace;w\ dx)$ that factor through $\beta_P$, and the constant term map
\[
%\CT_P^w:
L^2(\zspace;w\ dx)\rightarrow L^2(\zspace_P;w\ dx),\ \ f\mapsto\int_{U(F)\bs U(\A)}f(ug)\ du
\]
is the orthogonal projection.
We therefore get an orthogonal decomposition
\[
L^2(\zspace;w\ dx)=\oplus_{P\in\stds}L^2_{\cusp}(\zspace_P;w\ dx).
\]

\begin{comment}
We also note that the restriction of $\CT_P^{\Siegel}$ to $\functwt(\autspace)$ agrees with $\cnst_P$, i.e.
for any $f\in\functwt(\autspace)$
\begin{equation} \label{eq: ctt0}
\CT_P^{\Siegel}\iota^{\aut,T_0}f=(\cnst_Pf)\circ\proj_P^{\Siegel}
\end{equation}
where $\proj_P^{\Siegel}$ is the restriction of the projection $\zspace_P\rightarrow\autspace_P$ to $\Siegel_P$.
\end{comment}

\label{sec: HS}

We will apply it in the following situation.
For any $\lambda\in\aaa_0^*$ let $w_\lambda:\zspace\rightarrow\R_{>0}$ be the function
\[
w_\lambda(g)=e^{\sprod{\lambda}{\Ht_0(g)}}
\]
and consider the weighted $L^2$-space with $w=w_\lambda^{-2}$
\[
\funct^\lambda(\Siegel)=L^2(\Siegel,w_\lambda^{-2}\ dx).
\]
Then as before, we have an orthogonal decomposition
\begin{equation} \label{eq: orthd}
\funct^\lambda(\Siegel)=\oplus_{P\in\stds}\funct^\lambda_{\cusp}(\Siegel_P)
\end{equation}
where $\Siegel_P$ is the image of $\Siegel$ under $\beta_P$.
%(It is strictly contained in the relative Siegel domain $\Siegel^P$ defined in \S\ref{sec: relSiegel}.)

The family of spaces $\funct^\lambda(\Siegel)$, $\lambda\in\aaa_0^*$ is monotonous in the sense that
if all the coefficients of $\lambda'-\lambda$ with respect to the simple roots $\srts_0$ are non-negative,
then $\funct^\lambda(\Siegel)\subset\funct^{\lambda'}(\Siegel)$.

We will say that $\lambda\in\aaa_0^*$ is sufficiently positive (depending on the context) if the coefficients
of $\lambda$ with respect to $\srts_0$ are sufficiently large.

\subsection{The spaces $\funct^\lambda(\autspace)$}

Define $\funct^\lambda(\autspace)$ to be the space of functions $f$ on $\autspace$ such that
the induced norm from $\funct^\lambda(\Siegel)$ via the pullback $f\mapsto f^{\Siegel}$ is finite.
%In other words, $f\mapsto f^{\Siegel}$ defines a closed embedding
%\[
%\funct^\lambda(\autspace)\rightarrow \funct^\lambda(\Siegel).
%\]
As a Hilbertian space, $\funct^\lambda(\autspace)$ does not depend on the choice of $\Siegel$.
Alternatively, we may view $\funct^\lambda(\autspace)$ (as a Hilbertian space) as a weighted $L^2$ space
$L^2(\autspace,(w_\lambda\circ\sigma)^{-2}\ dx)$
for any right inverse $\sigma:\autspace\rightarrow\Siegel$ of $\proj^{\Siegel}$ (cf.\ \cite{MR1603257}*{\S2}).
%(\tilde w_\lambda^{\Siegel})^{-2}
%for any function $w$ on $\autspace$ such that $w(x)\in w_\lambda((\proj^{\Siegel})^{-1}(x))$ for every $x\in\autspace$.
%(For instance we can take
% where $\tilde w_\lambda^{\Siegel}$
%is either of the (equivalent) weights
% (\tilde w_\lambda^{\Siegel})^{-2})$ with weight function
%\[
%\tilde w_\lambda^{\Siegel}(x)=
%$w=\max_{(\proj^{\Siegel})^{-1}(x)}w_\lambda$.)
%\text{ or }w=\min_{(\proj^{\Siegel})^{-1}(x)}w_\lambda.
%\]
(All such weights are equivalent, also when we vary $\Siegel$.)

This is because for any nonnegative function $f$ on $\autspace$ we have
\[
\int_{\Siegel}f^{\Siegel}(x)\ dx=\int_{\autspace}f(x)\ \#((\proj^{\Siegel})^{-1}(x))\ dx
\]
and hence
%\begin{equation} \label{eq: clsdembd}
\[
\int_{\autspace}f(x)\ dx\le\int_{\Siegel}f^{\Siegel}(x)\ dx\le c_1\int_{\autspace}f(x)\ dx
\]
%\end{equation}
where $c_1$ is a constant (depending on the choice of $\Siegel$).

%Of course, $\tilde w_\lambda^{\Siegel}$
%depends on $\Siegel$ but a different choice of $\Siegel$
%gives an equivalent weight.
%but if $T_0'$ is another point such that
%$\proj(\Siegel')=\autspace$, then
%\[
%c_1\tilde w_\lambda^{\Siegel}(x)\le \tilde w_\lambda^{\Siegel'}(x)\le c_2\tilde w_\lambda^{\Siegel}(x),\ \ x\in\autspace
%\]
%for some constants $c_1,c_2>0$ depending only on $\lambda$, $T_0$ and $T_0'$.

The group $G(\A)$ acts on $\funct^\lambda(\autspace)$ by right translation.

Indeed, by \eqref{eq: Huc}, for any compact set $C\subset G(\A)$, the ratio $\frac{w_\lambda(xg)}{w_\lambda(x)}$
is bounded uniformly in $x\in\zspace$ and $g\in C$.
%, since
%\begin{equation} \label{eq: chngcmpt}
%\sup_{g\in G(\A),x\in C}\wgt(gx)\wgt(g)^{-1}=\sup_{k\in\K,x\in C}\wgt(kx)<\infty.
%\end{equation}
The analogous property for $w_\lambda\circ\sigma$ immediately follows.

%\S\ref{sec: fumg}

Recall that we defined the space $\umd(\autspace)$ of functions of uniform moderate growth on $\autspace$ using pointwise bounds
(i.e., as the union over $R>0$ of the spaces $\functb^R_{\smth}(\autspace)$).
Alternatively, We can define $\umd(\autspace)$ equivalently using the Hilbert spaces $\funct^\lambda(\autspace)$.
This is because we can compare the spaces $\functb^R(\autspace)$ and $\funct^\lambda(\autspace)$ as follows.

\begin{lemma} \label{lem: GGPS}
\begin{enumerate}
\item Any function of moderate growth on $\autspace$ belongs to
$\funct^\lambda(\autspace)$ for $\lambda$ sufficiently positive.  %\cite{MR1361168}*{I.2.2}
That is, for every $R>0$ we have a continuous embedding
\[
\functb^R(\autspace)\subset \funct^\lambda(\autspace)
\]
provided that $\lambda\in\aaa_0^*$ is sufficiently positive.
\item In the other direction, for any $\lambda\in\aaa_0^*$ there exists $R>0$ such that for any
bounded, compactly supported function $h$ on $G(\A)$,
$\delta(h)$ defines an operator from $\funct^\lambda(\autspace)$ to $\functb^R(\autspace)$.
If moreover $h$ is smooth, then
$\delta(h)$ defines an operator from $\funct^\lambda(\autspace)$ to $\functb_{\smth}^R(\autspace)$.

\item
Let $P\in\stds$. Let $\Siegel_P^{\ball}$ be a subset of $\Siegel_P$ of the form
\[
\Siegel_P^{\ball}=\{g\in\Siegel_P\mid \sprod{\alpha}{\Ht_0(g)-T_0}\le c_\alpha\text{ for all
}\alpha\in\srts_0\setminus\srts_0^P\}
\]
for some constants $c_\alpha$, $\alpha\in\srts_0\setminus\srts_0^P$.
Let $h$ be a compactly supported smooth function on $G(\A)$.
Then, the operator
\[
\funct^\lambda(\autspace)\rightarrow\funct^\lambda(\Siegel_P^{\ball}),\ \
f\mapsto(\delta(h)f)_P^{\Siegel}\rest_{\Siegel_P^{\ball}}
\]
is a Hilbert--Schmidt operator, and in particular compact.
%The following is a version of a familiar result due to Gelfand and Piatetski-Shapiro (cf.\ \cite{MR0232893}*{Theorem 2}).
\end{enumerate}
\end{lemma}

\begin{proof}
This is standard. All parts are proved using analysis on the Siegel domain.

The first part is clear.

%Of course we do not have containment in the other direction.
%However,

For the second part, let $h$ be a bounded, compactly supported function on $G(\A)$.
Although $\delta(h)$ does not factor through $L^1_{\loc}(\Siegel)$,
we can model it using an auxiliary Siegel domain.
Namely, by \eqref{eq: Huc} we can find a Siegel domain $\Siegel'$ that contains the right translate of $\Siegel$
by the support of $h$.
%(Namely, $\Siegel'=\Siegel^{T_0'}$ where $T_0'$ is such that $\sprod{\alpha}{T_0-T_0'}\gg0$
%for all $\alpha\in\srts_0$ -- see .)
Therefore, the convolution operator %\index{dbelP@$\delta^{\Siegel',\Siegel}(h)$}
\[
\delta^{\Siegel',\Siegel}(h):L^1_{\loc}(\Siegel')\rightarrow L^1_{\loc}(\Siegel),
\ \ f\mapsto\int_{G(\A)}h(g)f(\cdot g)\ dg
\]
is well defined and we have a commutative diagram
%$(\delta(h)f)^{\Siegel}=\delta^{\Siegel',\Siegel}(h)(f^{\Siegel'})$.
\[
\begin{tikzcd}
L^1_{\loc}(\autspace)\arrow{r}{f^{\Siegel'}}\arrow{d}{\delta(h)} &
L^1_{\loc}(\Siegel')\arrow{d}{\delta^{\Siegel',\Siegel}(h)}\\
L^1_{\loc}(\autspace)\arrow{r}{f^{\Siegel}}&L^1_{\loc}(\Siegel)
\end{tikzcd}
\]

By realizing $\delta^{\Siegel',\Siegel}(h)$ as an integral operator and estimating its kernel
(see e.g., ~\cite{MR1361168}*{I.2.5}) one shows that for $R\gg0$, $\delta^{\Siegel',\Siegel}(h)$ defines an operator
\[
\delta^{\Siegel',\Siegel}(h):\funct^\lambda(\Siegel')\rightarrow\functb^R(\Siegel).
\]
The second part follows.

To show the third part, note that $\delta^{\Siegel',\Siegel}(h)$ commutes with the constant term projections.
Hence, it induces operators
\[
\delta^{\Siegel'_P,\Siegel_P}(h):\funct^\lambda(\Siegel_P')\rightarrow\funct^\lambda(\Siegel_P)
\]
and respects the Harish-Chandra decomposition \eqref{eq: orthd}.

It follows that the map
\[
\funct^\lambda(\autspace)\rightarrow \funct^\lambda_{\cusp}(\Siegel_P),\ \  f\mapsto(\delta(h)f)_P^{\Siegel}
\]
is the composition the operator
$\funct^\lambda(\autspace)\rightarrow\funct^\lambda_{\cusp}(\Siegel'_P)$, $f\mapsto f_P^{\Siegel'}$
with $\delta^{\Siegel'_P,\Siegel_P}(h)$.
%\Erez{reference to \cite{MR1075727}}
%\begin{lemma}(cf.\ \cite{MR0232893}*{p. 14})
Hence, it suffices to show that the restriction of the operator
\[
\funct^\lambda(\Siegel'_P)\rightarrow\funct^\lambda(\Siegel_P^{\ball}),\ \
f\mapsto(\delta^{\Siegel'_P,\Siegel_P}(h)f)\rest_{\Siegel_P^{\ball}}
\]
to $\funct^\lambda_{\cusp}(\Siegel_P')$ is a Hilbert--Schmidt operator.
By \cite{MR1075727}*{\S1.6} it suffices to show that the evaluation map
\[
\operatorname{ev}_x:\funct^\lambda_{\cusp}(\Siegel_P')\rightarrow\C,\ \
f\mapsto(\delta^{\Siegel'_P,\Siegel_P}(h)f)(x),\ \ x\in\Siegel_P^{\ball}
\]
satisfies $x\mapsto\norm{\operatorname{ev}_x}\in\funct^\lambda(\Siegel_P^{\ball})$.
In fact, we show that the function $x\mapsto\norm{\operatorname{ev}_x}$, $x\in \Siegel_P^{\ball}$ belongs
to the space $\functb^{\rpd}(\Siegel_P^{\ball})$ of functions of rapid decay, defined by the norms
\[
\sup_{\Siegel_P^{\ball}}e^{R\norm{\Ht_0(\cdot)}}\abs{f},\ \ R>0.
\]

This follows by combining two facts.
The first is that as before, since $h$ is smooth, $\delta^{\Siegel'_P,\Siegel_P}(h)$ defines an operator
\[
\funct^\lambda(\Siegel'_P)\rightarrow\functb_{\smth}^R(\Siegel_P)
\]
for $R\gg0$. The second is that for every $R>0$ we have an operator
\[
\functb_{\smth,\cusp}^R(\Siegel_P)\rightarrow
\functb^{\rpd}(\Siegel_P^{\ball}),\ \ f\mapsto f\rest_{\Siegel_P^{\ball}}
%\funct^\lambda(\Siegel_P^{\ball}),\ \ f\mapsto f\rest_{\Siegel_P^{\ball}}
\]
(cf.\ \cite{MR1361168}*{I.2.10--11}). The lemma follows.
\end{proof}

It follows from the first two parts of the lemma by a standard argument (cf.\ \cite{MR1075727}*{\S2.4}),
that we can define the space of functions of uniform moderate growth equivalently as
the union over $\lambda$ of the smooth part $\funct^\lambda_{\smth}(\autspace)$ of $\funct^\lambda(\autspace)$.
That is, for any compact open subgroup $K$ of $G(\A_f)$ we have
%\begin{equation} \label{eq: isomumd}
\[
\umd(\autspace)^K=\cup_{\lambda\in\aaa_0^*}\funct^\lambda_{\smth}(\autspace)^K
\]
%\end{equation}
with the locally convex inductive limit topology on the right-hand side.

\subsection{}

In this section we show that the operators $\mu_s$ defined in \eqref{eq: defmu} are strict embeddings.
This is a special case of a more general setup. We start with the one-dimensional case.

Let $(X,\mu)$ be a measure space.
Let $c$ be a constant. We say that a transformation $\sigma:X\rightarrow X$ is $c$-renormalizing
if the pullback $\tau$ by $\sigma$ satisfies
\[
\norm{\tau f}=c\norm{f}
\]
for all $f\in L^2(X,\mu)$. Equivalently, $\mu(\sigma^{-1}A)=c^2\mu(A)$ for every measurable subset $A\subset X$.

\begin{lemma}
Assume that $\sigma:X\rightarrow X$ is a $c$-renormalizing transformation.

Let $h:X\rightarrow\R$ be a function such that $h(\sigma x)=h(x)+1$ for all $x\in X$.

For any subset $A\subset\R$ denote $X_A=h^{-1}(A)$.

Let $\tau^+$ be the pullback by $\sigma$ on functions on $X_{\R_+}$.

Let $D$ be a monic polynomial of degree $m$ with complex coefficients.
Assume that $\abs{r}<c$ for every root $r$ of $D$. Then, the operator
\[
L^2(X_{\R_+},\mu)\rightarrow L^2(X_{\R_{+}},\mu)\oplus L^2(X_{(0,m]},\mu),\ \ f\mapsto (D(\tau^+)f,f\rest_{X_{(0,m]}})
\]
is a strict embedding. Namely, there exist explicit constants $C_1,C_2>0$ such that for any
function $f$ on $X_{\R_+}$ we have
\[
\norm{f}\le C_1\norm{D(\tau^+)f}+C_2\norm{f\one_{X_{(0,m]}}}
\]
with respect to the $L^2(X_{\R_+},\mu)$-norm.
\end{lemma}

\begin{proof}
We prove it by induction on $m$. The case $m=0$ is obvious.
For the induction step, it is enough to consider the case $m=1$ i.e., $D=x-r$.
%(Note that the roots of $D$ are bounded in terms of the size of the coefficients of $D$ since $D$ is monic.)
By assumption,
\[
\norm{\tau^+ f}=c\norm{f\cdot\one_{X_{(1,\infty)}}}\ge c(\norm{f}-\norm{f\one_{X_{(0,1]}}}).
\]
On the other hand,
\[
\norm{\tau^+ f}\le\norm{D(\tau^+)f}+\abs{r}\norm{f}.
\]
Therefore,
\[
(c-\abs{r})\norm{f}\le\norm{D(\tau^+)f}+c\norm{f\one_{X_{(0,1]}}}.
\]
The lemma follows.
\end{proof}

We have the following more general, multidimensional version.
% is proved by easy induction.
%We omit the details.
%Let $(X,\mu)$ be a measure space as before and

\begin{corollary} \label{cor: forstrict}
Let $I$ be a finite index set.
For every $i\in I$, let $c_i$ be a constant and let $\sigma_i:X\rightarrow X$ be a $c_i$-renormalizing
transformation.

Let $V$ be a finite-dimensional vector space over $\R$ with linearly independent vectors $e_i$, $i\in I$
(not necessarily a basis).

Let $h:X\rightarrow V$ be a function such that $h(\sigma_ix)=h(x)+e_i$ for every $i\in I$ and $x\in X$.

For any subset $A\subset V$ denote $X_A=h^{-1}(A)$.

For every $i\in I$ let $\xi_i$ be a linear form on $V$ such that $\xi_i(e_i)>0$ and $\xi_i(e_j)=0$ for all $j\ne i$.
Let $V_+$ be the cone
\[
V_+=\{v\in V\mid\xi_i(v)>0\text{ for all }i\}.
\]

Let $\tau_i^+$ be the pullback by $\sigma_i$ on functions on $X_{V_+}$.

For every $i\in I$ let $D_i$ be a monic polynomial of degree $m_i$ such that $\abs{r}<c_i$ for every root $r$ of $D_i$.
Let
\[
B=\{v\in V\mid0<\xi_i(v)\le m_i\xi_i(e_i)\text{ for all }i\}\subset V_+.
\]
Then, the operator
\[
L^2(X_{V_+},\mu)\rightarrow L^2(X_{V_+},\mu)^I\oplus L^2(X_B,\mu),\ \ f\mapsto ((D_i(\tau_i^+)f)_{i\in I},f\rest_{X_B})
\]
is a strict embedding.
%Moreover, in this case, if $f$ is a function on $X_{V_+}$ such that $f\rest_{X_B}\in L^2(X_B,\mu)$ and
%$D_i(\tau_i)f\in L^2(X_{V_+},\mu)$ for all $i$, then $f\in L^2(X_{V_+},\mu)$.
\end{corollary}

The proof is by easy induction on the size of $I$ using the previous lemma. We omit the details.

\begin{comment}
\begin{proof}
We argue by induction on $d$. The case $d=0$ is trivial.
For the induction step we write $V_+=C_1\cap C_2$
where $C_1=\cap_{i=1}^{d-1}\xi_i^{-1}(\R_+)$ and $C_2=\xi_d^{-1}(\R_+)$ and similarly
$B=B_1\cap B_2$ where $B_1=\cap_{i=1}^{d-1}\xi_i^{-1}((0,m_i])$ and $B_2=\xi_d^{-1}((0,m_d])$.
Applying the induction hypothesis to the space $X_{C_2}$, %the restriction of $h$, the vectors $e_1,\dots,e_{d-1}$,
%the maps $\tau_1,\dots,\tau_{d-1}$ and the linear forms $\xi_1,\dots,\xi_{d-1}$,
we have a strict embedding
\[
L^2(X_{V_+},\mu)\rightarrow L^2(X_{V_+},\mu)^d,\ \ f\mapsto (D_1(\tau_1)f,\dots,D_{d-1}(\tau_{d-1})f,f\one_{X_{B_1}}).
\]
Composing it in the last coordinate with the strict embedding
\[
L^2(X_{V_+},\mu)\rightarrow L^2(X_{V_+},\mu)^2,\ f\mapsto (D_d(\tau_d)f,f\one_{X_{B_2}})
\]
of the previous lemma (applied to $X_{C_1}$, the function $\xi_d\circ h$ and the map $\sigma_d$) and noting that $\sigma_d^{-1}(X_{B_1})=X_{B_1}$ we obtain the first part.
The second part is similar.
\end{proof}
\end{comment}

Recall that the roots of a monic polynomial are bounded by the coefficients.
For instance, by Cauchy's bound, for any root $r$ of a monic polynomial $\sum a_ix^i$ we have
\[
\abs{r}\le 1+\max\abs{a_i}.
\]

We will apply the corollary above to the spaces $\funct^\lambda(\Siegel_P)$ considered in \S\ref{sec: HS}.

Fix $P\in\stds$.
For every $\alpha\in\srts_0\setminus\srts_0^P$ let $D_{\alpha}$ be a monic polynomial in one variable of degree $m_\alpha$ with complex coefficients.
Let
\begin{equation} \label{def: spball}
\Siegel_P^{\ball}=\{g\in\Siegel_P\mid \sprod{\alpha}{\Ht_0(g)-T_0}\le m_\alpha\sprod{\alpha}{\Ht_0(a_\alpha)}
\text{ for every }\alpha\in\srts_0\setminus\srts_0^P\}
\end{equation}
and
\[
\funct^\lambda(\Siegel_P^{\ball})=L^2(\Siegel_P^{\ball},w_\lambda^{-2}\ dx).
\]

\begin{proposition} \label{prop: IP(s)}
Suppose that for every $\alpha\in\srts_0\setminus\srts_0^P$, the pairing $\sprod{\lambda}{\Ht_0(a_\alpha)}$ is large with respect to the size
of the coefficients of $D_{\alpha}$. (More precisely, $e^{\sprod{\lambda}{\Ht_0(a_\alpha)}}>\abs{r}$
for every root $r$ of $D_\alpha$.)
Then, the operator
\begin{align*}
\funct^\lambda(\Siegel_P)&\rightarrow \funct^\lambda(\Siegel_P)^{\srts_0\setminus\srts_0^P}\oplus \funct^\lambda(\Siegel_P^{\ball})\\
f&\mapsto (D_{\alpha}(T_{a_\alpha})f)_{\alpha\in\srts_0\setminus\srts_0^P},f\rest_{\Siegel_P^{\ball}}
\end{align*}
%\oplus_{\alpha\in\srts_0\setminus\srts_0^P}D_{\alpha}(T_{a_\alpha})\oplus r_P^{T_0,\ball}:
is a strict embedding.
%Moreover, let $f$ be a function on $\Siegel_P$. Assume that
%\[
%D_{\alpha}(T_{a_\alpha})f\in \funct^\lambda(\Siegel_P)
%\]
%for all $\alpha\in\srts_0\setminus\srts_0^P$ and
%$f\rest_{\Siegel_P^{\ball}}\in \funct^\lambda(\Siegel_P^{\ball})$.
%Then $f\in \funct^\lambda(\Siegel_P)$.
\end{proposition}

Indeed, we simply apply the corollary above to the relative Siegel domain $X$
which is the inverse image under $\Ht_0:\zspace_P\rightarrow\aaa_0$ of the translate by $T_0$ of the cone
\[
\{Y\in\aaa_0\mid\sprod{\alpha}Y>0\text{ for all }\alpha\in\srts_0^P\},
\]
the measure $\mu=w_\lambda^{-2}\ dx$ on $X$, the vector space $V=\aaa_0$ and the map
\[
h:X\rightarrow V,\ \ h=\Ht_0-T_0.
\]
The index set $I$ is $\srts_0\setminus\srts_0^P$.
For every $\alpha\in I$, $e_\alpha=\Ht_0(a_\alpha)$ and $\xi_\alpha$ is $\alpha$ itself.
The transformation $\sigma_\alpha:X\rightarrow X$ is left translation by $a_\alpha\in\AAA_P$ and the constant $c_\alpha$ is
$\modulus_P(a_\alpha)^{\frac12}e^{\sprod{\lambda}{\Ht_0(a_\alpha)}}$.
Note that $\tau_\alpha=\modulus_P(a_\alpha)^{\frac12}T_{a_\alpha}$.

\subsection{Proof of Theorem \ref{thm: mainfin}} \label{sec: reducemain}
%Assuming Proposition \ref{prop: finite}, let us explain how it implies Theorem \ref{thm: mainfin}.

Let $\mnfld$ be a complex manifold.
For each $\alpha\in\srts_0$ %\srtscmp$ an integer $m_{\alpha}\ge0$ and
let $D_{\alpha}(s)$, $s\in\mnfld$ be a holomorphic family of monic polynomials of degree $m_\alpha$
in one variable whose coefficients are bounded.
%For $\alpha\in\srtscmp\setminus\srts_0$ we assume that the constant term of $D_\alpha(X,s)$ is nonzero for all $s\in\mnfld$.
Let $h(s)$, $s\in\mnfld$ be a holomorphic family of smooth, compactly supported functions on $G(\A)$.

\begin{comment}
For any $P\in\stds$ let {def: spball}
\[
\Siegel_P^{\ball}=\{g\in\Siegel_P\mid \sprod{\alpha}{\Ht_0(g)-T_0}\le m_\alpha\sprod{\alpha}{\Ht_0(a_\alpha)}
\text{ for all }\alpha\in\srts_0\setminus\srts_0^P\}.
\]
\end{comment}

Consider the system $\Xi^\lambda(s)$, $s\in\mnfld$ of equations on $f\in \funct^\lambda(\autspace)$ given by
\begin{gather*}
D_{\alpha}(s)(T_{a_\alpha})(f_P^{\Siegel})=0\text{ for all }P\in\stds, \alpha\in\srts_0\setminus\srts_0^P,\\
f_P^{\Siegel}\rest_{\Siegel_P^{\ball}}=(\delta(h(s))f)_P^{\Siegel}\rest_{\Siegel_P^{\ball}}\text{ for all }P\in\stds.
\end{gather*}
Recall that $f_P^{\Siegel}$, $P\in\stds$ are the components of $f^{\Siegel}$ with respect to
the Harish-Chandra decomposition \eqref{eq: orthd} and $\Siegel_P^{\ball}$ is defined in \eqref{def: spball}.

We write $\Xi^\lambda(s)$ in the form $\mu_sf=\nu_sf$ where $\mu_s$ and $\nu_s$ are defined in \eqref{eq: munu}.

Assume from now on that for every $\alpha\in\srts_0$, $\sprod{\lambda}{\Ht_0(a_\alpha)}$ is large with respect to the
coefficients of $D_{\alpha}(s)$ for all $s\in\mnfld$.
More precisely, $e^{\sprod{\lambda}{\Ht_0(a_\alpha)}}>\abs{r}$ for every root $r$ of $D_\alpha(s)$.
Recall that the map
\[
\funct^\lambda(\autspace)\rightarrow\oplus_{P\in\stds}\funct^\lambda(\Siegel_P),\ \
f\mapsto (f_P^{\Siegel})_P
\]
is a strict embedding. Hence, by Proposition \ref{prop: IP(s)} $\mu_s$ is a strict embedding for all $s\in\mnfld$.
On the other, by Lemma \ref{lem: GGPS}, $\nu_s$ is compact for all $s\in\mnfld$.
Thus, the system $\Xi^\lambda(s)$ satisfies the conditions of the Fredholm criterion
(Corollary \ref{cor: fred2}). In particular, it is locally of finite type.

%\begin{proposition}

%\end{proposition}

We can now finish the proof of Theorem \ref{thm: mainfin}. % from Proposition \ref{prop: finite}.
Let $\Xi_{\at}(s)$ be the system of equations on $\umd(\autspace)$ defined in \S\ref{sec: mainsys}.
Since the statement is local in $\mnfld$, we may indeed assume without loss of generality that
the coefficients of $D_{\alpha}(s)$ are bounded on $\mnfld$ for all $\alpha\in\srts_0$.
We may also assume that the system contains an equation of the form
\[
\delta(h(s))f=f
\]
where $h(s)$ is as above.

We first claim that the solutions of $\Xi_{\at}(s)$ are contained in $\funct^\lambda(\autspace)$, i.e.,
that $f^{\Siegel}\in \funct^\lambda(\Siegel)$ for any solution $f\in\Sol(\Xi_{\at}(s))$.
Let $f_P^{\Siegel}$, $P\in\stds$ be the components of $f^{\Siegel}$ with respect to  the Harish-Chandra decomposition \eqref{eq: orthd1}.
Then, $D_{\alpha}(s)(T_{a_\alpha})(f_P^{\Siegel})\equiv0$ for all $\alpha\in\srts_0\setminus\srts_0^P$.
On the other hand, since $f_P^{\Siegel}$ is of uniform moderate growth and cuspidal, its restriction to
$\Siegel_P^{\ball}$ is rapidly decreasing, and in particular belongs to $\funct^\lambda(\Siegel_P^{\ball})$.
By Proposition \ref{prop: IP(s)} we infer that $f_P^{\Siegel}\in \funct^\lambda(\Siegel_P)$.
Hence, $f\in \funct^\lambda(\Siegel)$ as claimed.

%By the lemma above, there exists $r_1>0$ such that
%$\Sol(\Xi_{\at}(s))\subset\functb^{r_1}(\autspace)$ for all $s\in \mnfld$.

It follows that the solutions of $\Xi_{\at}(s)$ are contained in those of $\Xi^\lambda(s)$.

By the local finiteness of $\Xi^\lambda(s)$ we may assume that there exists a finite-dimensional
vector space $L$ and an analytic family of injective operators $\gamma_s:L\rightarrow\funct^\lambda(\autspace)$, $s\in \mnfld$
such that $\Sol(\Xi^\lambda(s))\subset\gamma_s(L)$ for all $s\in \mnfld$.
It follows that $\Sol(\Xi_{\at}(s))$ is contained in the image of
the operator $\delta(h(s))\gamma_s:L\rightarrow\umd(\autspace)$, which depends analytically on $s$.
Now, it may happen that at our given point $s_0\in\mnfld$, $\delta(h(s_0))\gamma_{s_0}$ is not injective.
However, we claim that we can modify $h(s)$ to an analytic family $v(s)$ of compactly supported smooth functions
on $G(\A)$ such that $\delta(v(s))\gamma_s$ is injective near $s_0$ and $\delta(v(s))f=f$ whenever $\delta(h(s))f=f$
(so that $\Sol(\Xi_{\at}(s))$ is contained in $\delta(v(s))\gamma_s(L)$) for all $s\in\mnfld$.
Indeed, let $u$ be a smooth, nonnegative function on $G(\A)$ with total mass $1$ that is supported near the identity.
Then, $\delta(u)$ acts approximately as the identity on the finite-dimensional space $(I-\delta(h(s_0)))\gamma_{s_0}(L)$.
Let $v(s)=h(s)+u-u*h(s)=h(s)+u*(\delta_e-h(s))$.
Then, $\delta(v(s_0))\gamma_{s_0}$ is close to $\gamma_{s_0}$ and therefore injective,
while $\delta(v(s))f=f$ whenever $\delta(h(s))f=f$.
Thus, $\Xi_{\at}(s)$ is locally of finite type, as required.

\subsection{The case $\AAA_G\ne1$} \label{sec: AGne1}

Finally, we discuss the necessary (mild) modifications necessary to prove Theorem \ref{thm: mainfin} in the case where $\AAA_G\ne1$.

The main difference is that we do not consider the weight function $w_\lambda$ on all of $\zspace$ as before.
Instead, we split $\zspace$ into $2^r$ ``orthants'' $\zspace^{\underline{\epsilon}}$, $\underline{\epsilon}\in\{\pm1\}^r$
and on each one consider a different weight function $w_{\lambda_{\underline{\epsilon}}}$.

Formally, this is done as follows.
Recall that $z_1,\dots,z_r$ are fixed elements in $\AAA_G$ such that $\Ht_0(z_1),\dots,\Ht_0(z_r)$ is a basis of $\aaa_G$.
We write $\aaa_G$ as the (almost disjoint) union of the $2^r$ orthants $\aaa_G^{\underline{\epsilon}}$,
$\underline{\epsilon}=(\epsilon_1,\dots,\epsilon_r)\in\{\pm1\}^r$, namely the cones spanned by
$\epsilon_1\Ht_0(z_1),\dots,\epsilon_r\Ht_0(z_r)$.

Let $\autspace^{\underline{\epsilon}}$ be the inverse image of $\aaa_G^{\underline{\epsilon}}$
under $\Ht_G$. Similarly, define $\zspace^{\underline{\epsilon}}$, $\Siegel^{\underline{\epsilon}}$, etc.

Let $\underline{\lambda}$ be a family of vectors $\lambda_{\underline{\epsilon}}$, $\underline{\epsilon}\in\{\pm1\}^r$
in $\aaa_0^*$.
For every $\underline{\epsilon}\in\{\pm1\}^r$ we consider the weight function $w_{\lambda_{\underline{\epsilon}}}$ on $\zspace^{\underline{\epsilon}}$ as before.
The weighted $L^2$ space
\[
\funct^{\lambda_{\underline{\epsilon}}}(\Siegel^{\underline{\epsilon}})=L^2(\Siegel^{\underline{\epsilon}},
w_{\lambda_{\underline{\epsilon}}}^{-2}\ dx)
\]
gives rise to a Hilbertian space $\funct^{\lambda_{\underline{\epsilon}}}(\autspace^{\underline{\epsilon}})$.
Let $w_{\underline{\lambda}}$ be the weight function on $\zspace$ whose restriction to $\zspace^{\underline{\epsilon}}$
is $w_{\lambda_{\underline{\epsilon}}}$ for every $\underline{\epsilon}\in\{\pm1\}^r$. Then,
\[
\funct^{\underline{\lambda}}(\Siegel):=L^2(\Siegel,w_{\underline\lambda}^{-2}\ dx)=
\oplus_{\underline{\epsilon}\in\{\pm1\}^r}
\funct^{\lambda_{\underline{\epsilon}}}(\Siegel^{\underline{\epsilon}}).
\]
Correspondingly, we write
%\begin{equation} \label{eq: decompeps}
\[
\funct^{\underline{\lambda}}(\autspace)=\oplus_{\underline{\epsilon}\in\{\pm1\}^r}
\funct^{\lambda_{\underline{\epsilon}}}(\autspace^{\underline{\epsilon}}).
\]
%\end{equation}

For any monic polynomial $D$ of degree $m$ with nonzero constant coefficient $a_0$,
denote by $D^-$ the normalized upended polynomial $D^-=a_0^{-1}x^mD(x^{-1})$, which is also a monic polynomial of degree $m$.
For consistency, set $D^+=D$.
For every $\underline{\epsilon}\in\{\pm1\}^r$ and $i=1,\dots,r$, the operator $T_{z_i}^{\epsilon_i}$, and hence
also $D^{\epsilon_i}(T_{z_i}^{\epsilon_i})$, acts on
$\funct^{\lambda_{\underline{\epsilon}}}(\autspace^{\underline{\epsilon}})$.
%The direct sum of $D^{\epsilon_i}(T_{z_i}^{\epsilon_i})$ over $\epsilon\in\{\pm1\}^r$ defines an operator
%\[
%D^\#(T_{z_i}):\funct^\lambda(\autspace)\rightarrow\funct^\lambda(\autspace).
%\]

Now, for every $\alpha\in\srts_0$ let $D_\alpha(s)$, $s\in\mnfld$ be an analytic family of monic polynomials in one variable
of degree $m_{\alpha}$ whose coefficients are bounded on $\mnfld$.
For every $i=1,\dots,r$ let $\tilde D_i(s)$, $s\in\mnfld$ be an analytic family of monic polynomials in one variable of degree $\tilde m_i$
whose coefficients, as well as the inverse of the constant coefficient, are bounded on $\mnfld$.
Let $h(s)$, $s\in\mnfld$ be an analytic family of compactly supported functions on $G(\A)$.

Let $\eta_1,\dots,\eta_r$ be the dual basis of $\Ht_0(z_1),\dots,\Ht_0(z_r)$ in $\aaa_G^*$. For every
$\underline{\epsilon}\in\{\pm1\}^r$ and $P\in\stds$ let
\begin{align*}
\Siegel_P^{\underline{\epsilon},\ball}=\{g\in\Siegel_P^{\underline{\epsilon}}\mid& \sprod{\alpha}{\Ht_0(g)-T_0}\le m_\alpha\sprod{\alpha}{\Ht_0(a_\alpha)}
\text{ for every }\alpha\in\srts_0\setminus\srts_0^P,\text{ and,}\\
&\epsilon_i\sprod{\eta_i}{\Ht_G(g)}\le\tilde m_i\text{ for every }i=1,\dots,r\}.
\end{align*}

The system of equations $\Xi^\lambda(s)$ on $f\in\funct^{\underline{\lambda}}(\autspace)$ is given by
\begin{gather*}
f_P^{\Siegel^{\underline{\epsilon}}}\rest_{\Siegel_P^{\underline{\epsilon},\ball}}=
(\delta(h(s))f)_P^{\Siegel^{\underline{\epsilon}}}\rest_{\Siegel_P^{\underline{\epsilon},\ball}}\ \ \text{ for all }P\in\stds,\\
D_\alpha(s)(T_{a_\alpha})f_P^{\Siegel^{\underline{\epsilon}}}=0\text{ for all }P\in\stds\text{ and }\alpha\in\srts_0\setminus\srts_0^P,\\
\tilde D_i(s)^{\epsilon_i}(T_{z_i}^{\epsilon_i})f^{\underline{\epsilon}}=0\text{ for all }i=1,\dots,r,
\end{gather*}
for every $\underline{\epsilon}\in\{\pm1\}^r$.
The same argument as in the case $\AAA_G=1$ shows that the system $\Xi^\lambda(s)$ is of Fredholm type
provided that the following conditions are satisfied for every $\underline{\epsilon}\in\{\pm1\}^r$.
\begin{enumerate}
\item $e^{\sprod{\lambda^{\underline{\epsilon}}}{\Ht_0(a_\alpha)}}>\abs{x}$ for every root $x$ of $D_\alpha(s)$ for every $\alpha\in\srts_0$ and $s\in\mnfld$, and,
\item $e^{\epsilon_i\sprod{\lambda^{\underline{\epsilon}}}{\Ht_0(z_i)}}>\abs{x}$ for every root $x$ of $\tilde D_i^{\epsilon_i}(s)$ for every $i=1,\dots,r$ and $s\in\mnfld$.
\end{enumerate}
As before, we conclude that $\Xi^{\at}(s)$ is locally of finite type.
%As before, the local strong finiteness of $\Xi^\lambda(s)$ immediately implies that of $\Xi^{\at}(s)$.

\section{Conclusion of proof -- the number field case} \label{sec: conclusion}

We continue to assume that $F$ is a number field.
Combining the uniqueness result of \S\ref{sec: unique} and the local finiteness result of
\S\ref{sec: LF} we deduce the main theorem using the principle of meromorphic continuation.

The system of linear equations is described in \S\ref{sec: equations} and is motivated by
Proposition \ref{prop: unique} (or more precisely, Corollary \ref{cor: uniq}).
It also incorporates a basic result of Harish-Chandra (\S\ref{sec: HC}) which is the fulcrum for the local finiteness
provided by Theorem \ref{thm: mainfin}.
The meromorphic continuation of the intertwining operators
is then deduced from that of the Eisenstein series by a standard argument -- see \S\ref{sec: MIO}.
The functional equations are also an immediate consequence of Proposition \ref{prop: unique}.
Finally, the analysis of the singularities boils down to the case of a maximal parabolic subgroup.

\subsection{Characterization of automorphic forms} \label{sec: conc1}

Let $P\in\stds$.
Denote by $\tilde\AF_P^{\cusp}$ \index{AFPcusptilde@$\tilde\AF_P^{\cusp}$, $(\tilde\AF_P^{\cusp})^\perp$} the linear
span of the functions of the form $(f\circ\Ht_P) \cdot\varphi$ where $f\in C_c^\infty(\aaa_P)$ and $\varphi\in\AF_P^{\cusp}$.
It consists of rapidly decreasing functions.
Denote by $(\tilde\AF_P^{\cusp})^\perp$ the annihilator of $\tilde\AF_P^{\cusp}$ in $\umd(\autspace_P)$ with respect
to the sesquilinear form $(\cdot,\cdot)_{\autspace_P}$ given by \eqref{eq: innerP}.

By \cite{MR1361168}*{I.3.4}, if $\phi\in\umd(\autspace)$ and $\cnst_{G,P}\phi\in(\tilde\AF_P^{\cusp})^\perp$ for all $P\in\stds$
then $\phi\equiv0$.
More generally, for any $P\in\stds$,
\begin{equation} \label{eq: cn0}
\text{if $\phi\in\umd(\autspace_P)$ and $\cnst_{P,Q}\phi\in(\tilde\AF_Q^{\cusp})^\perp$ for all $Q\subset P$
then $\phi\equiv0$.}
\end{equation}

For any $\lambda\in\chars_P$ and $a\in\AAA_P$ consider the difference operator
\index{Diff3@$\diff_a^{P,\lambda}$, $\diff_{\underline{a}}^{P,\underline{\lambda}}$, $\diff_a^{P,\underline{\lambda}}$}
\[
\diff_a^{P,\lambda}\varphi=a\cdot \varphi-a^{\lambda}\varphi
\]
on functions on $\autspace_P$.
We use the same notation also if $\lambda\in\chars_Q$ where $Q\subset P$
(in which case, it depends only on the projection of $\lambda$ to $\chars_P$).
The operators $\diff_a^{P,\lambda}$, $a\in\AAA_P$, $\lambda\in\chars_P$
pairwise commute. More generally, for $\underline{\lambda}=(\lambda_1,\dots,\lambda_n)\in\chars_P^n$
(or more generally in $\chars_Q^n$ where $Q\subset P$) and
$\underline{a}=(a_1,\dots,a_n)\in\AAA_P^n$ we write
\[
\diff_{\underline{a}}^{P,\underline{\lambda}}=\prod_{i=1}^n\diff_{a_i}^{P,\lambda_i}.
\]
(If $a_i=a$ for all $i$, then we simply write $\diff_a^{P,\underline{\lambda}}$.)

Let $\z$ be the center of $\univ(\gggg)$. If $P=M\ltimes U\in\stds$ we will write $\z^M$ for the corresponding object for $M$.

\begin{lemma} \label{lem: autcusp}
Let $\phi\in\umd(\autspace)$.
Then, $\phi\in\AF_G$ if and only if the following two conditions are satisfied.
\begin{enumerate}
\item There exists a smooth, compactly supported, bi-$\K$-finite function $h$ on $G(\A)$ such that $\delta(h)\phi=\phi$.
\item For every $P\in\stds$ there exist an integer $n\ge0$ and $\underline{\lambda}\in\chars_P^n$
such that $\diff_{\underline{a}}^{P,\underline{\lambda}}\cnst_{G,P}\phi\in(\tilde\AF_P^{\cusp})^\perp$
for all $\underline{a}\in\AAA_P^n$.
\end{enumerate}
\end{lemma}

\begin{proof}
The ``only if'' direction follows from \cite{MR1361168}*{I.2.17 and I.3.1}.

Conversely, let $h$ be a smooth, compactly supported, bi-$\K$-finite function on $G(\A)$
and for every $P\in\stds$ let $n_P\ge0$ be an integer and $\underline{\lambda}_P\in\chars_P^{n_P}$.
Consider the linear space
\[
V=\{\phi\in\umd(\autspace)\mid\delta(h)\phi=\phi\text{ and }\diff_{\underline{a}}^{P,\underline{\lambda}_P}\cnst_{G,P}\phi
\in(\tilde\AF_P^{\cusp})^\perp\ \forall P\in\stds, \forall\underline{a}\in\AAA_P^{n_P}\}.
\]
By \eqref{eq: cn0}, for any $\phi\in V$ we have
\[
\big(\prod_{Q\subset P}\diff_{\underline{a}_Q}^{P,\underline{\lambda}_Q}\big)\cnst_{G,P}\phi\equiv0
\]
for any $P\in\stds$ and any collection $\underline{a}_Q\in\AAA_Q^{n_Q}$, $Q\subset P$. (Cf.\ \cite{MR1361168}*{I.3.5}.)
Hence $V$ is finite-dimensional by Theorem \ref{thm: mainfin}. Clearly, $V$ is $\z$-invariant.
Therefore, any $\phi\in V$ is $\z$-finite, and of course also $\K$-finite (since $h$ is bi-$\K$-finite).
Thus, $\phi$ is an automorphic form.
\end{proof}

\subsection{}

The following is a standard consequence of Harish-Chandra's finiteness theorem.
\begin{lemma} \label{lem: finio}
For any $\varphi\in\AF_P$ and $w\in\Weyl(P,Q)$ the automorphic forms $M(w,\lambda)\varphi$,
whenever defined, belong to a finite-dimensional linear subspace of $\AF_Q$ (independently of $\lambda$).
\end{lemma}

\begin{proof}
Let $\idem$ be an idempotent in the algebra of finite functions on $\K$ such that $\delta(\idem)\varphi=\varphi$.
Then, $\delta(\idem)(M(w,\lambda)\varphi)=M(w,\lambda)\varphi$ whenever $M(w,\lambda)\varphi$ is defined.
Let $I$ be a finite-codimensional ideal of $\z^M$ that annihilates the function
$m\in M(\A)\mapsto\modulus_P(m)^{-\frac12}\varphi(mg)$ for all $g\in G(\A)$ \cite{MR1361168}*{I.2.17}.
By \eqref{eq: monm}, the finite-codimensional ideal $wI$ of $\z^{M_Q}$ annihilates the function
$m\in M_Q(\A)\mapsto\modulus_Q(m)^{-\frac12}(M(w,\lambda)\varphi)(mg)$ for all $g\in G(\A)$.
It follows that there exists a finite-codimensional ideal of $\z$ (independent of $\lambda$) that annihilates
$M(w,\lambda)\varphi$ [loc.\ cit.]. We conclude the lemma by Harish-Chandra's finiteness theorem \cite{MR0232893}*{Theorem 1}.
\end{proof}

\subsection{A result of Harish-Chandra} \label{sec: HC}

By a basic result of Harish-Chandra, every automorphic form on $\autspace_P$ is an eigenfunction of $\delta(h)$
with eigenvalue $1$ for some $h\in C_c^\infty(G(\A))$ \cite{MR0219666}*{\S8}.\footnote{As noted there,
the argument was simplified by Jacquet and Borel.}
We need a uniform version of this fact, for families of automorphic forms, as follows.

\begin{lemma} (See \cite{MR1361168}*{I.4.5}) \label{lem: HC}
For every $P\in\stds$ let $n_P\ge0$ be an integer and $V_P$ a finite-dimensional subspace of $\AF_P^{\cusp}$.
Let $\Lambda=\oplus_{P\in\stds}\chars_P^{n_P}$.
Then, there exists an integer $d\ge0$ and for any $\underline{\lambda}_0\in\Lambda$ there exist
analytic functions $b_0,\dots,b_d:\Lambda\rightarrow\C$ such that $b_0(\underline{\lambda}_0)\ne0$, and
a bi-$\K$-finite function $h\in C_c^\infty(G(\A))$, with the following property.
Suppose that $\phi\in\AF_G$ and $\underline{\lambda}=(\lambda_1^P,\dots,\lambda_{n_P}^P)_{P\in\stds}\in\Lambda$
are such that for every $P\in\stds$, $\cnst_{G,P}^{\cusp}\phi$ is
of the form $\cnst_{G,P}^{\cusp}\phi=\sum_{i=1}^{n_P}(\varphi_i)_{\lambda_i^P}$ for some $\varphi_1,\dots,\varphi_{n_P}\in V_P$.
Then,
\[
\sum_{i=1}^db_i(\underline{\lambda})\delta(h)^i\phi=b_0(\underline{\lambda})\phi.
\]
\end{lemma}

\begin{remark}
Let $C$ be a compact, bi-$\K$-invariant neighborhood of $1$ in $G(\A)$.
Then, in the lemma above we may choose $h$ to be supported in $C$
(and $d$ is independent of $C$).
\end{remark}

\subsection{The system of linear equations} \label{sec: equations}

Fix $P\in\stds$ and $\varphi\in\AF_P$.
We now describe a holomorphic system $\Xi_{\fnl}(\lambda)$, \index{Xifnlambda@$\Xi_{\fnl}(\lambda)$}
$\lambda\in\chars_P$ of linear equations on $\psi\in\umd(\autspace)$ that is locally of finite type, and
that admits $E(\varphi,\lambda)$ as the unique solution for $\sprod{\Re\lambda}{\alpha^\vee}\gg0$ $\forall\alpha\in\srts_P$.

For any $Q\subset P$ let $m_Q\ge0$ be an integer and $\underline{\mu}_Q\in\chars_Q^{m_Q}$
be such that $\cnst_{P,Q}^{\cusp}\varphi\in\AF_Q^{\cusp}(\underline{\mu}_Q)$.

For any $Q\in\stds$ let
\[
n_Q=\sum_{w\in\Weyl^{\supset Q}(P)}m_{P_w},\ Q\in\stds,\\
\]
and
\[
\underline{\lambda}=(\vee_{w\in\Weyl^{\supset Q}(P)}w(\underline{\mu}_{P_w}+\lambda))_{Q\in\stds}
\in\oplus_{Q\in\stds}\chars_Q^{n_Q}.
\]
Using Lemma \ref{lem: finio} and the formula \eqref{eq: consterm} we may apply Lemma \ref{lem: HC}
to conclude that there exists a family of pairs $(h_i,c_i)$, $i\in I$ consisting of
a holomorphic family $h_i(\lambda)$, $\lambda\in\chars_P$ of smooth, compactly supported bi-$\K$-finite functions on $G(\A)$
and a holomorphic function $c_i:\chars_P\rightarrow\C$, such that
\begin{enumerate}
\item $\delta(h_i(\lambda))E(\varphi,\lambda)=c_i(\lambda)E(\varphi,\lambda)$ for all $i\in I$
provided that $\sprod{\Re\lambda}{\alpha^\vee}\gg0$ for all $\alpha\in\srts_P$.
\item For any $\lambda_0\in\chars_P$ there exists $i\in I$ such that $c_i(\lambda_0)\ne0$.
\end{enumerate}

\begin{subequations}
The system $\Xi_{\fnl}(\lambda)$ consists of the homogeneous set of equations
\begin{equation} \label{eq: HCeis}
\delta(h_i(\lambda))\psi=c_i(\lambda)\psi,\ \ i\in I
\end{equation}
and the non-homogeneous set of equations
\begin{equation} \label{eq: nonhom}
\prod_{w\in\Weyl^{\supset Q}(P)\setminus\{e\}}\diff^{Q,w(\underline{\mu}_{P_w}+\lambda)}_{\underline{a}_w}
(\cnst_{G,Q}\psi-\cnst_{P,Q}\varphi_\lambda)\in(\tilde\AF_Q^{\cusp})^\perp
\end{equation}
for any $Q\in\stds$ and any collection $\underline{a}_w\in\AAA_Q^{m_{P_w}}$, $w\in\Weyl^{\supset Q}(P)\setminus\{e\}$.
\end{subequations}

\begin{proposition} \label{prop: Ximain}
The system $\Xi_{\fnl}(\lambda)$ is holomorphic and locally of finite type.
In the region $\sprod{\Re\lambda}{\alpha^\vee}\gg0$ $\forall\alpha\in\srts_P$, it admits
$\psi=E(\varphi,\lambda)$ as its unique solution.
\end{proposition}

\begin{proof}
The system $\Xi_{\fnl}(\lambda)$ is clearly holomorphic.
In the region $\sprod{\Re\lambda}{\alpha^\vee}\gg0$ $\forall\alpha\in\srts_P$, the Eisenstein series
$\psi=E(\varphi,\lambda)$ satisfies \eqref{eq: HCeis} by the choice of $h_i(\lambda)$ and $c_i(\lambda)$,
while the equations \eqref{eq: nonhom} follow from Corollary \ref{cor: uniq}.

Conversely, suppose that $\psi\in\Sol(\Xi_{\fnl}(\lambda))$ for some $\lambda\in\chars_P$.
The equations \eqref{eq: nonhom} imply that for every $Q\in\stds$ we have
\begin{equation} \label{eq: allw}
\prod_{w\in\Weyl^{\supset Q}(P)}
\diff^{Q,w(\underline{\mu}_{P_w}+\lambda)}_{\underline{a}_w}(\cnst_{G,Q}\psi)\in(\tilde\AF_Q^{\cusp})^\perp
\end{equation}
for any collection $\underline{a}_w\in\AAA_Q^{m_{P_w}}$, $w\in\Weyl^{\supset Q}(P)$.
Let $i\in I$ be such that $c_i(\lambda)\ne0$.
Lemma \ref{lem: autcusp} and equations \eqref{eq: HCeis} and \eqref{eq: allw} imply that $\psi$ is an automorphic form.
Hence, the equation \eqref{eq: nonhom} is now equivalent to \eqref{eq: cuspe}.
It follows from \eqref{eq: Eunique} that $\psi=E(\varphi,\lambda)$ provided that $\sprod{\Re\lambda}{\alpha^\vee}\gg0$
for all $\alpha\in\srts_P$.

Finally, using \eqref{eq: allw} and the argument in the proof of Lemma \ref{lem: autcusp}, the system $\Xi_{\fnl}(\lambda)$
implies a system of the form considered in \S\ref{sec: mainsys}.
Hence, it is locally of finite type by Theorem \ref{thm: mainfin}.
\end{proof}

We can therefore invoke the principle of meromorphic continuation (Theorem \ref{thm: anal}) to the system $\Xi_{\fnl}(\lambda)$
to conclude the meromorphic continuation of $E(\varphi,\lambda)$, i.e., the first part of Theorem \ref{thm: main}.

\begin{comment}
\begin{remark}
We have not proved that the system $\Xi_{\fnl}(\lambda)$ is locally, strongly of finite type
(although we expect it to hold -- See Remark \ref{rem: stlcf}).
Nevertheless, we can still conclude that the Eisenstein series is holomorphic on the set of uniqueness of
$\Xi_{\fnl}(\lambda)$, which is open.
Indeed, as in \S\ref{sec: auxsys} and \S\ref{sec: reducemain} locally,
we can  represent $\Xi_{\fnl}(\lambda)$ as a system $\Xi'(\lambda)$ on the space $\funct^\lambda(\Siegel)$ which is
strongly of finite type. The sets of uniqueness of $\Xi_{\fnl}(\lambda)$ and $\Xi'(\lambda)$ coincide
and there is a holomorphic family of operators $\funct^\lambda(\Siegel)\rightarrow\umd(\autspace)$
that maps $\Sol(\Xi'(\lambda))$ to $\Sol(\Xi_{\fnl}(\lambda))$ for $\lambda$ in the uniqueness set.
\end{remark}
\end{comment}

\subsection{Meromorphic continuation of intertwining operators} \label{sec: MIO}

Next, we deduce the meromorphic continuation of the intertwining operators from that of the Eisenstein series.
Fix $P\in\stds$ and $w\in\Weyl(P,Q)$.
\begin{lemma}
Suppose that $a\in\AAA_Q$ is regular, i.e., $\sprod{\beta}{\Ht_Q(a)}\ne0$
for all $\beta\in\Phi_Q$. Then,
$w'^{-1}\Ht_Q(a)\notin w^{-1}\Ht_Q(a)+\aaa_0^P$ for any $w'\in\,_Q\Weyl_P\setminus\{w\}$.
\end{lemma}

\begin{proof}
Given $Q_1\in\stds$ and $w_1\in\Weyl(Q,Q_1)$ the validity of statement of the lemma does not change if we
replace $Q$ by $Q_1$, $a$ by $waw^{-1}$ and $w$ by $w_1w$. Likewise,
given $P_1\in\stds$ and $w_1\in\Weyl(P_1,P)$ the validity of statement of the lemma does not change if we
replace $P$ by $P_1$ and $w$ by $ww_1$.
Therefore, by \cite{MR1361168}*{I.1.10} we may assume that $Q=P$, $w=e$ and $\Ht_P(a)$ is in the positive Weyl chamber of $\aaa_P$.
In this case, it is well known that for any $w'\in\Weyl$ the coefficients of $w'^{-1}\Ht_P(a)-\Ht_P(a)$ in the basis $\Delta_0^\vee$
are non-positive, and the $\beta^\vee$-th coefficient is negative whenever $\beta\in\srts_0\setminus\srts_0^P$ and $w'\beta<0$.
In particular, $w'^{-1}\Ht_P(a)-\Ht_P(a)\notin\aaa_0^P$ if $w'\notin\,_P\Weyl_P\setminus\{e\}$
(and in fact, for any $w'\notin\Weyl^P$ since this property is $\Weyl^P$-bi-invariant in $w'$).
\end{proof}

Now let $\varphi\in\AF_P$.
For any $P'\subset P$ let $m_{P'}\ge0$ be an integer and $\underline{\mu}_{P'}\in\chars_{P'}^{m_{P'}}$
be such that $\cnst_{P,P'}\varphi\in\AF_{P'}(\underline{\mu}_{P'})$.

Consider the constant term $\cnst_{G,Q}E(\varphi,\lambda)$ given by \eqref{eq: EserCT}.
For any $w'\in\,_Q\Weyl_P$ we have
\[
E^Q(M(w',\lambda)(\cnst_{P,P_{w'}}\varphi),w'\lambda)\in\AF_Q(w'(\underline{\mu}_{P_{w'}}+\lambda)).
\]
Fix a regular element $a$ of $\AAA_Q$ and consider the difference operator
\[
D(\lambda)=\prod_{w'\in\,_Q\Weyl_P\setminus\{w\}}
\diff^{Q,w'(\underline{\mu}_{P_{w'}}+\lambda)-w\lambda}_a.
\]
Then, for any $w'\in\,_Q\Weyl_P\setminus\{w\}$
\[
D(\lambda) (E^Q(M(w',\lambda)(\cnst_{P,P_{w'}}\varphi),w'\lambda)_{-w\lambda})\equiv0
\]
and hence by \eqref{eq: EserCT} we have
\[
D(\lambda)((\cnst_{G,Q}E(\varphi,\lambda))_{-w\lambda})=D(\lambda)(M(w,\lambda)\varphi).
\]
On the other hand, by the lemma above,
$w'^{-1}\Ht_Q(a)\notin w^{-1}\Ht_Q(a)+\aaa_0^P$ for every $w'\in\,_Q\Weyl_P\setminus\{w\}$.
Let $f(x)$ be the polynomial
\[
f(x)=\prod_{i=1}^{m_P}(x-a^{w((\underline{\mu}_P)_i)})
\]
and let
\[
g(\lambda,x)=\prod_{w'\in\,_Q\Weyl_P\setminus\{w\}}
\prod_{i=1}^{m_{P_{w'}}}(x-a^{w'((\underline{\mu}_{P_{w'}})_i+\lambda)-w\lambda})
\]
which is a polynomial in $x$ whose coefficients are analytic functions in $\lambda\in\chars_P$.
The polynomials $f$ and $g(\lambda,\cdot)$ are coprime for generic $\lambda$.
Let $R(\lambda,x)$ be their resultant, which is a non-zero polynomial in $x$
whose coefficients are analytic functions in $\lambda\in\chars_P$.
Let $T_a$ be the linear transformation $\varphi\mapsto a\cdot\varphi$ on $\AF_Q(w\underline{\mu}_P)$.
Then, $f(T_a)=0$ and the restriction $\tilde D(\lambda)$ of $D(\lambda)$ to $\AF_Q(w\underline{\mu}_P)$ is
$g(\lambda,T_a)$. Hence, $\tilde D(\lambda)$ is invertible for generic $\lambda\in\chars_P$
and $R(\lambda,T_a)\tilde D(\lambda)^{-1}$ extends to a holomorphic function on $\chars_P$.
Thus,
\[
M(w,\lambda)\varphi=\tilde D(\lambda)^{-1}(D(\lambda)((\cnst_{G,Q}E(\varphi,\lambda))_{-w\lambda}))
\]
and this provides meromorphic continuation of $M(w,\lambda)\varphi$.

\begin{remark}
For $Q=P$ and $w=e$, the argument above shows that on any finite-dimensional linear subspace of $\AF_P$
\begin{equation} \label{eq: eisinj}
\text{the operator $\varphi\mapsto E(\varphi,\lambda)$ is injective for generic $\lambda\in\chars_P$.}
\end{equation}
\end{remark}

\subsection{Functional equations and singularities}

Let $w'\in\Weyl(P,P')$ and suppose that $\sprod{w'\Re\lambda}{\alpha^\vee}\gg0$ for all $\alpha\in\srts_{P'}$.
Then, on the one hand, by Proposition \ref{prop: unique}, we have
\[
\lead(E(M(w',\lambda)\varphi,w'\lambda))=\lead_{P'}((M(w',\lambda)\varphi)_{w'\lambda}).
\]
On the other hand, by Lemma \ref{lem: far} (applied to $w'\lambda$ and $ww'^{-1}\in\Weyl^{\supset Q}(P')$)
and \eqref{eq: cuspew} we have
\[
\lead(E(\varphi,\lambda))=\lead_{P'}((M(w',\lambda)\varphi)_{w'\lambda}).
\]
Therefore, the functional equation $E(M(w',\lambda)\varphi,w'\lambda)=E(\varphi,\lambda)$
follows from Theorem \ref{thm: lead}.

Moreover, if $w\in\Weyl(P,P')$ and $w'\in\Weyl(P',P'')$, then
\[
E(M(w'w,\lambda)\varphi,w'w\lambda)=E(\varphi,\lambda)=E(M(w,\lambda)\varphi,w\lambda)=
E(M(w',w\lambda)M(w,\lambda)\varphi,w'w\lambda).
\]
Thus, $M(w'w,\lambda)=M(w',w\lambda)M(w,\lambda)$ by \eqref{eq: eisinj}.

Consider now the singularities of $M(w,\lambda)$.
If $w$ is an elementary symmetry $s_\alpha$ for some $\alpha\in\srts_P$,
then as a function of $\lambda\in\chars_P$, $M(w,\lambda)$ depends only on $\sprod{\lambda}{\alpha^\vee}$.
In general, by decomposing $w$ into elementary symmetries \cite{MR1361168}*{I.1.8} and using the multiplicativity of
intertwining operators (cf.\ \cite{MR1361168}*{II.1.6}), it follows that
the singularities of $M(w,\lambda)$ are of the form $\sprod{\lambda}{\beta^\vee}=c$ for some
$\beta\in\Phi_P$ such that $w\beta<0$ and $c\in\C$.

On the other hand, the singularities of $E(\varphi,\lambda)$ are precisely those of its cuspidal components
\cite{MR1361168}*{I.4.10}. It follows from \eqref{eq: consterm} that the singularities of $E(\varphi,\lambda)$ are also
along root hyperplanes.\footnote{Note however, that this argument by itself does not imply that $E(\varphi,\lambda)$ is holomorphic
on $\iii\aaa_P^*$ in case $\varphi\in\AF_P^2$ since the cuspidal components of $E(\varphi,\lambda)$ involve intertwining operators
applied to the cuspidal components of $\varphi$ rather than $\varphi$ itself.}

This finishes the proof of Theorem \ref{thm: main}.

We remark that the proof shows the following (ostensibly) slightly stronger statement.
\begin{corollary}
For any bounded open subset $U$ of $\chars_P$,
there exists $r>0$ such that $\lambda\mapsto E(\varphi,\lambda)$
is a meromorphic function on $U$, with singularities along finitely many root hyperplanes, into
the Fr\'echet space $\functb^R_{\smth}(\autspace)$.
In other words, for any $\lambda_0\in\chars_P$ there exist $R>0$ and an integer $k\ge0$ such that the function
\[
\big(\prod_{\beta\in\Phi_P}\sprod{\lambda-\lambda_0}{\beta^\vee}^k\big)E(\varphi,\lambda)
\]
admits a convergent power series expansion in $\functb^R_{\smth}(\autspace)$ around $\lambda_0$,
\end{corollary}

\section{The function field case} \label{sec: conclusion2}

In this section we prove the main result in the function field case.
The argument is analogous to the number field case but it is simpler since we don't need any analysis
(in particular, local finiteness).
Throughout this section $F$ is a function field.

\subsection{Characterization of automorphic forms} \label{sec: conc2}
Let $P\in\stds$. In the number field case we considered the space $\umd(\autspace_P)$ of smooth functions
of uniform moderate growth. In the function field case we consider instead the space $\umdf(\autspace_P)$
of all functions on $\autspace_P$ that are right-invariant under some open subgroup of $G(\A)$.
Let $\rapid(\autspace_P)$ be the subspace of compactly supported locally constant functions.
We can identify $\umdf(\autspace_P)$ with the smooth part of the conjugate dual of $\rapid(\autspace_P)$
by the sesquilinear pairing $(\cdot,\cdot)_{\autspace_P}$ \eqref{eq: innerP}.

Analogously, we denote by $\tilde\AF_P^{\cusp}$ the linear span of the
functions of the form $(f\circ\Ht_P) \cdot\varphi$ where $f\in C_c(\Ht_M(M(\A)))$ and
$\varphi\in\AF_P^{\cusp}$ . (Recall that $\Ht_M(M(\A))$ is a lattice in $\aaa_P$.)
By \cite{MR1361168}*{I.2.9}, $\tilde\AF_P^{\cusp}\subset\rapid(\autspace_P)$.
Denote by $(\tilde\AF_P^{\cusp})^\perp\subset\umdf(\autspace_P)$ its annihilator with respect to $(\cdot,\cdot)_{\autspace_P}$. We have a direct sum decomposition
\begin{equation} \label{eq: dscnc}
\umdf(\autspace_P)=\umdf_{\cusp}(\autspace_P)\oplus(\tilde\AF_P^{\cusp})^\perp
\end{equation}
where
\[
\umdf_{\cusp}(\autspace_P)=\{\phi\in\umdf(\autspace_P)\mid\cnst_{P,Q}\phi\equiv0\text{ for all }Q\subsetneq P\}.
\]
We denote by $\phi\mapsto\phi^{\cusp}$ the projection $\umdf(\autspace_P)\rightarrow\umdf_{\cusp}(\autspace_P)$
with respect to \eqref{eq: dscnc}.
As usual, we write $\cnst_{G,P}^{\cusp}\phi=(\cnst_{G,P}\phi)^{\cusp}$ for any $\phi\in\umdf(\autspace)$.
By the argument of \cite{MR1361168}*{I.3.4},
\begin{equation} \label{eq: cncn}
\text{a function $\phi\in\umdf(\autspace)$ is identically $0$ if and only if
$\cnst_{G,P}^{\cusp}\phi\equiv0$ for all $P\in\stds$.}
\end{equation}

%cusp forms are compactly supported modulo the center in the function field case \cite{MR1361168}*{I.2.9}.
%that are right invariant under some open subgroup of $G(\A)$ (smooth functions).

For any $\lambda\in\charsA_P$ and $a\in\AAA_P$ consider the difference operator
\[
\diff_a^{P,\lambda}\varphi=a\cdot \varphi-a^\lambda\varphi
\]
on functions on $\autspace_P$.
These operators commute.
For $\underline{\lambda}=(\lambda_1,\dots,\lambda_n)\in\charsA_P^n$ and $\underline{a}=(a_1,\dots,a_n)\in\AAA_P^n$ we write
\[
\diff_{\underline{a}}^{P,\underline{\lambda}}=\prod_{i=1}^n\diff_{a_i}^{P,\lambda_i}.
\]

By the argument of \cite{MR1361168}*{I.3.6} a function $\phi\in\umdf_{\cusp}(\autspace_P)$ belongs to $\AF_P^{\cusp}$
if and only if there exists $n\ge0$ and $\underline{\lambda}\in\charsA_P^n$ such that
\[
\diff_{\underline{a}}^{P,\underline{\lambda}}\phi\equiv0\text{ for all }\underline{a}\in\AAA_P^n.
\]

Using \cite{MR1361168}*{I.3.5} we infer
\begin{lemma} \label{lem: autcusp2}
Let $\phi\in\umdf(\autspace)$. % be a $\K$-finite function of uniform moderate growth on $\autspace$.
Then, $\phi\in\AF_G$ if and only if for every $P\in\stds$ there exist an integer $n\ge0$ and $\underline{\lambda}\in\charsA_P^n$
such that $\diff_{\underline{a}}^{P,\underline{\lambda}}(\cnst_{G,P}\phi)\in(\tilde\AF_P^{\cusp})^\perp$
for all $\underline{a}\in\AAA_P^n$.
\end{lemma}

\subsection{}
For any $\phi\in \rapid(\autspace_P)$ and $g\in G(\A)$ the sum
\[
\sum_{\gamma\in P(F)\bs G(F)}\phi(\gamma g)
\]
has only finitely many non-zero terms and it gives rise to a $G(\A)$-equivariant linear map \index{thetaP@$\theta_P$}
\[
\theta_P:\rapid(\autspace_P)\rightarrow \rapid(\autspace)
\]
whose dual is the constant term map
\[
\cnst_{G,P}:\umdf(\autspace)\rightarrow \umdf(\autspace_P).
\]
Moreover, by the argument of \cite{MR1361168}*{II.1.12}
\[
\rapid(\autspace)=\sum_{P\in\stds}\theta_P(\tilde\AF_P^{\cusp}).
\]
In fact, this is just an equivalent formulation of \eqref{eq: cncn}.

For any compact open subgroup $K$ of $G(\A)$ denote by $\rapid(\autspace)^K$ the space of compactly supported
right $K$-invariant functions on $\autspace$, i.e., the $K$-fixed part of $\rapid(\autspace)$, and by
$\tilde\AF_P^{\cusp,K}$ the $K$-fixed part of $\tilde\AF_P^{\cusp}$. Then,
\begin{equation} \label{eq: pesudo}
\rapid(\autspace)^K=\sum_{P\in\stds}\theta_P(\tilde\AF_P^{\cusp,K}).
\end{equation}

We will need another simple fact.

\begin{lemma}
For any $P\in\stds$, an integer $n\ge0$ and a compact open subgroup $K$ of $G(\A)$ there
exists a finite subset $B\subset G(\A)$ such that for any $\underline{\lambda}\in\charsA_P^n$,
the restriction map $\phi\mapsto\phi\rest_B$ is injective on $\AF_P^{\cusp}(\underline{\lambda})^K$.
Dually, there exists a finite-dimensional subspace $L$ of $\tilde\AF_P^{\cusp,K}$ such that
for any $\underline{\lambda}\in\charsA_P^n$ we have
\[
\tilde\AF_P^{\cusp,K}=L+\sum_{\underline{a}\in\AAA_P^n}\diff_{\underline{a}}^{P,\underline{\lambda}}
(\tilde\AF_P^{\cusp,K}).
\]
\end{lemma}

Indeed, there exists a finite set $B_1\subset G(\A)$ such that the support of any right $K$-invariant function in $\AF_P^{\cusp}$
(or in $\tilde\AF_P^{\cusp}$)
is contained in $U(\A)M(F)\AAA_PB_1K$. (This easily follows from \cite{MR1361168}*{I.2.9} applied to $M$.)
Hence, the first part of the lemma reduces to the analogous statement about functions on the lattice $\AAA_P$, which is elementary. (We can take $B$ to be $B_2B_1$ for a suitable finite subset $B_2$ of $\AAA_P$.)
For the dual statement, note that the algebraic dual of $\tilde\AF_P^{\cusp,K}$ is $\umdf_{\cusp}(\autspace_P)^K$
and the annihilator of
\[
\sum_{\underline{a}\in\AAA_P^n}\diff_{\underline{a}}^{P,\underline{\lambda}}(\tilde\AF_P^{\cusp,K})
\]
in $\umdf_{\cusp}(\autspace_P)^K$ is $\AF_P^{\cusp}(\underline{\lambda})^K$.
Hence, we can take $L$ to be the projection under $\phi\mapsto\phi^{\cusp}$ of the space of
right $K$-invariant functions on $\autspace_P$ supported in $U(\A)M(F)BK$.

By \eqref{eq: pesudo}, we conclude

\begin{corollary} \label{cor: commonU}
For any compact open subgroup $K$ of $G(\A)$ and an integer $n\ge0$ there exists a finite-dimensional subspace $U$ of
$\rapid(\autspace)^K$ such that
\[
\rapid(\autspace)^K=U+\sum_{P\in\stds,\underline{a}\in\AAA_P^n}\theta_P(\diff_{\underline{a}}^{P,\underline{\lambda}}
(\tilde\AF_P^{\cusp,K}))
\]
for any $\underline{\lambda}\in\charsA_P^n$.
\end{corollary}

\subsection{The system of linear equations} \label{sec: equations2}

Fix $P\in\stds$ and $\varphi\in\AF_P$.

For any $Q\subset P$ fix $m_Q\ge0$ and $\underline{\mu}_Q^{\cusp}\in\charsA_Q^{m_Q}$ such that
$\cnst_{P,Q}^{\cusp}\varphi\in\AF_Q(\underline{\mu}_Q^{\cusp})$.
As usual, for any $\lambda\in\chars_P$, denote by $\tilde\lambda$ its image under the projection map \eqref{eq: restA}.

\begin{proposition} \label{prop: Ximain2}
In the region $\sprod{\Re\lambda}{\alpha^\vee}\gg0$ $\forall\alpha\in\srts_P$, the Eisenstein series
$E(\varphi,\lambda)$ is the unique function $\psi\in\umdf(\autspace)$ satisfying the linear equations
\begin{equation} \label{eq: nonhom2}
\prod_{w\in\Weyl^{\supset Q}(P)\setminus\{e\}}\diff^{Q,w(\underline{\mu}_{P_w}^{\cusp}+\tilde\lambda)}_{\underline{a}_w}
(\cnst_{G,Q}\psi-\cnst_{P,Q}\varphi_\lambda)\in(\tilde\AF_Q^{\cusp})^\perp
\end{equation}
for any $Q\in\stds$ and any collection $\underline{a}_w\in\AAA_Q^{m_{P_w}}$, $w\in\Weyl^{\supset Q}(P)\setminus\{e\}$.
\end{proposition}

\begin{proof}
The equations \eqref{eq: nonhom2} are satisfied for $\psi=E(\varphi,\lambda)$ by \eqref{eq: cuspe} .

Conversely, suppose that $\psi$ satisfies \eqref{eq: nonhom2} for some $\lambda\in\chars_P$.
Then, for every $Q\in\stds$
\begin{equation} \label{eq: allw2}
\prod_{w\in\Weyl^{\supset Q}(P)}
\diff^{Q,w(\underline{\mu}_{P_w}^{\cusp}+\tilde\lambda)}_{\underline{a}_w}(\cnst_{G,Q}\psi)\in(\tilde\AF_Q^{\cusp})^\perp
\end{equation}
for any collection $\underline{a}_w\in\AAA_Q^{m_{P_w}}$, $w\in\Weyl^{\supset Q}(P)$.
Thus, by Lemma \ref{lem: autcusp2}, $\psi$ is an automorphic form.
Hence, the relation \eqref{eq: nonhom2} is now equivalent to \eqref{eq: cuspe}.
It follows from \eqref{eq: Eunique} that $\psi=E(\varphi,\lambda)$ provided that $\sprod{\Re\lambda}{\alpha^\vee}\gg0$
for all $\alpha\in\srts_P$.
\end{proof}

\subsection{Algebraic version of the principle of meromorphic continuation} \label{sec: AVPMC}
Recall that the principle of meromorphic continuation admits an easier algebraic analogue
which has already been used many times in the literature (see \cite{MR892097}*{p. 127} or \cite{MR1671189}*{\S1}).

We first introduce some terminology.
Let $V$ be a vector space over $\C$, $V^*$ its dual space,
$D$ an affine variety over $\C$ and $\C[D]$ its ring of regular functions.
Let $V[D]=V\otimes\C[D]$. For any $\lambda\in D$, the evaluation at $\lambda$ homomorphism $\C[D]\rightarrow\C$
gives rise to a linear map $V[D]\rightarrow V$  which we denote by $\mu\mapsto\mu(\lambda)$.

A regular family $\Xi$ of linear systems of equations on $V^*$ is a family
of elements $\mu_i\in V[D]$, $\nu_i\in\C[D]$, $i\in I$.
For each $\lambda\in D$ it gives rise to a linear system of equations $\Xi(\lambda)$ on $v^*\in V^*$ given by
\[
\sprod{v^*}{\mu_i(\lambda)}=\nu_i(\lambda),\ \ i\in I.
\]

\begin{theorem} \label{thm: algcont}
In the above setup, suppose that $V$ has countable dimension and $D$ is irreducible.
Let $\C(D)$ be the field of fractions of $\C[D]$.
Suppose that we are given a regular family of linear systems $\Xi$ of equations on $V^*$ as above.
Assume that there exists a non-empty, open (in the Hausdorff topology) subset
$D'$ of $D$ such that for all $\lambda\in D'$ the system $\Xi(\lambda)$ has a unique solution.
Then, there exists a unique element
\[
A\in\Hom_{\C[D]}(V[D],\C(D))=\Hom_{\C}(V,\C(D))
\]
such that $A(\mu_i)=\nu_i$ for all $i\in I$. Moreover, for all $\lambda\in D$ outside the union of countably many hypersurfaces,
$Av\in\C(D)$ is regular in $\lambda$ for all $v\in V$ and $(v\mapsto Av(\lambda))\in V^*$ is the unique
solution of $\Xi(\lambda)$.
\end{theorem}

See the references above for the easy proof.

\subsection{Rationality of Eisenstein series} \label{sec: FFEis}
Let $\varphi\in\AF_P$.
Using Theorem \ref{thm: algcont} we prove that the Eisenstein series $E(\varphi,\lambda)$, originally defined
in the region $\sprod{\Re\lambda}{\alpha^\vee}\gg0$ $\forall\alpha\in\srts_P$ is a rational functions on $\lambda\in \chars_P$.

Fix a compact open subgroup $K$ of $G(\A)$ such that $\varphi$ is right $K$-invariant.
Consider the space $V=\rapid(\autspace)^K$ and its dual space $V^*$ of all right $K$-invariant functions on $\autspace$.
The system $\Xi(\lambda)$, $\lambda\in \chars_P$ consisting of the linear equations \eqref{eq: nonhom2}
is a regular family of linear systems on $V^*$ (since $\tilde\AF_Q^{\cusp}\subset\rapid(\autspace_Q)$).
By Proposition \ref{prop: Ximain2} and Theorem \ref{thm: algcont} we deduce that for any $g\in G(\A)$ the function
$E(g,\varphi,\lambda)$ is a rational function on $\chars_P$.
(Note that we do not need to use the equations \eqref{eq: HCeis} or \S\ref{sec: LF},
the local finiteness part.)

We claim that in fact there exists a polynomial $p$ on $\chars_P$ such that
$p(\lambda)E(g,\varphi,\lambda)$ is a polynomial on $\chars_P$ for all $g\in G(\A)$.

Indeed, by \eqref{eq: allw2}, we have
\[
(E(\varphi,\lambda),\theta_Q\phi)_{\autspace}=0
\]
for any $Q\in\stds$,
\[
\phi\in\prod_{w\in\Weyl^{\supset Q}(P)}\diff^{Q,w(\underline{\mu}_{P_w}^{\cusp}+\tilde\lambda)}_{\underline{a}_w}(\tilde\AF_Q^{\cusp})
\]
and any collection $\underline{a}_w\in\AAA_Q^{m_{P_w}}$, $w\in\Weyl^{\supset Q}(P)$.
Hence, by Corollary \ref{cor: commonU}, there exists a finite-dimensional subspace $U$ of $V$
such that if $p$ is a polynomial on $\chars_P$ such that
$p(\lambda)(E(\varphi,\lambda),\phi)_{\autspace}$ is a polynomial on $\chars_P$
for all $\phi\in U$, then $p(\lambda)(E(\varphi,\lambda),\phi)_{\autspace}$
is a polynomial on $\chars_P$ for all $\phi\in V$.

\subsection{Rationality of intertwining operators}
The proof of the rationality of $\lambda\in\chars_P\mapsto M(w,\lambda)\varphi$ for any $\varphi\in\AF_P$, $w\in\Weyl(P,Q)$
is identical to the number field case, taking into account
that the operators $D(\lambda)$ defined in \S\ref{sec: MIO} are polynomial functions in $\lambda\in\chars_P$
(which factor through $\charsA_P$).

Similarly, the functional equation $E(M(w,\lambda)\varphi,w\lambda)=E(\varphi,\lambda)$ for any $w\in\Weyl(P,P')$
and $M(w'w,\lambda)=M(w',w\lambda)\circ M(w,\lambda)$ for any $w\in\Weyl(P,P')$ and $w'\in\Weyl(P',P'')$
are proved exactly as in the number field case.

Finally, we consider the singularities of $M(w,\lambda)$.
For any $\alpha\in\Phi_P$ let $\alpha^*\in M(\A)/M(\A)^1$ be the element defined in \cite{MR1361168}*{I.1.11}.
By definition, a root hyperplane of $\chars_P$ is a hypersurface given by the equation $\alpha^{*\lambda}=c$
for some $\alpha\in\Phi_P$ and $c\in\C^*$.
We say that a rational function on $\chars_P$ has singularities along root hyperplanes if
it is regular outside a finite union of root hyperplanes.
We show that for any $\varphi\in\AF_P$, the intertwining operator $M(w,\lambda)\varphi$ has singularities along root hyperplanes
(and in fact a more precise statement).
If $w$ is an elementary symmetry $s_\alpha$ for some $\alpha\in\srts_P$,
then as a function of $\lambda\in\chars_P$, $M(w,\lambda)\varphi$ depends only on $\alpha^{*\lambda}$.
In general, by decomposing $w$ into elementary symmetries \cite{MR1361168}*{I.1.8} and using the multiplicativity of
intertwining operators, it follows that
the singularities of $M(w,\lambda)$ are of the form $\beta^{*\lambda}=c$ for some
$\beta\in\Phi_P$ such that $w\beta<0$ and $c\in\C^*$.

Finally, as in the number field case, it follows from \cite{MR1361168}*{I.4.10} and \eqref{eq: consterm} that the singularities of
$E(\varphi,\lambda)$ are also along root hyperplanes.

This finishes the proof of Theorem \ref{thm: main}.

\appendix

\section{Proof of Principle of Meromorphic Continuation (Theorem \ref{thm: anal})}
\begin{comment}
We may write $\Xi(s)$ equivalently as
\[
\sprod{\nu}{[\mu_i(s)]v}=\sprod{\nu}{c_i(s)},\ i\in I,\ \nu\in \spc_i'.
\]
Thus, we may assume without loss of generality that $\spc_i=\C$ for all $i\in I$.
Hence, $\mu_i(s)$, $s\in\mnfld$ is an analytic family of continuous linear functionals on $\spc$
(with respect to the weak topology on $\spc'$) and $c_i$ is a scalar-valued analytic function on $\mnfld$.
\end{comment}

We will show that for every $s_0\in\overline{\mnfld_{\unq}^\circ}$ there exists an open, connected neighborhood $W$ of $s_0$ in $\mnfld$
with the following property. For every $s_1\in\mnfld_{\unq}\cap W$ there exists a scalar-valued holomorphic function $f$ on $W$ such that
\begin{enumerate}
\item $f(s_1)\ne0$.
\item $\mnfld_{\unq}\supset W_f:=\{s\in W\mid f(s)\ne0\}$.
\item There exists a holomorphic function $u:W\rightarrow \spc$ such that $u(s)=f(s)v(s)$ for all $s\in W_f$.
%\item $f(s)v(s)$ can be extended to a holomorphic function on $W$.
\end{enumerate}
This implies that $\overline{\mnfld_{\unq}^\circ}$ is open and hence, since $\mnfld$ is connected
and $\mnfld_{\unq}^\circ$ is nonempty by assumption, $\overline{\mnfld_{\unq}^\circ}=\mnfld$.
Thus, the theorem would follow from the statement above.

Since the statement is local in $s_0$ we may assume, by passing to a neighborhood of $s_0$,
that $\mnfld$ is connected, $\mnfld_{\unq}^\circ\ne\emptyset$, and there exist a finite-dimensional vector space $L$
and an analytic family of injective operators $\lambda_s:L\rightarrow\spc$, $s\in\mnfld$
such that $\Sol(\Xi(s))\subset \Img\lambda_s$ for all $s\in\mnfld$.
%We may assume of course that $W$ is connected.
%In this case, $W_f$ is connected for any non-zero holomorphic function $f$ on $W$.

By considering the pullback of $\Xi$ under $\lambda_s$ we may assume without loss of generality that $\spc=L$ is finite-dimensional.
We may think of $\Xi$ as a (possibly infinite) system of linear equations in $n$ variables (where $n=\dim L$) whose coefficients
depend analytically on $s$.

%By assumption, $W\cap \mnfld_{\unq}^\circ\ne\emptyset$.
Let $s_1\in\mnfld_{\unq}$, so that $\Xi(s_1)$ admits a unique solution $v(s_1)$.
Then, we can extract from $\Xi$ a nonsingular subsystem $\hat\Xi$
consisting of $n$ equations such that $\hat\Xi(s_1)$ admits $v(s_1)$ as its unique solution.
Let $f(s)$ be the determinant of the coefficients of the system $\hat\Xi(s)$, $s\in \mnfld$.
Then, $s_1\in \mnfld_f=\{s\in\mnfld\mid f(s)\ne0\}$, and if $s\in \mnfld_f$, then $\Sol(\hat\Xi(s))$ is a singleton which we write as $\{\hat v(s)\}$.
A fortiori, $\Sol(\Xi(s))\subset\{\hat v(s)\}$ for all $s\in \mnfld_f$.
Thus, $v(s)=\hat v(s)$ for all $s\in \mnfld_f\cap \mnfld_{\unq}$.
Moreover, by Cramer's rule $f(s)\hat v(s)$ extends to a holomorphic function on $\mnfld$,
and in particular $\hat v(s)$ is holomorphic on $\mnfld_f$.
Observe that $\mnfld_f\cap \mnfld_{\unq}^\circ$ is a nonempty open set since $\mnfld_f$ is dense in $\mnfld$ and $\mnfld_{\unq}^\circ\ne\emptyset$.
Since $\hat v(s)$ solves $\Xi(s)$ on $\mnfld_f\cap \mnfld_{\unq}$ and $\mnfld_f$
is connected (as a complement of a hypersurface) we infer by analytic continuation that $\hat v(s)\in\Sol(\Xi(s))$
for all $s\in \mnfld_f$. Thus, $\Sol(\Xi(s))=\{\hat v(s)\}$ for all $s\in \mnfld_f$.
It follows that $\mnfld_{\unq}\supset \mnfld_f$ and $v(s)=\hat v(s)$ for any $s\in \mnfld_f$.
Our claim follows.
$\qed$

\begin{remark}
Let $\spc$ be a \lctvs.
We say that a family $A_s$, $s\in\mnfld$ of subsets of $\spc$ is \emph{weakly} of finite type
if there exist a finite-dimensional vector space
$L$ and an analytic family $\lambda_s$, $s\in\mnfld$ of operators $L\rightarrow\spc$ such that
$A_s\subset\Img\lambda_s$ for all $s\in\mnfld$. (We do not require that $\lambda_s$ are injective.)
We can similarly define the corresponding local notion.
Suppose that $\Xi(s)$ is an analytic system of linear equations on $\spc$ which is locally, weakly of finite type
in the sense that the set of solutions of $\Xi(s)$ is locally, weakly of finite type.

Then, a similar argument to the above shows that for every $s_0\in\overline{\mnfld_{\unq}^\circ}$ there exists an open, connected neighborhood $W$ of $s_0$ in $\mnfld$
and a non-zero holomorphic function $f_1$ on $W$ with the following property.
For every $s_1\in\mnfld_{\unq}\cap W_{f_1}$ there exists a holomorphic function $f$ on $W$ such that
\begin{enumerate}
\item $f(s_1)\ne0$.
\item $\mnfld_{\unq}\supset W_f$.
\item There exists a holomorphic function $u:W\rightarrow \spc$ such that $u(s)=f(s)v(s)$ for all $s\in W_f$.
%\item $f(s)v(s)$ can be extended to a holomorphic function on $W$.
\end{enumerate}

This implies that $\mnfld_{\unq}$ contains an open dense subset $U$ of $\mnfld$ such that $v$ is holomorphic on $U$
and meromorphic on $\mnfld$. However, we do not know whether in this generality $\mnfld_{\unq}$ is open
and $v$ is holomorphic on $\mnfld_{\unq}$.

In fact, in a previous version of the paper we worked with this weaker notion of finite type.
However, we realized that it is better to work with the stronger notion.
\end{remark}

\printindex

%\bibliographystyle{amsalpha}
%\bibliography{../Bibfiles/all}
%\end{document}

\def\cprime{$'$} 
% \bib, bibdiv, biblist are defined by the amsrefs package.

\end{document}